\documentclass[12pt]{extarticle}
\usepackage{amsmath}
\usepackage{amsthm}
\usepackage{amssymb}
\usepackage{amsfonts}

\usepackage[utf8]{inputenc}
\usepackage[margin=1in]{geometry}

\usepackage{comment}
\usepackage{soul}

\usepackage{mathtools}
\usepackage{thmtools}
\usepackage{array}

\usepackage{hyperref}
\usepackage{url}
\usepackage{xurl}

\usepackage{tikz}
\usetikzlibrary{cd}
\usetikzlibrary{math}
\usetikzlibrary{
	arrows.meta,
	decorations.markings
	}
\usepackage{xifthen}
\usepackage{etoolbox}

\usepackage{caption}

\usepackage{stackrel}

\usepackage{authblk}

\usepackage{basic-commands}

\title{Vietoris-Rips Complexes of Split-Decomposable Spaces}

\author{Mario G\'{o}mez}
\affil{Department of Mathematics, 
	Florida State University\\
	\texttt{mrg23b@fsu.edu}}

\DeclareMathOperator{\Cl}{Cl}
\DeclareMathOperator{\wf}{wf}

\DeclareMathOperator{\cut}{\operatorname{cut}}
\DeclareMathOperator{\blocks}{\cB}
\DeclareMathOperator{\bctree}{BC}

\usepackage{stix}
\newcommand{\osum}{\modtwosum}

\makeatletter
\def\keywords{\xdef\@thefnmark{}\@footnotetext}
\makeatother

\begin{document}
	\maketitle
	
	\begin{abstract}
	Split-metric decompositions are an important tool in the theory of phylogenetics, particularly because of the link between the tight span and the class of totally decomposable spaces, a generalization of metric trees whose decomposition does not have a ``prime'' component. Their close relationship with trees makes totally decomposable spaces attractive in the search for spaces whose persistent homology can be computed efficiently. We study the subclass of circular decomposable spaces, finite metrics that resemble subsets of $\Sp^1$ and can be recognized in quadratic time. We give an $O(n^2)$ characterization of the circular decomposable spaces whose Vietoris-Rips complexes are cyclic for all distance parameters, and compute their homotopy type using well-known results on $\Sp^1$. We extend this result to a recursive formula that computes the homology of certain circular decomposable spaces that fail the previous characterization. Going beyond totally decomposable spaces, we identify an $O(n^3)$ decomposition of $\vr_r(X)$ in terms of the blocks of the tight span of $X$, and use it to induce a direct-sum decomposition of the homology of $\vr_r(X)$.
\end{abstract} 	
	
	\tableofcontents
	\section{Introduction}
\indent A phylogenetic network is a graph that models the evolutionary history of a set of species. Although we usually think of evolutionary history as a tree, reticulate events (such as when a cell infected with two strands of the same virus produces a hybrid of the parents) can create cycles in this network. A \define{split-metric decomposition} is one way to produce a phylogenetic network that models a given set of species, using only the distances between them. Formally, a split-metric decomposition is a canonical expression of a finite metric as a linear combination of certain pseudo-metrics called split-metrics \cite{metric-decomposition}. In general, the decomposition of a finite metric may leave a ``split-prime'' residue, and a space where the residue is 0 is called \define{split decomposable}. Tree metrics, for example, are split decomposable.\\
\indent There are several reasons why split decomposable spaces, as a theoretical model for genetic data, are attractive from the point of view of Topological Data Analysis (TDA). From a practical standpoint, \cite{viral-evolution} used persistent homology to detect and quantify reticulate events in viral evolution, while \cite{tda-covid-recombination} developed a topological pipeline to detect variants of interest among SARS-CoV-2 strands. Computationally, \cite{vr-hyp-geodesic-defect} noted that computing the persistent homology of a tree metric with millions of points may last up to a week, even though the result should have been 0 \cite[Appendix Theorem 2.1]{viral-evolution}. Their key insight was that a particular reordering of the points reduced the computation time to a few minutes \cite[Remark 6]{vr-hyp-geodesic-defect}. In other words, finding the persistent homology of a tree is an extremely fast computation under the right setup.\\
\indent From a theoretical point of view, there are key results in TDA and the theory of split decompositions that are ripe for interaction. The Vietoris-Rips (VR) complex of a subset of the circle was characterized by \cite{aa17}. Any finite subset of $\Sp^1$ is split decomposable by \cite[Theorem 5]{metric-decomposition}, and any space whose split decomposition is equivalent to that of a circular set is called \define{circular decomposable}\footnote{This terminology is used in e.g. \cite{CDM-structure, CDM-recognition, CDM-note}.}. A natural question is to what extent do the results of \cite{aa17} apply to circular decomposable spaces. Furthermore, \cite[Proposition 2.3]{osman-memoli} proved that the VR complex of $X$ is homotopy equivalent to the metric thickening of $X$ inside of its tight span\footnote{Tight spans are also called \define{injective envelopes}.} (a metric space with many desirable properties that contains $X$). This result has increased the range of distances on which the VR complexes of higher dimensional spheres are known (see e.g. Theorem 10 and Corollary 7.2 of \cite{osman-memoli}). On the other hand, the tight span of a split decomposable space has been studied at least since \cite{metric-decomposition} and, recently, the polytopal structure of such a tight span was completely characterized by Huber, Koolen, and Moulton in \cite{TDM-polytopal-structure-2}.\\
\indent Regarding our work, in Section \ref{sec:circular-metrics}, we use the results of \cite{aa17} to characterize the homotopy type of circular decomposable spaces whose metric is also monotone in the sense of \cite{CDM-recognition} (see Corollary \ref{cor:VR_circular_monotone}). We also give conditions on the isolation indices of a circular decomposable space that guarantee the metric to be monotone (Theorem \ref{thm:monotone_equals_star}). However, not every circular decomposable metric is also monotone. In Section \ref{sec:circular-metrics-non-monotone}, we compute the homology of circular decomposable spaces $X$ in terms of a strict subset $Y \subset X$ (Theorems \ref{thm:MV_non_critical} and \ref{thm:MV_critical}). The results on the homotopy type of clique complexes of cyclic graphs \cite{aa17} are instrumental in this section. Lastly, in Section \ref{sec:blocks_of_X}, we study a decomposition of the VR complex of $X$ induced by the block decomposition (or more generally, a \define{block cover}) of an injective space $E$ that contains an isometric copy of $X$. Our main result in this section is that the VR complex of $X$ decomposes as a disjoint union or wedge of the VR complexes of certain subsets $\overline{X}_K \subset X$, each corresponding to a connected union $K$ of blocks of $E$ (Theorem \ref{thm:VR_wedge}). A direct consequence is that the homology of $X$ is a direct sum of the homology groups of all $\overline{X}_K$, with only a small modification required for $H_0$ (Theorem \ref{thm:VR_wedge_homology}).\\
\indent We believe these results can help speed up the computation of persistent homology, especially in phylogenetic settings where TDA has already made important contributions \cite{viral-evolution, vr-hyp-geodesic-defect, tda-covid-recombination}. Circular decomposable spaces in particular are very efficient to work with. For example, the path that follows the circular ordering in such a space is an exact solution for the Traveling Salesman Problem \cite{CDM-structure, CDM-note}. Crucially, identifying that an $n$ point metric space is circular decomposable and finding the decomposition can be done in $O(n^2)$ time \cite{CDM-structure, CDM-recognition}. Our Corollary \ref{cor:VR_circular_monotone} depends on a function $\sigma:X \to X$ that can also be computed in quadratic time. Hence, identifying a circular decomposable metric that satisfies the hypothesis of Corollary \ref{cor:VR_circular_monotone} and computing its persistent homology can be done in $O(n^2(k + \log(n)))$ time thanks to \cite{ph-circle}.\\
\indent The immediate benefit of a block-induced decomposition of $\vr_r(X)$ is breaking up the computation of persistent homology into smaller spaces. We make use of a cubic-time algorithm from \cite{cut-points-algorithm} that computes the block decomposition of the tight span of any $X$ and implement Corollary \ref{cor:VR_wedge} in Python. Our experiments show that the block decomposition is faster than the Python implementation of persistent homology \cite{ripser-python} starting from homological dimension 3. However, if we have a priori knowledge of the block decomposition, we may avoid using the cubic-time algorithm to obtain even larger speedups. See Corollary \ref{cor:VR_wedge_subset} and the discussion in Section \ref{sec:discussion}. These results are explained by the fact that a tight span that has multiple blocks intuitively resembles a tree, except that nodes are replaced with blocks,  so finding the persistent homology of $X$ reduces to computing on each ``node'' and pasting the results.\\
\indent All in all, split decomposable spaces are a more general model for phylogenetic evolution that allows for the existence of cycles without deviating too much from a tree. We hope our results improve the computational speed of persistent homology beyond metric trees.

\paragraph{Related work.} Proposition 3.7 of \cite{vr_metric_gluings} says that $\vr_r(X \vee Y) \simeq \vr_r(X) \vee \vr_r(Y)$. Corollary \ref{cor:VR_wedge_subset} generalizes this proposition in that $X$ and $Y$ don't have to be glued themselves, they only have to be a subset of a metric gluing $Z_1 \vee Z_2$. Taking this idea further, in Theorem \ref{thm:VR_wedge} we prove that an injective space $E$ that decomposes as a wedge sum of other injective spaces induces a decomposition of $\vr_r(X)$ for any finite $X \hookrightarrow E$. In other words, instead of gluing VR complexes, we use the decomposition of the ambient injective space $E$ to realize $\vr_r(X)$ as a gluing of VR complexes of subsets of $X$ that is more complicated than a wedge union.

\paragraph{Acknowledgments.}
The author thanks Facundo M\'{e}moli for introducing him to the paper \cite{metric-decomposition} for useful discussions and suggestions during the writing of this paper, along with Tom Needham and Henry Adams for further discussions. The author also acknowledges funding from BSF Grant 2020124 received during the time he was finishing his PhD. 	\section{Preliminaries}
In this section, we set the notation and collect the most important theorems and definitions from our sources.

\subsection{Notation}
Let $(X,d_X)$ be a metric space. The diameter of $\sigma \subset X$ is $\diam(\sigma) = \sup_{x,y \in X} d_X(x,y)$, and its radius is $\rad(\sigma) = \inf_{x \in X} \sup_{p \in \sigma} d_X(x,p)$. The (open) Vietoris-Rips complex (or VR complex for short) is the simplicial complex
\begin{equation*}
	\vr_r(X) := \{ \sigma \text{ finite} : \diam(\sigma) < r \}.
\end{equation*}

We model the circle $\Sp^1$ as the quotient $\R/\Z$ and equip it with the geodesic distance $d_{\Sp^1}$ scaled so that $\Sp^1$ has circumference 1. We say that a path $\gamma:[0,1] \to \Sp^1$ is clockwise if any lift $\widetilde{\gamma}:[0,1] \to \R$ is increasing -- this defines the clockwise direction on $\Sp^1$.

\begin{defn}[Cyclic order]
	\label{def:circular-order}
	Given three points $a,b,c \in \Sp^1$, we define the \define{cyclic order} by setting $a \prec b \prec c$ if $a,b,c$ are pairwise distinct and the clockwise path from $a$ to $c$ contains $b$. We write $a \preceq b \prec c$ or $a \prec b \preceq c$ to allow $a=b$ or $b=c$, respectively.
\end{defn}

\begin{defn}[Circular sum]
	\label{def:circular-sum}
	Fix $n > 0$ and consider a function $f:\{1, \dots, n\} \to \R$. Given $a, b \in \{1, \dots, n\}$, we define the \define{circular sum} as
	\begin{equation*}
		\osum_{i=a}^b f(i) := \sum_{i=a} ^b f(i)
	\end{equation*}
	when $a \leq b$ and
	\begin{equation*}
		\osum_{i=a}^b f(i) := \sum_{i=a} ^n f(i) + \sum_{i=1}^b f(i)
	\end{equation*}
	when $a > b$. The value of $n$ will be clear from context.
\end{defn}

\begin{remark}
	The expression $\sum_{i=a}^b f(i)$ when $a > b$ usually represents an empty sum and its value is 0 by convention, while the notation $\osum_{i=a}^b f(i)$ is usually not 0. However, the degenerate cases in some results (like those in Section \ref{subsec:circular-metrics-monotone}) could involve empty circular sums. Since we are deviating from the above convention, we never write empty circular sums, and always state degenerate cases separately if they would require an empty circular sum.
\end{remark}

\indent Two equivalent formulations of the circular sum are:
\begin{equation*}
	\osum_{i=a}^b f(i) = \sum_{i=0}^{b-a} f( a+i \hspace{-0.5em} \mod(n) ) = \sum_{a \preceq i \preceq b} f(i)
\end{equation*}

\noindent In fact, we think of $\osum_{i=a}^{b} f(i)$ as a summation of $f$ over the vertices of a regular $n$-gon inscribed on the circle. Indeed, if we label the vertices clockwise from $1$ to $n$, the sum $\osum_{i=a}^b f(i)$ adds the value of $f$ at all the vertices in the clockwise path between the $a$-th and the $b$-th vertices (inclusive). If $a > b$, the clockwise path starting at the $a$-th vertex has to go through the $n$-th vertex before getting to the $b$-th vertex.

\subsection{Background}
\paragraph{Split systems.}
Our main object of study is the split-metric decomposition of a finite metric. Given a finite set $X$, a \define{split} is a partition $S = \{A, B\}$ of $X$. We denote splits by $A | B$, $B | A$, or $A | \overline{A}$ where $\overline{A}$ is the complement of $A$ in $X$. Given $x \in X$, we denote by $S(x)$ the element of $S$ that contains $x$. A collection $\cS = \{S_1, \dots, S_m\}$ of splits is called a \define{split system}. A weighted split system is a pair $(\cS, \alpha)$ where $\alpha:\cS \to \R_{>0}$. We write the value of $\alpha$ at $S \in \cS$ as $\alpha_S$.

\paragraph{Weakly compatible split systems.}
A split system is \define{weakly compatible} if there are no three splits $S_1, S_2, S_3$ and four elements $x_0, x_1, x_2, x_3 \in X$ such that
\begin{equation*}
	S_j(x_i) = S_j(x_0) \text{ if and only if } i=j.
\end{equation*}
Any finite metric space $(X, d_X)$ has an associated weakly compatible split system (which could be empty) with weights induced by the metric. For any $a_1, a_2, b_1, b_2 \in X$ (not necessarily distinct), define
\begin{equation*}
	\beta_{\{a_1,a_2\}, \{b_1, b_2\}} := \max \left\{
	\begin{array}{c}
		d_X(a_1, b_1) + d_X(a_2, b_2)\\
		d_X(a_1, b_2) + d_X(a_2, b_1)\\
		d_X(a_1, a_2) + d_X(b_1, b_2)
	\end{array}
	\right\}
	-\big[ d_X(a_1, a_2) + d_X(b_1, b_2) \big].
\end{equation*}
If $A|B$ is a split of $X$, the \define{isolation index} of $A|B$ is the number
\begin{equation*}
	\alpha_{A|B} := \min_{\substack{a_1, a_2 \in A\\b_1, b_2 \in B}} \beta_{\{a_1,a_2\}, \{b_1, b_2\}}.
\end{equation*}
Note that $\beta_{\{a_1,a_2\}, \{b_1, b_2\}} \geq 0$ and, as a consequence, $\alpha_{A|B} \geq 0$. If $\alpha_{A|B} \neq 0$, the split $A|B$ is called a \define{$d_X$-split}. The set
\begin{equation*}
	\cS(X,d_X) := \{ A|B \text{ split of } X  \text{ such that } \alpha_{A|B} \neq 0 \}
\end{equation*}
is called the \define{system of $d_X$-splits} of $(X, d_X)$ and it is weighted by the isolation indices $\alpha_{A|B}$. Bandelt and Dress proved that $\cS(X,d_X)$ is always weakly compatible \cite[Theorem 3]{metric-decomposition}. However, not every metric space $(X,d_X)$ has $d_X$-splits. We call such a space \define{split-prime} \cite{metric-decomposition}.

\begin{example}
	\label{ex:prime_metrics}
	The complete bipartite graph $K_{2,3}$ with the shortest path metric $d$ has no $d$-splits. It is not hard to prove (see also the introduction of Section 2 of \cite{metric-decomposition}) that any $d$-split $A|\overline{A}$ must be $d$-convex in the sense that $x,y \in A$ and $d(x,y) = d(x,z) + d(z,y)$ implies $z \in A$ (likewise for $\overline{A}$). It can be checked that $K_{2,3}$ cannot be split into two $d$-convex sets, so it is prime. The hypercube graph $H_n$ with the shortest path metric is also split-prime for $n \geq 3$ (see the description after Proposition 3 in \cite{metric-decomposition}).
\end{example}
\begin{figure}[ht]
	\centering
\hfill
\begin{minipage}{0.40\linewidth}
	\centering
	
	\begin{tikzpicture}
		\def \radius {2.5pt};
		
\coordinate (P1) at (0,0);
		\coordinate (P2) at (0,-1);
		\coordinate (P3) at (0,1);
		
		\coordinate (Q1) at (-1.5,0);
		\coordinate (Q2) at (1.5,0);
		
\fill (P1) circle (\radius);
		\fill (P2) circle (\radius);
		\fill (P3) circle (\radius);
		\fill (Q1) circle (\radius);
		\fill (Q2) circle (\radius);
		
\draw (Q1) -- (P1);
		\draw (Q1) -- (P2);
		\draw (Q1) -- (P3);
		
		\draw (Q2) -- (P1);
		\draw (Q2) -- (P2);
		\draw (Q2) -- (P3);
	\end{tikzpicture}
\end{minipage}
\begin{minipage}{0.40\linewidth}
	\centering
	
	\begin{tikzpicture}[scale = 1,
		z={(-.3cm,-.2cm)}, line join=round, line cap=round ]
		
		\def \radius {2.5pt};
		\def \a {2};
		
\draw ( 0, 0, 0) -- (\a, 0, 0) -- (\a,\a, 0) -- ( 0,\a, 0) -- ( 0, 0, 0);
		\draw ( 0, 0,\a) -- (\a, 0,\a) -- (\a,\a,\a) -- ( 0,\a,\a) -- ( 0, 0,\a);
		\draw ( 0, 0, 0) -- ( 0, 0,\a);
		\draw (\a, 0, 0) -- (\a, 0,\a);
		\draw (\a,\a, 0) -- (\a,\a,\a);
		\draw ( 0,\a, 0) -- ( 0,\a,\a);
		
\fill ( 0, 0, 0) circle (\radius);
		\draw[fill=white] ( 0, 0,\a) circle(\radius);
		\draw[fill=white] ( 0,\a, 0) circle(\radius);
		\fill ( 0,\a,\a) circle(\radius);
		\draw[fill=white] (\a, 0, 0) circle(\radius);
		\fill (\a, 0,\a) circle(\radius);
		\fill (\a,\a, 0) circle(\radius);
		\draw[fill=white] (\a,\a,\a) circle(\radius);
	\end{tikzpicture} 
\end{minipage}
\hfill 	\caption{\textbf{Left:} The complete bipartite graph $K_{2,3}$. There are no splits of $K_{2,3}$ into $d$-convex sets, so $K_{2,3}$ has no $d$-splits. \textbf{Right:} The hypercube graph $H_3$. The distance between any pair of white vertices is 2, so their $\beta_{\{a_1,a_2\}, \{b_1, b_2\}}$ coefficient is 0. This prevents $H_3$ from having any $d$-split.}
	\label{fig:prime_metrics}
\end{figure}

\paragraph{Split-metric decompositions.}
Given a finite set $X$ and a split $A|B$ of $X$, a \define{split-metric} is the pseudometric defined by
\begin{equation*}
	\delta_{A|B}(x,y) :=
	\begin{cases}
		0 & x,y \in A \text{ or } x,y \in B\\
		1 & \text{otherwise}.
	\end{cases}
\end{equation*}
\indent One of the main results of Bandelt and Dress is the decomposition of any finite metric as a linear combination of split-metrics and a split-prime residue. 
\begin{theorem}[{\cite[Theorem 2]{metric-decomposition}}]
	\label{thm:split-decomposition}
	Any metric $d_X:X \times X \to \R$ on a finite set $X$ decomposes as
	\begin{equation}
		\label{eq:split-decomposition}
		d_X = d_0 + \sum_{A|B \in \cS(X,d_X)} \alpha_{A|B} \cdot \delta_{A|B}
	\end{equation}
	where $d_0$ is split-prime and the sum runs over all $d_X$-splits $A|B$. Moreover, the decomposition is unique in the following sense. Let $\cS'$ be a weakly compatible split system on $X$ with weights $\lambda_S>0$ for $S \in \cS'$. If $d_X = d_0 + \sum_{S \in \cS'} \lambda_S \cdot \delta_S$, then there is a bijection $f:\cS' \to \cS(X,d_X)$ such that $\alpha_{f(S)} = \lambda_{S}$.
\end{theorem}

\noindent Equation (\ref{eq:split-decomposition}) is called the \define{split-metric decomposition} of $(X,d_X)$, or \define{split decomposition} for brevity.

\begin{remark}
	\label{rmk:totally-decomposable-metric}
	Theorem 2 is stated in more generality in \cite{metric-decomposition}. The version therein holds for any symmetric function $d_X:X \times X \to \R$ with 0 diagonal if we allow negative coefficients in (\ref{eq:split-decomposition}). We won't need the full result, but note that if $d_0$ is a (pseudo)metric, then $d_X$ is also a (pseudo)metric.
\end{remark}

\begin{defn}
	\label{def:totally-decomposable-sum}
	Given a weighted weakly compatible split system $(\cS, \alpha)$ on $X$, we define the pseudometric
	\begin{equation*}
		d_{\cS,\alpha} := \sum_{A|B \in \cS(X,d_X)} \alpha_{A|B} \cdot \delta_{A|B}
	\end{equation*}
\end{defn}

\begin{defn}[Totally decomposable spaces]
	\label{def:totally-decomposable-metric}
	A metric space $(X,d_X)$ is called \define{totally decomposable} if $d_0=0$ in Equation (\ref{eq:split-decomposition}) or, equivalently, if $d_X = d_{\cS, \alpha}$ where $\cS = \cS(X,d_X)$ and $\alpha_S$ is the isolation index of $S \in \cS$.
\end{defn}

\paragraph{Circular collections of splits.}
Suppose $X$ is a set with $n$ points. Then a weakly compatible split system $\cS$ on $X$ can have at most $\binom{n}{2}$ splits. The reason is that the vector space of symmetric functions $f:X \times X \to \R$ with $0$ diagonal has dimension $\binom{n}{2}$ and the uniqueness in Theorem \ref{thm:split-decomposition} implies that the set of split-metrics $\{ \delta_S : S \in \cS \}$ is linearly independent \cite[Corollary 4]{metric-decomposition}.\\
\indent There are examples of split systems that achieve this bound. Let $V_n$ be the vertices of a regular $n$-gon. Any line that passes through two distinct edges $(v_{i-1}, v_{i})$ and $(v_j, v_{j+1})$ separates $V_n$ into the sets $\{v_i, \dots, v_j\}$ and $\{v_{j+1}, \dots, v_{i-1}\}$. Since there are $n$ edges, this construction gives a split system $\cS(V_n)$ with $\binom{n}{2}$ different splits, and it can be verified that $\cS(V_n)$ is weakly compatible. It turns out that, up to a bijection of the underlying sets, this is the only example.

\begin{theorem}[{\cite[Theorem 5]{metric-decomposition}}]
	\label{thm:circular-splits}
	The following conditions are equivalent for a weakly compatible split system $\cS$ on a set $X$ with $|X|=n$:
	\begin{enumerate}
		\item $|\cS| = \binom{n}{2}$ ;
		\item There exists a bijection $f:X \to V_n$ such that $f(\cS) = \cS(V_n)$;
		\item The set $\{\delta_S : S \in \cS\}$ is a basis of the vector space of all symmetric functions $d_X:X \times X \to \R$ with 0 diagonal.
	\end{enumerate}
\end{theorem}
\noindent Note that such an $X$ inherits the cyclic order from $V_n$. Explicitly, if $f(x_i) = v_i$, then $x_i \prec x_j \prec x_k$ if and only if $v_i \prec v_j \prec v_k$.\\

\indent We later study not only $\cS(V_n)$ but its subsets as well.
\begin{defn}
	\label{def:circular-splits}
	Let $\cS$ be a split system on a set $X$. If there exists a bijection $f:X \to V_n$ such that $f(\cS) \subset \cS(V_n)$, we say that $\cS$ is a \define{circular} split system. Likewise, a finite metric space $(X, d_X)$ is \define{circular decomposable} if $\cS(X,d_X)$ is circular. If in addition $|\cS|=\binom{n}{2}$, we say that $\cS$ is a \define{full} circular system. We define \define{full} circular decomposable metrics analogously.
\end{defn}

\paragraph{Clique complexes of cyclic graphs.}
The main tool used in \cite{aa17} to compute the homotopy type of $\vr_r(\Sp^1)$ are directed cyclic graphs.

\begin{defn}
	\label{def:cyclic_graph_0}
	A directed graph $\vec{G} = (V,E)$ is called \define{cyclic} if $V$ has a cyclic order $x_1 \prec x_2 \prec \cdots \prec x_n$ and for every $(x_i, x_j) \in E$, either $j=i+1$ or $(x_{i+1}, x_j), (x_{i}, x_{j-1}) \in E$.
\end{defn}

For example, the 1-skeleton of $\vr_r(X)$ for any finite $X \subset \Sp^1$ and $r < 1/2$ can be given a natural orientation that makes it a cyclic graph. The clockwise orientation of $\Sp^1$ endows $X$ with a cyclic order, and a 1-simplex of $\vr_r(X)$ is oriented as $(x_i, x_j)$ if the shortest path from $x_i$ to $x_j$ is clockwise. We are interested in studying VR complexes that have cyclic 1-skeleton. Given that general metric spaces don't have a natural direction, we use the next definition to emulate the clockwise direction of $\Sp^1$.

\begin{defn}
	\label{def:monotone_metrics}
	Let $X = \{1, \dots, n\}$. A metric $d_X$ on $X$ is \define{monotone}\footnote{We adopted the term ``monotone'' from \cite{CDM-recognition}.} if $X$ has a cyclic order $1 \prec 2 \prec \cdots \prec n$ and there exists a function $M:X \to X$  such that $M(a) \neq a$ for all $a \in X$ and:
	\begin{enumerate}
		\item\label{it:mon_1} If $a \preceq b \prec c \preceq M(a)$, then $d_{b,c-1} < d_{bc}$ and $d_{b+1,c} \leq d_{bc}$.
		\item\label{it:mon_2} If $M(a) \prec b \prec c \preceq a$, then $d_{b,c-1} < b_{bc}$ and $d_{b+1,c} \leq d_{bc}$.
		\item\label{it:mon_3} $a \prec b \preceq M(a) \Rightarrow M(b) \preceq a \prec b$ and $M(b) \prec a \prec b \Rightarrow a \prec b \preceq M(a)$.
	\end{enumerate}
\end{defn}

The most natural example of monotone metrics are antipodal subsets of $\Sp^1$. In Section \ref{sec:circular-metrics}, we will show that the metric of any finite $X \subset \Sp^1$ is monotone.
\begin{example}
	\label{ex:circle_monotone}
	Let $n = 2k$, and let $p_1, p_2, p_3, \cdots, p_n$ be points of $\Sp^1$ in clockwise order so that $p_{i+k}$ is the antipodal point of $p_i$ (index sums are done modulo $n$). Let $d_{ab} := d_{\Sp^1}(p_a, p_b)$. Note that for any $a \preceq b \prec c \preceq a+k$, the points $p_a, p_b, p_c, p_{a+k}$ lie in that order in a semicircle, so $d_{b,c-1} < d_{bc}$ and $d_{b+1,c} < d_{bc}$. The same conclusion holds if $a+k \preceq b \prec c \preceq a$, so $M(a) = a+k$ for all $a$. Furthermore, for any $a \prec b \preceq M(a) = a+k$, we have $M(b) = b+k \preceq a +2k = a \prec b$. If $M(b) = b+k \prec a \preceq b$, then $b+2k \prec a+k \preceq b+k$. In other words, $b \prec M(a) \preceq M(b)$, which combined with $M(b) \prec a \prec b$, implies $a \prec b \preceq M(a)$.
\end{example}

\indent The function $M$ from Definition \ref{def:monotone_metrics} endows $G_r$, the 1-skeleton of $\vr_r(X)$, with an orientation. If $\{a, b\}$ is an edge of $G_r$, we orient it from $a$ to $b$ if $a \prec b \preceq M(a)$ and from $b$ to $a$ if $M(a) \prec b \prec a$. The last condition of Definition \ref{def:monotone_metrics} ensures that this orientation is not ambiguous whenever $a \neq M(b)$ because it prevents both $a \prec b \preceq M(a)$ and $b \prec a \preceq M(b)$ from holding simultaneously. Indeed, the first inequality along with $a \neq M(b)$ forces $M(b) \prec a \prec b$. However, the choice of orientation is not clear when $a = M(b)$ because both $a \prec b \preceq M(a)$ and $b \prec a \preceq M(b)$ hold. Despite this ambiguity, this orientation makes $G_r$ into a cyclic graph for all values of $r$ that we are interested in.

\begin{lemma}
	\label{lemma:monotone_implies_cyclic}
	Let $X = \{1, \dots, n\}$ and fix $0 < r < \rad(X)$. Let $G_r$ be the 1-skeleton of $\vr_r(X)$. If $d_X$ is monotone and $\vr_r(X)$ is not contractible, then $G_r$ is cyclic.
\end{lemma}
\begin{proof}
	\indent We claim that $\vr_r(X)$ is contractible if it has any edge $\{a, b\}$ such that $a \prec b \preceq M(a)$ and $a = M(b)$. Indeed, if $\{a, M(a)\}$ belongs to $\vr_r(X)$, then $r \geq d_{a, M(a)}$. By item \ref{it:mon_1} of Definition \ref{def:monotone_metrics}, we have $0 < d_{a,a+1} < d_{a,a+2} < \cdots < d_{a, M(a)} \leq r$. Similarly, starting with $b = M(a)+1$ in item \ref{it:mon_2}, we have $d_{M(a)+1, a} \geq d_{M(a)+2,a} \geq \cdots \geq d_{a-1,a}$. Lastly, note that $M(a)$ can never be $a$, so $a \prec b \preceq M(a)$ and $a = M(b)$ imply $b \preceq M(a) \prec a = M(b)$ and, thus, $b \prec M(a)+1 \preceq M(b)$. Then item \ref{it:mon_1} yields $d_{M(a)+1,a} \leq d_{M(a), a}$. In other words, $d_{ac} \leq d_{a,M(a)} \leq r$ for every $c \in X$, so $\vr_r(X)$ is a cone over $a$ and, thus, contractible\footnote{In fact, $r \geq \sup_{b \in X} d_{ab} \geq \rad(X)$.}.\\
	\indent If $\vr_r(X)$ is not contractible, every edge is oriented without ambiguity. Then for any edge $(a,b)$ in $\vec{G}_r$, we verify that $(a+1,b), (a,b-1) \in G_r$ whenever $b \neq a+1$. Since $(a,b) \in G_r$, we have $d_{ab} \leq r$ and either $a \prec b \preceq M(a)$ or $M(b) \prec a \prec b$. We assume $a \prec b \preceq M(a)$ because this inequality is implied by $M(b) \prec a \prec b$ and condition \ref{it:mon_3} of Definition \ref{def:monotone_metrics}. Since $a \prec b \preceq M(a)$ and $b \neq a-1$, we also have $a \prec a+1 \prec b \preceq M(a)$ and $a \prec b-1 \prec b \preceq M(a)$. Then by Definition \ref{def:monotone_metrics}, $d_{a+1,b} \leq d_{ab} \leq r$ and $d_{a,b-1} \leq d_{ab} \leq r$. Hence, $(a+1,b)$ and $(a,b-1)$ are edges of $\vec{G}_r$, verifying that $\vec{G}_r$ is cyclic.
\end{proof}

Thanks to the Lemma above, we can extend the definition of cyclic graphs to undirected graphs using an analogous function $M$.
\begin{defn}
	\label{def:cyclic_graph}
	Let $G$ be a graph with vertex set $X = \{1, \dots, n\}$. We say that $G$ is \define{cyclic} if $X$ has a cyclic order $1 \prec 2 \prec \cdots \prec n$ and there exists a function $M:X \to X$ such that:
	\begin{itemize}
		\item $a \prec b \preceq M(a) \Rightarrow M(b) \preceq a \prec b$ and $M(b) \prec a \prec b \Rightarrow a \prec b \preceq M(a)$.
		\item If $\{a, b\}$ is an edge of $G$ and $a \prec b \preceq M(a)$ then $b = a+1$ or $\{a+1, b\}$ and $\{a, b-1\}$ are edges of $G$,
	\end{itemize}
\end{defn}

Unoriented cyclic graphs can be oriented the same way as monotone metrics: an edge $\{a, b\}$ is oriented from $a$ to $b$ if $a \prec b \preceq M(a)$. Thanks to this, the results and constructions of \cite{aa17} also hold for undirected cyclic graphs. In particular, we extend the definitions of cyclicity and winding fraction to clique complexes.
\begin{defn}
	\label{def:cyclic_complex}
	A clique complex $\Cl(G)$ is called \define{cyclic} if its 1-skeleton $G$ is cyclic. Similarly, we define $\wf(\Cl(G)) := \wf(G)$ when $G$ is cyclic.
\end{defn}

The following is the main theorem of interest from \cite{aa17}.
\begin{theorem}[{\cite[Theorem 4.4]{aa17}}]
	\label{thm:aa17}
	If $G$ is a cyclic graph, then
	\begin{equation*}
		\Cl(G) \simeq
		\begin{cases}
			\Sp^{2l+1} & \text{if } \frac{l}{2l+1} < \wf(G) < \frac{l+1}{2l+3} \text{ for some } l=0, 1, \dots\\
			\bigvee^{n-2k-1} \Sp^{2l} & \text{if } \wf(G) = \frac{l}{2l+1} \text{ and } G \text{ dismantles to } C_{n}^{k}.
		\end{cases}
	\end{equation*}
\end{theorem}

\paragraph{Buneman complex.}
There is a polytopal complex associated to any weakly compatible split system $\cS$ on $X$ with weights $\alpha_S > 0$. Let
\begin{equation*}
	U(\cS) := \{A \subset X : \text{there exists } S \in \cS \text{ with } A \in S\}.
\end{equation*}
Given $\phi \in \R^{U(\cS)}$, let $\supp(\phi) := \{A \in U(\cS): \phi(A) \neq 0\}$. Define
\begin{equation*}
	H(\cS, \alpha) := \{ \phi \in \R^{U(\cS)} : \phi(A) \geq 0 \text{ and } \phi(A)+\phi(\overline{A}) = \tfrac{1}{2}\alpha(A|\overline{A}) \text{ for all } A \in U(\cS) \}.
\end{equation*}
Note that $H(\cS,\alpha)$ is polytope isomorphic to a hypercube of dimension $|\cS|$. The \emph{Buneman complex} of $(\cS,\alpha)$ is defined by
\begin{equation*}
	B(\cS,\alpha) := \{ \phi \in H(\cS,\alpha) : A_1, A_2 \in \supp(\phi) \text{ and } A_1 \cup A_2 = X \Rightarrow A_1 \cap A_2 = \emptyset\}.
\end{equation*}
Both $H(\cS,\alpha)$ and $B(\cS,\alpha)$ are equipped with the $L^1$ metric:
\begin{equation*}
	d_1(\phi,\phi') := \sum_{A \in U(\cS)} |\phi(A)-\phi'(A)|.
\end{equation*}
The space $(B(\cS,\alpha), d_1)$ admits an isometric embedding of $(X, d_{\cS,\alpha})$ via the map $x \mapsto (\phi_x: U(\cS) \to \R_{\geq 0})$ where
\begin{equation*}
	\phi_x(A) =
	\begin{cases}
		\frac{1}{2} \alpha_{A|\overline{A}} & \text{if } x \notin A \\
		0 & \text{otherwise}.
	\end{cases}
\end{equation*}

\paragraph{Tight span.}
\begin{defn}
	\label{def:tight_span}
	Let $(X,d_X)$ be a metric space. The \define{tight span} of $X$ is the set $T(X,d_X) \subset \R^X$ such that for all $f \in T(X,d_X)$:
	\begin{enumerate}
		\item $f(x) + f(y) \geq d_X(x,y)$ for all $x,y \in X$;
		\item $f(x) = \sup_{y \in X} \big[d_X(x,y)-f(y)\big]$ for all $x \in X$.
	\end{enumerate}
\end{defn}

These properties imply, in particular, that $f(x) \geq 0$ for all $f \in T(X,d_X)$ and $x \in X$.The tight span is equipped with the $L^\infty$ metric
\begin{equation*}
	d_\infty(f,g) = \sup_{x \in X} |f(x) - g(x)|,
\end{equation*}
and there exists an isometric embedding $h_\bullet: X \hookrightarrow T(X,d_X)$ defined by
\begin{equation*}
	x \mapsto \big( h_x: y \mapsto d_X(x,y) \big).
\end{equation*}
Tight spans are examples of injective metric spaces. A metric space $E$ is \define{injective} if for every 1-Lipschitz map $f:X \to E$ and for every isometric embedding $X \hookrightarrow \widetilde{X}$, there exists a 1-Lipschitz extension $\widetilde{f}:\widetilde{X} \to E$ that makes the following diagram commute:
\begin{equation*}
	\begin{tikzcd}
		X \arrow[hook]{r} \arrow{dr}[swap]{f}
		& \widetilde{X} \arrow[dashed]{d}{\widetilde{f}}
		\\
		& E
	\end{tikzcd}
\end{equation*}
In particular, $T(X,d_X)$ is the smallest injective space that contains $X$, i.e. there is no closed, injective subset $E \subset T(X, d_X)$ such that $X \hookrightarrow E$.\\
\indent Tight spans are an important tool in topological data analysis thanks to Theorem \ref{thm:VR-tight-span} below. Given a metric space $(E,d_E)$, $X \subset E$ and $r>0$, the \define{metric thickening} of $X$ in $E$ is the set
\begin{equation*}
	B_r(X;E) := \{e \in E | \text{ exists } x \in X \text{ with } d_E(x,e) < r\}.
\end{equation*}
If there is an isometric embedding $\iota:X \hookrightarrow E$, we will write $B_r(X; E)$ instead of $B_r(\iota(X); E)$.
\begin{theorem}[{\cite[Proposition 2.3]{osman-memoli}}]
	\label{thm:VR-tight-span}
	Let $(X,d_X)$ be a metric space, and let $E$ be an injective space such that $X \hookrightarrow E$. Then the Vietoris-Rips complex $\vr_{2r}(X,d_X)$ and the metric thickening $B_r(X; E)$ are homotopy equivalent for every $r>0$.
\end{theorem}

\noindent Below we list important topological and geometric properties of injective spaces and tight spans.
\begin{prop}
	\label{prop:TS_properties}
	Any injective metric space $E$ is contractible and geodesic.
\end{prop}
\noindent See Theorem 1.1 of \cite{isbell-six-theorems} and the discussion at the start of Section 2 of \cite{lang-injective-hulls}.

\begin{lemma}
	\label{lemma:eval_tight_span}
	Let $(X, d_X)$ be a metric space. For any $f \in L^\infty(X)$, $f \in T(X, d_X)$ if and only if
	\begin{equation*}
		f(x) = d_\infty(f, h_x).
	\end{equation*}
\end{lemma}
\noindent See property (2.4) of \cite{isbell-six-theorems}.

\begin{lemma}[{\cite[Lemma 6.2]{osman-memoli}}]
	\label{lemma:wedge_tight_span}
	If $E$ and $F$ are injective metric spaces, then so is their metric gluing along any two points.
\end{lemma}

\paragraph{Tight spans of totally decomposable metrics.}
The tight span of totally decomposable metrics has been studied by several authors throughout the years \cite{metric-decomposition, TDM-dim-2, TDM-polytopal-structure-1, TDM-polytopal-structure-2}. We use the notation and results of \cite{TDM-polytopal-structure-2}, one of the most recent papers. Given a weighted split system $(\cS, \alpha)$ on $X$, define the map
\begin{align*}
	\kappa:\R^{U(\cS)} &\to \R^X\\
	\phi &\mapsto (x \mapsto d_1(\phi,\phi_x)),
\end{align*}
where $d_1$ is the $L^1$ metric on $\R^{U(\cS)}$. The first salient property of $\kappa$ is that it sends each element $\phi_x$ to $h_x$. In other words, $\kappa$ is ``constant'' on the embedded copy of $X$ inside the Buneman complex and the tight span. The map has much stronger properties when $\cS$ is weakly compatible.

\begin{theorem}
	\label{thm:kappa-buneman-tight-span}
	Let $(\cS, \alpha)$ be a weighted split system  on $X$. The map $\kappa$ is 1-Lipschitz. Furthermore, $\kappa(B(\cS,\alpha)) = T(X,d_{\cS,\alpha})$ if and only if $\cS$ is weakly compatible.
\end{theorem}
See the introduction to Section 4 of \cite{TDM-polytopal-structure-1} to see why $\kappa(\phi_x) = h_x$ and why $\kappa$ is 1-Lipschitz. The equivalence between $\kappa(B(\cS,\alpha)) = T(X,d_{\cS,\alpha})$ and weak compatibility is the main result of \cite{comparison-tight-span-median-clusters}.

\begin{remark}
	As natural as the choice of $\kappa$ may seem, it is not an isometry in general. This is the main result of \cite{TDM-dim-2}. Note, however, that \cite{TDM-dim-2} uses the map $\Lambda_d$. It can be checked that $\kappa$ and $\Lambda_d$ differ by a constant using equation (2) in Section 4 of \cite{TDM-polytopal-structure-1}.
\end{remark}

The map $\kappa$ also induces a strong relationship between the polytopal structures of the Buneman complex and the tight span. We need more definitions to state this result.\\

\begin{defn}
	\label{def:blocks_cuts}
	Let $T$ be a connected polytopal complex. A vertex $v \in T$ is a \define{cut-vertex} if $T - \{v\}$ is disconnected. Note that if $T'$ is a connected component of $T - \{v\}$, then $T' \cup \{v\}$ is a connected subcomplex of $T$. We say that $v$ \define{separates} two subsets $A, B \subset T$ if $A \setminus \{v\}$ and $B \setminus \{v\}$ are contained in different connected components of $T \setminus \{v\}$. A maximal subcomplex $B \subset T$ that does not have cut-vertices is called a \define{block}. We denote the set of cut-vertices of $T$ as $\cut(T)$ and the set of blocks as $\blocks(T)$.
\end{defn}

\indent The following Definition and Lemma are a straightforward generalization of the block-cut tree of a graph and Theorem 1 of \cite{bc-tree}, respectively.
\begin{defn}
	\label{def:block_cut_graph}
	We define the \define{block-cut tree} of $T$ as the graph $\bctree(T)$ with vertex set $\cut(T) \cup \blocks(T)$ and edges of the form $(c, B)$ for every pair $c \in \cut(T), B \in \blocks(T)$ such that $c \in B$.
\end{defn}

\begin{lemma}
	\label{lemma:block_cut_graph_is_tree}
	$\bctree(T)$ is a tree for any any connected polytopal complex $T$.
\end{lemma}

\indent The block structure of the Buneman complex $B(\cS,\alpha)$ can be determined from the properties of $\cS$ in a straightforward way. We adopt the definition of compatibility from \cite{metric-decomposition} and the incompatibility graph from \cite{buneman-graph-cuts}.

\begin{defn}
	\label{def:incompatibility-graph}
	Let $\cS$ be a split system on $X$. Two splits $S, S' \in \cS$ are called \define{compatible} if the following equivalent conditions hold:
	\begin{itemize}
		\item There exist $A \in S$, $A' \in S'$ with $A \cap A' = \emptyset$;
		\item There exist $A \in S$, $A' \in S'$ with $A \cup A' = X$;
		\item There exist $A \in S$, $A' \in S'$ with $A \subset A'$ or $A' \subset A$.
	\end{itemize}
	If neither condition holds, we say that $S$ and $S'$ are \define{incompatible}. We define the \define{incompatibility graph} $I(\cS)$ to be the graph with vertex set $\cS$ and edge set consisting of all pairs $\{S,S'\}$ where $S$ and $S'$ are distinct incompatible splits.
\end{defn}

\begin{theorem}
	\label{thm:Buneman_block}
	Let $(\cS, \alpha)$ be a weighted split system on $X$. There is a bijective correspondence between the connected components of the incompatibility graph and the blocks of $B(\cS,\alpha)$ \cite[Theorem 5.1]{buneman-graph-cuts}. Moreover, for every $\cS' \in \pi_0(I(\cS))$, the block corresponding to $\cS'$ is isomorphic as a polytopal complex to $B(\cS', \alpha|_{\cS'})$.
\end{theorem}
See also the note at the start of Section 3.3 of \cite{TDM-polytopal-structure-2}. We denote the block of $B(\cS,\alpha)$ corresponding to $\cS' \in \pi_0(I(\cS))$ as $B_{\cS'}(\cS,\alpha)$.\\

\indent We need one more definition before stating the link between the polytopal structures of $B(\cS,\alpha)$ and $T(X,d_{\cS,\alpha})$. A split system $\cS$ is called \define{octahedral} if $|\cS|=4$ and there exists a partition $X = X_1 \cup \cdots \cup X_6$ such that $S_i = (X_i \cup X_{i+1} \cup X_{i+2}) | (X_{i+3} \cup X_{i+4} \cup X_{i+5})$ for $1 \leq i \leq 3$ (indices are taken modulo 6) and $S_4 = (X_1  \cup X_3 \cup X_5) | (X_2 \cup X_4 \cup X_6)$.

\begin{theorem}[Theorems 15 and 18 of \cite{TDM-polytopal-structure-2}]
	\label{thm:kappa-block-bijection}
	Let $(\cS,\alpha)$ be a weighted weakly compatible split system on $X$. Then $\kappa$ induces a bijection between $\blocks(B(\cS,\alpha))$ and $\blocks(T(X,d_{\cS,\alpha}))$ such that:
	\begin{itemize}
		\item If $\cS' \in \pi_0(I(\cS))$ is not octahedral, then the block $B_{\cS'}(\cS,\alpha)$ is isomorphic to the block $\kappa(B_{cS'}(\cS,\alpha)) \subset T(X,d_{\cS,\alpha})$ as polytopal complexes.
		\item If $\cS' \in \pi_0(I(\cS))$ is octahedral, then $\kappa(B_{\cS'}(\cS,\alpha))$ is a block isomorphic to a rhombic dodecahedron.
	\end{itemize}
\end{theorem} 	\section{VR complexes of circular decomposable metrics}
\label{sec:circular-metrics}
Given that circular decomposable spaces share many features with finite subsets of $\Sp^1$, we want to explore how similar their VR complexes are. For example, one might expect $\vr_r(X)$ to be cyclic when $(X, d_X)$ is circular decomposable, but the upcoming example shows otherwise. Hence, our objective in this section is to find conditions on the isolation indices that characterize when the VR complexes of a circular decomposable space are cyclic.

\begin{example}
	\label{ex:non-cyclic-5}
	Consider the metric wedge of $\Sp^1$ and an interval as shown in Figure \ref{fig:non-cyclic-5}. Let $Y = \{x_1 \prec x_2 \prec x_3 \prec x_4 \prec y\}$ and $X = \{x_1, \dots, x_5\}$. $Y$ has a circular split system and all its VR complexes are cyclic because $Y \subset \Sp^1$. Since $d_X(x_5, x_i) = d_Y(y, x_i) + (\tfrac{1}{5}+\epsilon)$ for $i \neq 5$, it follows that $\alpha_{\{x_5\}, X \setminus \{x_5\}} = \alpha_{\{y\}, Y \setminus \{y\}} + (\frac{1}{5}+\epsilon)$. Hence, $X$ also has a circular split system and it inherits the cycic order from $Y$ so that $x_1 \prec x_2 \prec x_3 \prec x_4 \prec x_5$. However, $d_X(x_5, x_1), d_X(x_5, x_4) > d_X(x_1, x_4)$, so $\vr_r(X)$ is not cyclic for $\frac{2}{5} \leq r < \frac{2}{5}+\epsilon$. Intuitively, this is because the edge $[x_1, x_4]$ appears in $\vr_r(X)$ before $[x_5, x_1]$ and $[x_5, x_4]$ do, even though $x_5$ is ``between'' $x_1$ and $x_4$.
	\begin{figure}[ht]
		\centering
		\begin{tikzpicture}
	\def \radius {2};
	\def \ptrad {3pt};
	
\coordinate (P1) at (180:\radius);
	\coordinate (P2) at (108:\radius);
	\coordinate (P3) at (36:\radius);
	\coordinate (P4) at (324:\radius);
	\coordinate (P5) at (252:\radius);
	
\fill (P1) circle (\ptrad) node [right] {$y$};
	\fill (P2) circle (\ptrad) node [above left] {$x_1$};
	\fill (P3) circle (\ptrad) node [above right] {$x_2$};
	\fill (P4) circle (\ptrad) node [below right] {$x_3$};
	\fill (P5) circle (\ptrad) node [below] {$x_4$};
	
\draw (P1) arc (180:108:\radius) node[midway, above left] {$1/5$};
	\draw (P2) arc (108:36:\radius) node[midway, above] {$1/5$};
	\draw (P3) arc (36:-36:\radius) node[midway, right] {$1/5$};
	\draw (P4) arc (324:252:\radius) node[midway, below] {$1/5$};
	\draw (P5) arc (252:180:\radius) node[midway, below left] {$1/5$};
	
\coordinate (P6) at (-2.5*\radius, 0);
	
	\draw (P1) -- (P6) node [midway, above] {$1/5+\varepsilon$};
	\fill (P6) circle (\ptrad) node [left] {$x_5$};
\end{tikzpicture} 		
		\caption{The space $Y := \{x_1, y, x_3, x_4, x_5\}$ consists of the vertices of a regular pentagon inscribed on the circle. We attach an edge $e$ of length $\frac{1}{5}+\epsilon$ to the circle at the point $y$ and define $x_2$ as the boundary of $e$ different from $y$. The circular decomposition of $Y$ induces a circular decomposition of $X := \{x_1, x_2, x_3, x_4, x_5\}$ that satisfies $d_X(x_1,x_2), d_X(x_2, x_3) > d_X(x_1,x_3)$. See Example \ref{ex:non-cyclic-5}.}
		\label{fig:non-cyclic-5}
	\end{figure}
\end{example}

For the rest of the section, we fix $X = \{1, \dots, n\}$ and assume that $d_X$ is circular decomposable and that the bijection $f:X \to V_n$ from Theorem \ref{thm:circular-splits} satisfies $f(i) = v_i$. In particular, $X$ inherits the cyclic order from $V_n$ (see Definition \ref{def:circular-order}) so that $i \prec j \prec k$ if the clockwise path from $v_i$ to $v_k$ contains $v_j$. We write $d_{ab} := d_X(a,b)$ for any $1 \leq a, b \leq n$.

\subsection{An expression for $d_X$ in terms of $\alpha_{ij}$}
We begin by simplifying Equation (\ref{eq:split-decomposition}) in order to write $d_X(a,b)$ as a sum of isolation indices rather than a linear combination of split metrics. Thanks to Theorem \ref{thm:circular-splits}, the splits $S$ in a full circular system $\cS(X, d_X)$ have the form $S_{ij} := A_{ij} | \overline{A}_{ij}$ where $A_{ij} := \{i, \dots, j-1\}$ and $1 \leq i < j \leq n$. Note that there are $\binom{n}{2}$ splits of the form $S_{ij}$ and that they are all distinct. Then if $\alpha_{ij} := \alpha_{S_{ij}}$, we have
\begin{equation}
	\label{eq:circular-distance-original}
	d_X = \sum_{1 \leq i < j \leq n} \alpha_{ij} \delta_{S_{ij}}
\end{equation}
by Theorem \ref{thm:split-decomposition}. If $\cS(X,d_X)$ is circular but not full, then $f(\cS(X,d_X)) \subsetneq \cS(V_n)$ (see Definition \ref{def:circular-splits}). In that case, we set $\alpha_{ij}=0$ for any split $S_{ij} \notin \cS(X,d_X)$ and Equation (\ref{eq:circular-distance-original}) still holds.

\indent Currently, $A_{ij}$ is only defined when $i < j$, but we can extend the definition to all pairs $1 \leq i, j \leq n$ by noting that $A_{ij} = \{k : i \preceq k \prec j\}$. This is a well-defined expression regardless of whether $i<j$ or not, and in fact,
\begin{itemize}
	\item If $1 \leq j < i \leq n$, $A_{ij} = \{i, \dots, n\} \cup \{1, \dots, j-1\} = \overline{A}_{ji}$;
	\item In the edge case $i=j$, $A_{ij} = X$.
\end{itemize}
\noindent Hence, we define:
\begin{itemize}
	\item For $1 \leq j < i \leq n$, $S_{ij} := S_{ji}$ and $\alpha_{ij}:= \alpha_{ji}$;
	\item If $i=j$, we set $\alpha_{ij} := 0$ and leave $S_{ij}$ undefined.
\end{itemize}
\noindent The reason to leave $S_{ij}$ undefined when $i=j$ is because because $A_{ij} = X$ forces $\overline{A}_{ij}=\emptyset$, and a pair $\{A, \overline{A}\}$ is only a valid split if both $A$ and $\overline{A}$ are non-empty. However, defining the coefficients $\alpha_{ii}$ as 0 will be convenient for later calculations.

\begin{remark}
	\label{rmk:isolation-indices-reflected}
	We summarize the above discussion for future reference. For any $1 \leq i, j \leq n$, $A_{ij}$ is defined as $\{k : i \preceq k \prec j\}$, and the isolation indices $\alpha_{ij}$ satisfy the relations $\alpha_{ij} = \alpha_{ji}$ and $\alpha_{ii} = 0$. If $i \neq j$, we define the split $S_{ij} := A_{ij} | \overline{A}_{ij}$ and note that it also satisfies the relation $S_{ji} = S_{ij}$.
\end{remark}

With the notation in place, we can evaluate the split metrics in Equation (\ref{eq:circular-distance-original}) and simplify $d_{ab}$.
\begin{lemma}
	\label{lemma:circular-distance}
	For any distinct $a, b \in X$, $\displaystyle d_{ab} = \osum_{i=a+1}^{b} \osum_{j=b+1}^{a} \alpha_{ij}$.
\end{lemma}
\begin{proof}
	Suppose $a < b$. If $1 \leq i < j \leq n$, then $a \in A_{ij}$ and $b \notin A_{ij}$ if and only if $i \leq a < j \leq b$. Similarly, $a \notin A_{ij}$ and $b \in A_{ij}$ if and only if $a < i \leq b < j$. Let $\chi$ be an indicator function. Then by Equation (\ref{eq:circular-distance-original}) and the symmetry of $\alpha_{ij}$,
	\begin{align*}
		d_{ab}
		&= \sum_{1 \leq i < j \leq n} \alpha_{ij} \delta_{S_{ij}}(a,b) \\
		&= \sum_{1 \leq i < j \leq n} \alpha_{ij} \cdot \chi(i \leq a < j \leq b) + \sum_{1 \leq i < j \leq n} \alpha_{ij} \cdot \chi(a < i \leq b < j)\\
		&= \sum_{i=1}^{a} \sum_{j=a+1}^{b} \alpha_{ij} + \sum_{i=a+1}^{b} \sum_{j=b+1}^{n} \alpha_{ij}
		= \sum_{i=a+1}^{b} \sum_{j=1}^{a} \alpha_{ij} + \sum_{i=a+1}^{b} \sum_{j=b+1}^{n} \alpha_{ij} \\
		&= \osum_{i=a+1}^{b} \osum_{j=b+1}^{a} \alpha_{ij}.
	\end{align*}
	If $a>b$, we obtain the desired formula for $d_{ab} = d_{ba}$ by swapping the roles of $a$ and $b$ in the equation above.
\end{proof}

\subsection{Inequalities in circular decomposable spaces}
\label{sec:inequalities-circular-decomposable}
A nice consequence of Lemma \ref{lemma:circular-distance} is that we can characterize inequalities between distances in $X$ in terms of isolation indices. We will use these expressions to characterize the circular decomposable spaces that have cyclic VR complexes. We begin by rewriting inequalities between distances in terms of isolation indices.

\begin{lemma}
	\label{lemma:dists_smaller}
	If $a \prec b \prec c$, we have
	\begin{itemize}
		\item $d_{ab} \leq d_{ac}$ if and only if $\osum_{i=a+1}^{b} \osum_{j=b+1}^{c} \alpha_{ij} \leq \osum_{i=c+1}^{a} \osum_{j=b+1}^{c} \alpha_{ij}$.
		\item $d_{bc} \leq d_{ac}$ if and only if $\osum_{i=b+1}^{c} \osum_{j=a+1}^{b} \alpha_{ij} \leq \osum_{i=c+1}^{a} \osum_{j=a+1}^{b} \alpha_{ij}$.
	\end{itemize}
\end{lemma}
\begin{proof}
	The lemma follows from the observations that
	\begin{align*}
		d_{ab}
		= \osum_{i=a+1}^{b} \osum_{j=b+1}^{a} \alpha_{ij}
		&= \osum_{i=a+1}^{b} \osum_{j=b+1}^{c} \alpha_{ij} + \osum_{i=a+1}^{b} \osum_{j=c+1}^{a} \alpha_{ij}\\
d_{ac}
		= \osum_{i=a+1}^{c} \osum_{j=c+1}^{a} \alpha_{ij}
		&= \osum_{i=a+1}^{b} \osum_{j=c+1}^{a} \alpha_{ij} + \osum_{i=b+1}^{c} \osum_{j=c+1}^{a} \alpha_{ij}\\
d_{bc}
		= \osum_{i=b+1}^{c} \osum_{j=c+1}^{b} \alpha_{ij}
		&= \osum_{i=b+1}^{c} \osum_{j=c+1}^{a} \alpha_{ij} + \osum_{i=b+1}^{c} \osum_{j=a+1}^{b} \alpha_{ij}
	\end{align*}
	and the symmetry of $\alpha_{ij}$.
\end{proof}

If $\vr_r(X)$ is cyclic for every $0 < r < \rad(X)$, several notable inequalities between distances in $X$ must be satisfied. For example, suppose that we have a function $M:X \to X$ as in Definition \ref{def:cyclic_graph} that makes $\vr_r(X)$ a cyclic complex for all $0 < r < \rad(X)$. If $a \prec c \preceq M(a)$ and $d_{ac} < r$, then $\vr_r(X)$ contains a directed edge $(a,c)$ and, by cyclicity, $(a,c-1)$ and $(a+1,c)$ as well. This only happens if $d_{a,c-1} < r$ and $d_{a+1,c} < r$, so if we let $r \searrow d_{ac}$, we obtain $d_{a,c-1} \leq d_{ac}$ and $d_{a+1,c} \leq d_{ac}$. Iterating this argument shows that $\vr_r(X)$ being cyclic forces the chain of inequalities
\begin{equation*}
	d_{a,a+1} \leq d_{a,a+2} \leq \cdots \leq d_{a,c-1} \leq d_{ac}.
\end{equation*}
Thanks to Lemma \ref{lemma:dists_smaller}, we can characterize $d_{a,c-1} \leq d_{ac}$ and $d_{a+1,c} \leq d_{ac}$ in terms of isolation indices, with the added benefit that if we rearrange the sums of isolation indices corresponding to either $d_{a,c-1} \leq d_{ac}$ or $d_{a+1,c} \leq d_{ac}$, we can produce several more inequalities. Although it would be desirable that $d_{a,c-1} \leq d_{ac}$ implied $d_{a,c-2} \leq d_{a,c-1}$ in order to obtain the chain above, the rearrangement yields a different conclusion.

\begin{lemma}
	\label{lemma:unimodal_alpha}
	Let $d_X$ be a circular decomposable metric. Then for any $a \neq c$,
	\begin{itemize}
		\item $d_{a,c-1} \leq d_{ac} \Leftrightarrow \osum_{i=a+1}^{c} \alpha_{ic} \leq \osum_{i=c+1}^{a} \alpha_{ic}$,
		\item $d_{a+1,c} \leq d_{ac} \Leftrightarrow \osum_{i=a+1}^{c} \alpha_{i,a+1} \leq \osum_{i=c+1}^{a} \alpha_{i,a+1}$.
	\end{itemize}
	As a consequence, if $a \prec b \prec c$, $d_{a,c-1} \leq d_{ac}$ implies $d_{b,c-1} \leq d_{bc}$ and $d_{a+1,c} \leq d_{ac}$ implies $d_{a+1,b} \leq d_{ab}$.
\end{lemma}
\begin{proof}
	By Lemma \ref{lemma:dists_smaller}, $d_{a,c-1} \leq d_{ac}$ if and only if $\osum_{i=a+1}^{c-1} \alpha_{ic} \leq \osum_{i=c+1}^{a} \alpha_{ic}$ and $d_{a+1,c} \leq d_{ac}$ if and only if $\osum_{i=a+2}^{c} \alpha_{i,a+1} \leq \osum_{i=c+1}^{a+1} \alpha_{i,a+1}$. Since $\alpha_{a+1,a+1} = 0$, the latter is equivalent to $\osum_{i=a+1}^{c} \alpha_{i,a+1} \leq \osum_{i=c+1}^{a} \alpha_{i,a+1}$. Then if $a \prec b \prec c$,
	\begin{align*}
		d_{a,c-1} \leq d_{ac}
		&\Leftrightarrow \osum_{i=a+1}^{c} \alpha_{ic} = \osum_{i=a+1}^{b} \alpha_{ic} + \osum_{i=b+1}^{c} \alpha_{ic} \leq \osum_{i=c+1}^{a} \alpha_{ic}\\
		&\Rightarrow \osum_{i=b+1}^{c} \alpha_{ic} \leq \osum_{i=c+1}^{a} \alpha_{ic} + \osum_{i=a+1}^{b} \alpha_{ic} = \osum_{i=c+1}^{b} \alpha_{ic} \\
		&\Leftrightarrow d_{b,c-1} \leq d_{bc}.
	\end{align*}
	The proof of $d_{a+1,c} \leq d_{ac} \Rightarrow d_{a+1,b} \leq d_{ab}$ is analogous.
\end{proof}

Note that instead of $d_{a,c-1} \leq d_{ac} \Rightarrow d_{a,c-2} \leq d_{a,c-1}$, Lemma \ref{lemma:unimodal_alpha} produces $d_{a,c-1} \leq d_{ac} \Rightarrow d_{a+1,c-1} \leq d_{a+1,c}$. This implication is unfortunately not enough to produce the chain
\begin{equation*}
	d_{a,a+1} \leq d_{a,a+2} \leq \cdots \leq d_{a,c-1} \leq d_{ac}.
\end{equation*}
However, if we had a second chain $d_{b,b+1} \leq d_{b,b+2} \leq \cdots$ for some $a \prec b \prec c$, then it must at least reach $d_{b,c-1} \leq d_{bc}$ thanks to Lemma \ref{lemma:unimodal_alpha}. Hence, it will be informative to study the relationship between $\osum_{i=a+1}^{c} \alpha_{ic}$ and $\osum_{i=c+1}^{a} \alpha_{ic}$.\\
\indent The proof of Lemma \ref{lemma:unimodal_alpha} uses the fact that the sum $\osum_{i=c+1}^{b} \alpha_{ic}$ is at its smallest when $b = c+1$ and increases as $b$ cycles through $c+1, c+2, \dots, c-1$. Likewise, the sum $\osum_{i=b+1}^{c} \alpha_{ic}$ achieves its maximum when $b = c$ and steadily decreases in the same range. We focus on the point when $\osum_{i=c+1}^{b} \alpha_{ic}$ becomes larger than $\osum_{i=b+1}^{c} \alpha_{ic}$.

\begin{defn}
	\label{def:unimodal_sum}
	Let $d_X$ be a circular decomposable metric and fix $c \in X$. We define $\sigma(c)$ as the last element $a$ from the sequence $c+1, c+2, \dots, c-1$ such that $\displaystyle \osum_{i=a+1}^{c} \alpha_{ic} \geq \osum_{i=c+1}^{a} \alpha_{ic}$.
\end{defn}

\begin{remark}
	\label{rmk:unimodal_sum}
	It is possible that $\sigma(c)$ is not defined if $\osum_{i=c+2}^c \alpha_{ic} < \alpha_{c+1,c}$ holds, but we will assume this never happens. In Section \ref{sec:blocks_of_TDS}, we will show that if the above inequality holds for some $c \in X$, then the homotopy types of $\vr_r(X)$ and $\vr_r(X \setminus \{c\})$ differ by an isolated point, at most. Since this is a small difference, we will apply the results of this section to $X \setminus C$, where $C$ is the set of $c \in X$ for which $\sigma(c)$ is not defined, and then add an isolated point to $\vr_r(X \setminus C)$ according to Proposition \ref{prop:point_with_long_edge}.
\end{remark}

\noindent For convenience, we record the following consequence of Lemma \ref{lemma:unimodal_alpha}.
\begin{corollary}
	\label{cor:unimodal_alpha}
	Let $(X, d_X)$ be a circular decomposable space. Then for any $a \neq c$,
	\begin{itemize}
		\item $c \prec a \preceq \sigma(c) \Leftrightarrow \osum_{i=a+1}^{c} \alpha_{ic} \geq \osum_{i=c+1}^{a} \alpha_{ic} \Leftrightarrow d_{a,c-1} \geq d_{ac}$.
		\item $\sigma(c) \prec a \prec c \Leftrightarrow \osum_{i=a+1}^{c} \alpha_{ic} < \osum_{i=c+1}^{a} \alpha_{ic} \Leftrightarrow d_{a,c-1} < d_{ac}$.
	\end{itemize}
	The equality $d_{a,c-1} = d_{ac}$ happens if and only if $\osum_{i=c+1}^{a} \alpha_{ic} = \osum_{i=a+1}^{c} \alpha_{ic}$.
\end{corollary}

\begin{example}
	\label{ex:sigma}
	Suppose that $\alpha_{ij} = 1$ for every $i \neq j$. If $1 \leq c < n/2 < a \leq n$, we have $\osum_{i=c+1}^a \alpha_{ic} = a-c$ and $\osum_{i=a+1}^c \alpha_{ic} = (n-1) - \osum_{i=c+1}^a \alpha_{ic} = (n-1) - (a-c)$ (recall $\alpha_{cc}=0$). Note that $a-c$ is smaller than $(n-1) - (a-c)$ whenever $a-c \leq (n-1)/2$, so $\sigma(c) = c + \lfloor \frac{n-1}{2} \rfloor$. Similarly, when $1 \leq a < n/2 \leq c \leq n$, we have $\osum_{i=a+1}^c \alpha_{ic} = (c-1)-a$ (recall $\alpha_{cc}=0$) and $\osum_{i=c+1}^a \alpha_{ic} = (n-1) - (c-1-a)$. Then $(n-1) - (c-1-a) \leq (c-1-a)$ when $c-a \geq \frac{n+1}{2}$, so $\sigma(c) = c - \lceil \frac{n+1}{2} \rceil$. This equals $c + \lfloor \frac{n-1}{2} \rfloor$ because $\lfloor \frac{n-1}{2} \rfloor + \lceil \frac{n+1}{2} \rceil = n$. Hence, $\sigma(c) = c + \lfloor \frac{n-1}{2} \rfloor$ for all $1 \leq c \leq n$.\\
	\indent Table \ref{tab:sigma} shows the resulting distance matrix for $n=6$, with the maximum value of each row in boldface and the entries $(\sigma(c), c)$ underlined for every $1 \leq c \leq n$. By Corollary \ref{cor:unimodal_alpha}, $d_{a,c-1} < d_{ac}$ if $\sigma(c) \prec a \prec c$ and $d_{a,c-1} \geq d_{ac}$ otherwise. We can read this in Table \ref{tab:sigma} as follows. We pick any entry $(a,c)$ with $a \neq c$ and go down in the $c$-th column, wrapping around to the top if we reach the bottom of the matrix. If we reach the diagonal $(c,c)$ before crossing the line in column $c$, then $\sigma(c) \prec a \prec c$. Otherwise, $c \prec a \preceq(c)$. In other words, if the entry $(a,c)$ is ``under'' the line in column $c$ and ``above'' the diagonal, then $d_{a,c-1} < d_{ac}$. Otherwise, $d_{a,c-1} \geq d_{ac}$.\\
	\indent We can also visualize the second conclusion of Lemma \ref{lemma:unimodal_alpha}. If we have a chain $d_{a,a+1} < d_{a,a+2} < \cdots < d_{a,a+m}$, then every entry $(a,c)$ with $a+1 \preceq c \preceq m$ is below the line in column $c$. Consequently, the entries $(a+1,c)$ are below the line in column $c$ for all $a+1 \prec c \preceq m$, so $d_{a+1,a+2} < \cdots < d_{a+1,m}$.
	
	\begin{table}[h]
		\centering
		\begin{tabular}{l|rrrrrr}
			& 1 & 2 & 3 & 4 & 5 & 6 \\
			\hline
			1 & 0 & 5 & 8 & \textbf{9} & 8 & 5 \\ \cline{6-6}
			2 & 5 & 0 & 5 & 8 & \textbf{9} & 8 \\ \cline{7-7}
			3 & 8 & 5 & 0 & 5 & 8 & \textbf{9} \\ \cline{2-2}
			4 & \textbf{9} & 8 & 5 & 0 & 5 & 8 \\ \cline{3-3}
			5 & 8 & \textbf{9} & 8 & 5 & 0 & 5 \\ \cline{4-4}
			6 & 5 & 8 & \textbf{9} & 8 & 5 & 0 \\ \cline{5-5}
		\end{tabular}
		\caption{Distance matrix of a circular decomposable metric on $\{1, \dots, 6\}$ such that $\alpha_{ij}=1$ for all $i \neq j$. The maximum value of each row is boldface, and the entries $(\sigma(a), a)$ underlined for every $1 \leq a \leq n$. As in Example \ref{ex:sigma}, if the entry $(a,c)$ is ``under'' the line in column and ``above'' the diagonal entry $(c,c)$ (wrapping around the bottom of the matrix if necessary), then $\sigma(c) \prec a \prec c$ and $d_{a,c-1} < d_{ac}$. Otherwise, $d_{a,c-1} \geq d_{ac}$.}
		\label{tab:sigma}
	\end{table}
\end{example}

\subsection{Monotone circular decomposable metrics}
\label{subsec:circular-metrics-monotone}

In this section, we use $\sigma$ to produce a candidate function $M:X \to X$ that might satisfy Definition \ref{def:monotone_metrics}. As we argued before, we at least need the chain
\begin{equation}
	\label{eq:increasing_chain}
	d_{a,a+1} < d_{a,a+2} < \cdots < d_{a,m-1} < d_{am}
\end{equation}
for some $m \in X$, which by Corollary \ref{cor:unimodal_alpha}, is equivalent to $\sigma(b) \prec a \prec b$ for all $a \prec b \preceq m$. In fact, if there existed $m$ such that $\sigma(b) \prec a \prec b \Leftrightarrow a \prec b \preceq m$, we would have the second half $d_{am} \geq d_{m+1,a} \geq \cdots \geq d_{a-1,a}$. These inequalities suggest that $M(a) := m$ might satisfy Definition \ref{def:monotone_metrics}. Of course, there are other requirements like $d_{a+1,c} \leq d_{ac}$ when $a \prec c \preceq M(a)$, but we will prove them later using the properties of $\sigma$.

\begin{defn}
	\label{def:M_from_sigma}
	Let $d_X$ be a circular decomposable metric and fix $a \in X$. We define $M(a)$ be the last element $m$ of the sequence $\{a+1, a+2, \dots, a-1\}$ that satisfies $\sigma(m) \prec a \prec m$.
\end{defn}

Note that, a priori, it is not required that $\sigma(b) \prec a \prec b$ for every $a \prec b \prec m$. However, we can set conditions on $\sigma$ to ensure that this happens. Heuristically, if all isolation indices have roughly the same value, $\sigma(a)$ and $M(a)$ must happen halfway through $a+1$ and $a-1$.

\begin{example}
	\label{ex:sigma_and_M}
	In Example \ref{ex:sigma}, we found that $\sigma(a) = a+\lfloor \frac{n-1}{2} \rfloor$ when $\alpha_{ij} = 1$ for every $i \neq j$. Note that $\sigma(c) = c+\lfloor \frac{n-1}{2} \rfloor \prec a \prec c$ if and only if $c \prec a + \lceil \frac{n+1}{2} \rceil \prec c + \lceil \frac{n+1}{2} \rceil$, and combining these inequalities yields $a \prec c \prec a + \lceil \frac{n+1}{2} \rceil$. Hence, $M(a) = a + \lceil \frac{n+1}{2} \rceil - 1 = a + \lfloor \frac{n}{2} \rfloor$.
\end{example}

In general, we don't need such a rigid formula like $\sigma(c) = c+r$ and $M(a) = a+s$ (for $1 \leq r, s \leq n$) like in Examples \ref{ex:sigma} and \ref{ex:sigma_and_M} to ensure that $d_X$ is monotone. At the very least, we need $\sigma(c) \prec a \prec c$ for any $a \prec c \prec M(a)$ in order for the chain in Equation (\ref{eq:increasing_chain}) to hold. Since we want a condition that can be checked using only the isolation indices, we instead study the following implication that we call Property (\ref{eq:star}):
\begin{align}
	\label{eq:star}
	\begin{split}
		a \prec b \preceq \sigma(a) &\Rightarrow \sigma(a) \preceq \sigma(b) \prec a.
\end{split}
	\tag{$\star$}
\end{align}

\noindent The following consequence of Property (\ref{eq:star}) will be useful later.
\begin{lemma}
	\label{lemma:star_2}
	Let $d_X$ be a circular decomposable metric such that $\sigma$ satisfies Property (\ref{eq:star}). Then $a \prec b \preceq \sigma(a) \preceq \sigma(b) \prec a$ for any $a \prec b \preceq \sigma(a)$.
\end{lemma}
\begin{proof}
	By Property \ref{eq:star}, $a \prec b \preceq \sigma(a)$ implies $\sigma(a) \preceq \sigma(b) \prec a$. Combining the two inequalities yields the conclusion.
\end{proof}

Intuitively, Property (\ref{eq:star}) and Lemma \ref{lemma:star_2} say that $\sigma:X \to X$ is non-decreasing in the cyclic order and that it does not increase so fast that it violates $\sigma(c) \prec a \prec c$ for some $a \prec c \preceq \sigma(c)$.

\begin{example}
	\label{ex:sigma_and_M_star}
	The function $\sigma$ from Example \ref{ex:sigma_and_M} satisfies Property (\ref{eq:star}). Indeed, $\sigma(a) = a+\lfloor \frac{n-1}{2} \rfloor$ for any $a$, so $a \prec b \preceq a+\lfloor \frac{n-1}{2} \rfloor$ implies $a + \lfloor \frac{n-1}{2} \rfloor \prec b + \lfloor \frac{n-1}{2} \rfloor \preceq a + 2\lfloor \frac{n-1}{2} \rfloor$. In particular, $a + 2\lfloor \frac{n-1}{2} \rfloor$ and $b + \lfloor \frac{n-1}{2} \rfloor$ can never reach $a$ because $2 \lfloor \frac{n-1}{2} \rfloor < n$. Hence, we obtain $\sigma(a) \prec \sigma(b) \prec a$, a special case of Property (\ref{eq:star}).
\end{example}

\indent Thanks to Property (\ref{eq:star}), we can show that $\sigma(c) \prec a \prec c$ for every $a \prec c \preceq M(a)$. We then verify that the maximum of $d_{ac}$ for a fixed $a$ happens when $c = M(a)$.
\begin{lemma}
	\label{lemma:star_implies_unimodal}
	Let $d_X$ be a circular decomposable metric such that $\sigma$ satisfies Property (\ref{eq:star}). Let $M$ be the function from Definition \ref{def:M_from_sigma}. Then $\sigma(b) \prec a \prec b$ for every $a \prec b \preceq M(a)$.
\end{lemma}
\begin{proof}
	Let $m := M(a)$ and suppose, for a contradiction, that $b \prec a \preceq \sigma(b)$ for some $a \prec b \prec m$. The previous inequality implies $b \prec m \prec a$, so we also have $b \prec m \prec \sigma(b)$. By Property (\ref{eq:star}), $\sigma(b) \preceq \sigma(m) \prec b$. However, by Definition \ref{def:M_from_sigma}, $\sigma(m) \prec a \prec b \prec m$, so we get $\sigma(b) \preceq \sigma(m) \prec a \prec b$ as well. This contradicts the assumption $b \prec a \preceq \sigma(b)$.
\end{proof}

\begin{corollary}
	\label{cor:monotone_1}
	Let $d_X$ be a circular decomposable metric such that $\sigma$ satisfies Property (\ref{eq:star}). Let $M$ be the function from Definition \ref{def:M_from_sigma}. Then:
	\begin{itemize}
		\item $a \prec b \preceq M(a)$ if and only if $d_{a, b-1} < d_{ab}$.
		\item $M(a) \prec b \preceq a$ if and only if $d_{a, b-1} \geq d_{ab}$.
	\end{itemize}
\end{corollary}
\begin{proof}
	If $b=a$, there is nothing to prove. Otherwise, by Definition \ref{def:M_from_sigma}, $M(a) \prec b \prec a \Rightarrow b \prec a \preceq \sigma(b)$, while Lemma \ref{lemma:star_implies_unimodal} yields $a \prec b \preceq M(a) \Rightarrow \sigma(b) \prec a \prec b$. Then, by Corollary \ref{cor:unimodal_alpha},
	\begin{itemize}
		\item $a \prec b \preceq M(a) \Leftrightarrow \sigma(b) \prec a \prec b \Leftrightarrow d_{a,b-1} < d_{ab}$, and
		\item $M(a) \prec b \prec a \Leftrightarrow b \prec a \preceq \sigma(b) \Leftrightarrow d_{a,b-1} \geq d_{ab}$.
	\end{itemize}
\end{proof}

Corollary \ref{cor:monotone_1} concludes the proof of the chain in Equation (\ref{eq:increasing_chain}), but we haven't fully exploited Lemma \ref{lemma:unimodal_alpha} yet. In particular, for any $a \prec c \preceq M(a)$, Equation (\ref{eq:increasing_chain}) and Lemma \ref{lemma:unimodal_alpha} imply that $d_{c,c+1} \leq d_{c, c+2} \leq \cdots \leq d_{c, M(a)}$ and, thus, $c \prec M(a) \preceq M(c)$ by Corollary \ref{cor:monotone_1}. This ``monotonicity'' of $M$ can be generalized further.
\begin{lemma}
	\label{lemma:star_M}
	Let $(X, d_X)$ be a circular decomposable space and let $M:X \to X$ be the function from Definition \ref{def:M_from_sigma}.
	\begin{itemize}
		\item $a \prec b \prec M(a) \Rightarrow b \prec M(a) \preceq M(b)$,
		\item $M(a) \prec b \prec a \Rightarrow M(b) \preceq M(a) \prec b$.
	\end{itemize}
\end{lemma}
\begin{proof}
	Assume $a \prec b \prec M(a)$. By definition of $M(a)$, $\sigma(M(a)) \prec a \prec M(a)$. Since $a \prec b \prec M(a)$, we also have $\sigma(M(a)) \prec b \prec M(a)$. By Definition \ref{def:M_from_sigma}, $M(b)$ is the last $m$ of the sequence $b+1, b+2, \dots, b-1$ with $\sigma(m) \prec b \prec m$, so $b \prec M(a) \preceq M(b)$.\\
	\indent Now suppose $M(a) \prec b \prec a$. If $b = M(a)+1$, the fact that $M(b) \neq b$ by definition implies $M(b) \preceq b-1 =M(a) \prec b$. Suppose $b \neq M(a)+1$. By definition of $M(a)$, we have $M(a)+1 \prec a \preceq \sigma(M(a)+1)$ which, together with the hypotheses on $b$, implies $M(a)+1 \prec b \prec \sigma(M(a)+1)$. By definition of $M(b)$, $M(b) \prec M(a)+1 \prec b$, so $M(b) \preceq M(a) \prec b$.
\end{proof}

With the previous properties at our disposal, we prove that the $M$ from Definition \ref{def:M_from_sigma} almost satisfies the first two conditions of Definition \ref{def:monotone_metrics}. Note the the edge cases in the following Lemma and in Definition \ref{def:monotone_metrics} differ.
\begin{lemma}
	\label{lemma:star_implies_monotone_p1}
	Let $d_X$ be a circular decomposable metric such that $\sigma$ satisfies Property \ref{eq:star}. Let $M$ be the function from Definition \ref{def:M_from_sigma}. Then:
	\begin{enumerate}
		\item\label{it:sim_1} $a \preceq b \prec c \preceq M(a)$ implies $d_{b,c-1} < d_{bc}$.
		\item\label{it:sim_2} $a \preceq b \prec c \prec M(a)$ implies $d_{b+1,c} \leq d_{bc}$.
		\item\label{it:sim_3} $M(a) \preceq b \prec c \preceq a$ implies $d_{b+1,c} \leq d_{bc}$.
		\item\label{it:sim_4} $M(a) \prec b \prec c \preceq a$ implies $d_{b,c-1} < b_{bc}$.
	\end{enumerate}
\end{lemma}
\begin{proof}
	Throughout this proof, we assume  $b \prec b+1 \prec c$ as there is nothing to show if $c = b+1$.\\
	\noindent \ref{it:sim_1}. By Corollary \ref{cor:monotone_1}, $a \prec c \preceq M(a)$ implies $d_{a,c-1} < d_{ac}$ and Lemma \ref{lemma:unimodal_alpha} gives $d_{b,c-1} < d_{bc}$.\\
	\noindent \ref{it:sim_2}. Since $a \preceq b \prec M(a)$, Lemma \ref{lemma:star_M} gives $b \prec M(a) \preceq M(b)$, and since $b \prec c \prec M(a)$, we get $b \prec c \prec M(b)$. Lemma \ref{lemma:star_M} then implies $M(c) \preceq M(b) \prec b$. Recall that $c \prec M(c) \preceq c-1$ by definition, so $c \prec M(c) \preceq M(b) \prec b \prec c$. Thus, $M(c) \prec b+1 \preceq c$, so by Corollary \ref{cor:monotone_1}, $d_{b+1,c} \leq d_{bc}$.\\
	\noindent \ref{it:sim_3}. By Corollary \ref{cor:monotone_1}, $M(a) \preceq b \prec b+1 \preceq a$ implies $d_{b+1,a} \leq d_{ba}$. Then Lemma \ref{lemma:unimodal_alpha} gives $d_{b+1,c} \leq d_{bc}$ because $b \prec c \preceq a$.\\
	\noindent \ref{it:sim_4}. By definition of $M(a)$, $b \prec a \preceq \sigma(b)$ and, since $b \prec c \preceq a$, we also have $b \prec c \preceq \sigma(b)$. By Property (\ref{eq:star}), $\sigma(b) \preceq \sigma(c) \prec b$. Combining the last two inequalities gives
	\begin{equation*}
		b \prec c \preceq \sigma(b) \preceq \sigma(c) \prec b,
	\end{equation*}
	so $\sigma(c) \prec b \prec c$. Then Corollary \ref{cor:unimodal_alpha} yields $d_{b,c-1} < d_{bc}$.
\end{proof}

Note that the edge cases of Lemma \ref{lemma:star_implies_monotone_p1} and Definition \ref{def:monotone_metrics} are not the same. For example, fix $a \in X$ and set $m = M(a)$. According to Definition \ref{def:monotone_metrics}, we should have $d_{a+1,m} \leq d_{am}$ because $a \preceq a \prec m = M(a)$, but this is precisely the case that is not covered by item \ref{it:sim_2} of Lemma \ref{lemma:star_implies_monotone_p1}. The next example shows a metric space where the opposite inequality $d_{a+1,m} > d_{am}$ actually holds.

\begin{example}
	\label{ex:star_implies_monotone_p1}
	Consider the points $p_1 = 0$, $p_2 = 0.1$, $p_3 = 0.3$, $p_4 = 0.6$, $p_5 = 0.8$ in $\Sp^1$. The set $p_1 \prec p_2 \prec \cdots \prec p_5$ is circular decomposable, and its distance matrix is in Table \ref{tab:star_implies_monotone}. By Corollary \ref{cor:unimodal_alpha}, $\sigma(c)$ is the last $a$ in the sequence $c+1, c+2, \dots, c-1$ such that $d_{a,c-1} \geq d_{ac}$. Then $\sigma(1) = \sigma(2) = 3$, $\sigma(3)=4$, $\sigma(4)=5$, and $\sigma(5)=2$. We also find $M(1) = 4$, $M(2) = 4$, $M(3) = 5$, $M(4) = 2$ and $M(5) = 3$. Even though $\sigma$ satisfies Property (\ref{eq:star}), $M$ does not satisfy Definition \ref{def:monotone_metrics} because $1 \preceq 1 \prec 4 \preceq M(1)$ and yet $d_{24} = 0.5 > 0.4 = d_{14}$. Instead, if we defined $\overline{M}(1) = 3$ and $\overline{M}(a) = M(a)$ for $a \neq 1$, then $1 \preceq 1 \prec 4 \preceq \overline{M}(1)$ no longer holds and we don't need to have $d_{24} \leq d_{14}$.
	
	\begin{table}[h]
		\centering
		\begin{tabular}{l|ccccc}
			& 1 & 2 & 3 & 4 & 5 \\
			\hline
			1 & 0 & 0.1 & 0.3 & \textbf{0.4} & 0.2 \\
			2 & 0.1 & 0 & 0.2 & \textbf{0.5} & 0.3 \\ \cline{6-6}
			3 & 0.3 & 0.2 & 0 & 0.3 & \textbf{0.5} \\ \cline{2-3}
			4 & 0.4 & \textbf{0.5} & 0.3 & 0 & 0.2 \\ \cline{4-4}
			5 & 0.2 & 0.3 & \textbf{0.5} & 0.2 & 0 \\ \cline{5-5}
		\end{tabular}
		\caption{Distance matrix of $\{0, 0.1, 0.3, 0.6, 0.8\} \subset \Sp^1$. The entries $(a, M(a))$ are in boldface and $(\sigma(a), a)$ are underlined for every $1 \leq a \leq 5$.}
		\label{tab:star_implies_monotone}
	\end{table}
\end{example}

Despite the discrepancies between Lemma \ref{lemma:star_implies_monotone_p1} and Definition \ref{def:monotone_metrics}, our choice of $M$ is not far off and, with a little more information on $M(M(a))$, we can adjust our choices. To simplify notation, fix $a \in X$ and let $m = M(a)$. If $M(m) \preceq a \prec m$, we get $d_{am} \geq d_{m,a+1}$ by Corollary \ref{cor:monotone_1} -- this is the inequality that we want. If $m \prec a \prec M(m)$, we get $d_{a-1,m} < d_{am} < d_{m,a+1}$ instead, so we need to adjust $M(a)$ to avoid this case.

\begin{defn}
	\label{def:M_bar}
	Let $(X, d_X)$ be a circular decomposable space, and let $M$ be the function from Definition \ref{def:M_from_sigma}. Fix $a \in X$ and let $m = M(a)$. Then
	\begin{equation*}
		\overline{M}(a) :=
		\begin{cases}
			M(a) & M(m) \preceq a \prec m \\
			M(a)-1 & m \prec a \prec M(m).
		\end{cases}
	\end{equation*}
\end{defn}

Having modified our original $M$, we can show that $\overline{M}$ satisfies Definition \ref{def:monotone_metrics}.
\begin{prop}
	\label{prop:star_implies_monotone}
	Let $d_X$ be a circular decomposable metric such that $\sigma$ satisfies Property \ref{eq:star}. Let $\overline{M}$ be the function from Definition \ref{def:M_bar}. Then:
	\begin{enumerate}
		\item\label{it:fix_1} $a \preceq b \prec c \preceq \overline{M}(a)$ implies $d_{b,c-1} < d_{bc}$.
		\item\label{it:fix_2} $a \preceq b \prec c \preceq \overline{M}(a)$ implies $d_{b+1,c} \leq d_{bc}$.
		\item\label{it:fix_3} $\overline{M}(a) \prec b \prec c \preceq a$ implies $d_{b+1,c} \leq d_{bc}$.
		\item\label{it:fix_4} $\overline{M}(a) \prec b \prec c \preceq a$ implies $d_{b,c-1} < b_{bc}$.
	\end{enumerate}
\end{prop}
\begin{proof}
	Let $m = M(a)$. We have two cases depending on the value of $\overline{M}(a)$.\\
	\noindent \underline{Case 1}: $\overline{M}(a) = M(a)$. By Lemma \ref{lemma:star_implies_monotone_p1}, the Proposition holds except for the edge case $c = \overline{M}(a)$ in item \ref{it:fix_2}, so we assume $a \preceq b \prec c = \overline{M}(a)$. Since $\overline{M}(a) = m$, Definition \ref{def:M_bar} requires $M(m) \preceq a \prec m$ and, thus, $M(m) \preceq b \prec c = m$. Then Lemma \ref{lemma:star_implies_monotone_p1} item \ref{it:sim_3} implies $d_{b+1,m} \leq d_{bm}$.\\
	\noindent \underline{Case 2}: $\overline{M}(a) = m-1$. If $a \preceq b \prec c \preceq \overline{M}(a) \prec M(a)$, Lemma \ref{lemma:star_implies_monotone_p1} items \ref{it:sim_1} and \ref{it:sim_2} imply $d_{b,c-1} < d_{bc}$ and $d_{b+1,c} \leq d_{bc}$. At the same time, $\overline{M}(a) \prec b \prec c \preceq a$ implies $M(a) \preceq b \prec c \preceq a$, so Lemma \ref{lemma:star_implies_monotone_p1} items \ref{it:sim_3} and \ref{it:sim_4} imply $d_{b+1,c} \leq d_{bc}$ and, if $b \neq M(a)$, $d_{b,c-1} < d_{bc}$. Hence, we only have to verify $d_{b,c-1} < d_{bc}$ when $M(a) = b \prec c \preceq a$. Since $\overline{M}(a) = m-1$, we must have $m \prec a \prec M(m)$ and, thus, $m = b \prec c \prec M(m)$. Thus, by Lemma \ref{lemma:star_implies_monotone_p1} item \ref{it:sim_1}, $d_{b,c-1} \leq d_{bc}$.
\end{proof}

\begin{prop}
	\label{prop:M_non_ambiguous}
	Let $d_X$ be a circular decomposable metric such that $\sigma$ satisfies Property \ref{eq:star}. Let $\overline{M}$ be the function from Definition \ref{def:M_bar}. Then:
	\begin{enumerate}
		\item\label{it:M_1} $a \prec b \preceq \overline{M}(a) \Rightarrow \overline{M}(b) \preceq a \prec b$, and
		\item\label{it:M_2} $\overline{M}(b) \prec a \prec b \Rightarrow a \prec b \preceq \overline{M}(a)$.
	\end{enumerate}
\end{prop}
\begin{proof}~\\
	\ref{it:M_1}. Assume, for contradiction, that $b \prec a \prec \overline{M}(b)$ for some $a \prec b \preceq \overline{M}(a)$. Since $a \neq \overline{M}(b)$, we also have $b \prec a+1 \preceq \overline{M}(b)$, so by Proposition \ref{prop:star_implies_monotone} item \ref{it:fix_3}, $d_{ba} = d_{b,(a+1)-1} < d_{b,a+1}$. At the same time, Proposition \ref{prop:star_implies_monotone} item \ref{it:fix_1} and $a \prec b \preceq \overline{M}(a)$ imply $d_{a+1,b} \leq d_{ab}$, a contradiction. Hence, we must have $\overline{M}(b) \preceq a \prec b$ for every $a \prec b \preceq \overline{M}(a)$.\\
	\ref{it:M_2}. Assume, for contradiction, that $\overline{M}(a) \prec b \prec a$ for some $\overline{M}(b) \prec a \prec b$. We have three cases depending on the values of $\overline{M}(a)$ and $\overline{M}(b)$.\\
	\underline{Case 1}: If $\overline{M}(a) = M(a)$, the assumption becomes $M(a) \prec b \prec a$, so by Corollary \ref{cor:monotone_1}, $d_{a,b-1} \geq d_{ab}$. However, $\overline{M}(b) \prec a \prec b$ and Proposition \ref{prop:star_implies_monotone} item \ref{it:fix_4} imply $d_{a,b-1} < d_{ab}$, a contradiction.\\
	\underline{Case 2}: Since the assumptions are symmetric on $a$ and $b$, the case of $\overline{M}(b) = M(b)$ is analogous.\\
	\underline{Case 3}: If $\overline{M}(a) = M(a)-1$ and $\overline{M}(b) = M(b)-1$, Definition \ref{def:M_bar} requires $M(a) \prec a \prec M^2(a)$ and $M(b) \prec b \prec M^2(b)$. With this constraint, the inequalities $\overline{M}(a) \prec b \prec a$ and $\overline{M}(b) \prec a \prec b$ imply $M(a) \preceq b \prec a \prec a+1 \preceq M^2(a)$ and $M(b) \preceq a \prec b \prec M^2(b)$. Then Lemma \ref{lemma:star_implies_monotone_p1} items \ref{it:sim_1} and \ref{it:sim_2} imply $d_{ab} = d_{b,(a+1)-1} < d_{b,a+1}$ and $d_{a+1,b} \leq d_{ba}$, a contradiction.\\
	This finishes the proof that $\overline{M}(b) \prec a \prec b$ implies $a \prec b \preceq \overline{M}(a)$.
\end{proof}

In fact, we can also prove a converse to Propositions \ref{prop:star_implies_monotone} and \ref{prop:M_non_ambiguous}. If $d_X$ is a monotone metric and $\overline{M}$ satisfies Definition \ref{def:monotone_metrics}, then $\sigma$ must satisfy Property (\ref{eq:star}). We prove this in two parts.
\begin{lemma}
	\label{lemma:monotone_implies_star_p1}
	If $d_X$ is a monotone circular decomposable metric such that $\overline{M}$ satisfies Definition \ref{def:monotone_metrics}, then $a \prec b \preceq \sigma(a)$ implies $\sigma(b) \prec a \prec b$.
\end{lemma}
\begin{proof}
	By Definition \ref{def:M_from_sigma}, $a \prec b \preceq \sigma(a)$ implies $M(b) \prec a \prec b$. Since $M(b)-1 \preceq \overline{M}(b) \preceq M(b)$ by Definition \ref{def:M_bar}, we may write $\overline{M}(b) \preceq M(b) \prec a \prec b$. The last condition in Definition \ref{def:monotone_metrics} then gives $a \prec b \preceq \overline{M}(a)$. By Definition \ref{def:monotone_metrics}, $d_{a,b-1} < d_{ab}$, so Corollary \ref{cor:unimodal_alpha} yields $\sigma(b) \prec a \prec b$.
\end{proof}

\begin{lemma}
	\label{lemma:monotone_implies_star_p2}
	If $d_X$ is a monotone circular decomposable metric such that $\overline{M}$ satisfies Definition \ref{def:monotone_metrics}, then $\sigma$ satisfies Property (\ref{eq:star}).
\end{lemma}
\begin{proof}
	We want to show that $a \prec b \preceq \sigma(a)$ implies $\sigma(a) \preceq \sigma(b) \prec a$. Assume, for contradiction, that $a \preceq \sigma(b) \prec \sigma(a)$ instead. Depending on the position of $\sigma(b)$ relative to $b$, we may have either $a \prec b \preceq \sigma(b) \prec \sigma(a)$ or $a \preceq \sigma(b) \prec b \preceq \sigma(a)$. However, Lemma \ref{lemma:monotone_implies_star_p1} implies $\sigma(b) \prec a \prec b$, which disallows $a \preceq \sigma(b) \prec b$, so $a \prec b \preceq \sigma(b) \prec \sigma(a)$. Then for any $\sigma(b) \prec c \preceq \sigma(a)$, we have $\sigma(b) \prec c \prec b$ and $a \prec c \preceq \sigma(a)$. By Definition \ref{def:M_from_sigma}, $c \prec b \preceq M(c)$ and $M(c) \prec a \prec c$ but, since $a \prec b \prec c$, we also have $M(c) \prec b \prec c$. This is a contradiction.
\end{proof}

The results of this section can be summarized into the following characterization of monotone circular decomposable metrics.
\begin{theorem}
	\label{thm:monotone_equals_star}
	Let $d_X$ be a circular decomposable metric. Then $\overline{M}$ satisfies Definition \ref{def:monotone_metrics} (and hence, $d_X$  is monotone) if and only if $\sigma$ satisfies Property \ref{eq:star}.
\end{theorem}
\begin{proof}
	Propositions \ref{prop:star_implies_monotone} and \ref{prop:M_non_ambiguous} imply that $\overline{M}$ satisfies Definition \ref{def:monotone_metrics} whenever $\sigma$ satisfies Property (\ref{eq:star}), and the converse is given by Lemma \ref{lemma:monotone_implies_star_p2}.
\end{proof}

With the previous theorem, we can write a formula for the homotopy type of the VR complex of any circular decomposable metric that satisfies Property (\ref{eq:star}).
\begin{corollary}
	\label{cor:VR_circular_monotone}
	Let $d_X$ be a circular decomposable metric such that $\sigma$ satisfies Property (\ref{eq:star}). Fix $r>0$ and let $G_r$ be the 1-skeleton of $\vr_r(X)$. Then
	\begin{equation*}
		\vr_r(X) \simeq
		\begin{cases}
			\Sp^{2l+1} & \text{if } \frac{l}{2l+1} < \wf(G_r) < \frac{l+1}{2l+3} \text{ for some } l=0, 1, \dots\\
			\bigvee^{n-2k-1} \Sp^{2l} & \text{if } \wf(G_r) = \frac{l}{2l+1} \text{ and } G_r \text{ dismantles to } C_{n}^{k}.
		\end{cases}
	\end{equation*}
\end{corollary}
\begin{proof}
	By Theorem \ref{thm:monotone_equals_star}, $\overline{M}$ satisfies Definition \ref{def:monotone_metrics}. Hence, by Lemma \ref{lemma:monotone_implies_cyclic}, $G_r$ is cyclic for any $0 < r < \rad(X)$ and Theorem \ref{thm:aa17} applies.
\end{proof}
 	\section{Non-monotone circular decomposable metrics}
\label{sec:circular-metrics-non-monotone}
In this section, we describe how to compute the homology groups of some circular decomposable metrics that are not monotone. However, circular decomposable metrics are not too far from being monotone, so we can use the properties of monotone metrics and cyclic graphs to compute the homology of more general circular decomposable metrics. Concretely, our strategy to compute $H_*(\vr_r(X))$ will be to split $\vr_r(X)$ into the cyclic piece and the non-cyclic piece, compute the homology of the cyclic piece with \ref{thm:aa17}, and use the Mayer-Vietoris sequence to find the homology of $\vr_r(X)$.

\subsection{Cyclic component of $\vr_r(X)$ for a circular decomposable non-monotone $X$.}
\label{sec:cyclic_component}
\indent Fix $r>0$ and suppose $(X, d_X)$ is a circular decomposable metric. For ease of notation, we denote $\vr_r(X)$ with $V_X$ and its 1-skeleton with $G_X$. We use the following function to test if $d_X$ is monotone.

\begin{defn}
	\label{def:M_r}
	Let $(X, d_X)$ be a circular decomposable metric and fix $r>0$. Given $a \in X$, let $M_r(a)$ be the last $b$ in the sequence $a+1, a+2, \dots, a-1$ such that $d_X(c, d) \leq r$ for all $a \preceq c \prec d \preceq b$.
\end{defn}

Notice that for any $a \preceq c \prec d \preceq b \preceq M_r(a)$, both $d_{cd}$ and $d_{ab}$ are smaller than $r$, so $\{a,b\} \in V_X \Rightarrow \{c,d\} \in V_X$ holds. However, we don't know if the same implication holds when $M_r(a) \prec b \preceq c \prec d \preceq a$. The implication holds if $d_X$ is monotone and, in that case, we can use the results from the previous section to find $H_*(V_X)$. If $d_X$ is not monotone, there exists $M_r(a) \prec b \prec a$ and $b \preceq c \prec d \preceq a$ such that $d_{ab} \leq r < d_{cd}$. Intuitively, $\vr_r(X)$ is not cyclic because the edge $\{a,b\}$ appeared at a lower $r$ than it was supposed to. These are the edges we want to remove, so we set
\begin{equation*}
	E_X := \left\{ \{a,b\} \subset X : M_r(a) \prec b \prec a \text{ and } d_{ab} \leq r < d_{cd} \text{ for some } b \preceq c \prec d \preceq a \right\}.
\end{equation*}

\begin{defn}
	\label{def:cyclic_component}
	Let $r>0$ and suppose $(X, d_X)$ is a circular decomposable metric. We form a graph $G_X^c$ by removing the edges in $E_X$ from $G_X$ and define $V_X^c := \Cl(G_X^c)$. We call $V_X^c$ the \define{cyclic component} of $\vr_r(X)$.
\end{defn}

\indent By definition, $V_X^c$ is a cyclic complex because we removed the edges that made $V_X$ non-cyclic. Plus, if $V_X$ was cyclic to begin with, then $V_X^c = V_X$. On the other hand, let $Y$ be the vertex set of $\operatorname{St}(E_X)$, the closed star of $E_X$ in $V_X$. Notice that the subcomplex of $\vr_r(X)$ induced by $Y$ equals $V_Y = \vr_r(Y)$ and that $\operatorname{St}(E_X) \subset V_Y$.\\
\indent Intuitively, we have separated $V_X$ into its cyclic part $V_X^c$ and the subcomplex $V_Y$ induced by edges that prevent $V_X$ from being cyclic. We want to use these complexes and the Mayer-Vietoris sequence to compute the homology of $V_X$, and there are a couple of facts to verify. First, let $V_Y' := V_X^c \cap V_Y$.

\begin{lemma}
	\label{lemma:VY'_cyclic}
	$V_Y'$ is a cyclic clique complex.
\end{lemma}
\begin{proof}
	First, note that $V_Y'$ is a clique complex because both $V_X^c$ and $V_Y$ are, and that $X$ induces a cyclic order $y_1 \prec \cdots \prec y_m$ on the points of $Y$.\\
	\indent To verify that $V_Y'$ is cyclic, define $M_Y:Y \to Y$ by setting $M_Y(y_i)$ to be the last $z$ from the sequence $y_{i+1}, y_{i+2}, \dots, y_{i-1}$ such that $y_i \preceq z \preceq M_r(y_i)$. Let $y_a, y_b, y_c \in Y$ such that $y_a \preceq y_b \prec y_c \preceq M_Y(y_a)$. To show that $V_Y'$ is cyclic, we have to prove that $\{y_b, y_c\} \in V_Y'$ implies $\{y_{b+1}, y_{c}\}, \{y_{b}, y_{c-1}\} \in V_Y'$. If $\{y_b, y_c\} \in V_Y' \subset V_X^c$, the fact that $V_X^c$ is cyclic and $y_a \preceq y_b \prec M_Y(y_a) \preceq M_r(y_a)$ mean that $\{y_{b+1}, y_{c}\}, \{y_{b}, y_{c-1}\} \in V_X^c \subset V_X$ by definition of $M_r$. Since $V_Y$ is the subcomplex of $V_X$ induced by $Y$, we must have $\{y_{b+1}, y_{c}\}, \{y_{b}, y_{c-1}\} \in V_Y$. Hence, $\{y_{b+1}, y_{c}\}$, $\{y_{b}, y_{c-1}\} \in V_X^c \cap V_Y = V_Y'$. The second condition in Definition \ref{def:cyclic_graph} follows analogously.
\end{proof}

Now we make sure that we didn't lose any simplices when splitting $V_X$ into $V_X^c$ and $V_Y$.
\begin{lemma}
	\label{lemma:MV_union}
	$V_X = V_X^c \cup V_Y$.
\end{lemma}
\begin{proof}
	Let $\sigma$ be a simplex of $V_X$. If $\sigma$ does not contain an edge from $E_X$, then $\sigma \in V_X^c$ by definition of $V_X^c$. If $\sigma$ does contain an edge from $E_X$, then $\sigma \in \operatorname{St}(E_X) \subset V_Y$.
\end{proof}

Lastly, we record the following for future use.
\begin{prop}
	\label{prop:MV}
	Let $\iota_X:V_Y' \hookrightarrow V_X^c$, $\iota_Y:V_Y' \hookrightarrow V_Y$, $\tau_X:V_X^c \hookrightarrow V_X$ and $\tau:V_Y \hookrightarrow V_X$ be the natural inclusions. The homology groups of $V_X$, $V_X^c$, $V_Y$, and $V_Y'$ satisfy the following Mayer-Vietoris sequence:
	\begin{equation}
		\label{eq:MV}
		\cdots \to H_{k}(V_Y') \xrightarrow{(\iota_X, \iota_Y)} H_{k}(V_X^c) \oplus H_{k}(V_Y) \xrightarrow{\tau_X - \tau_Y} H_{k}(V_X) \xrightarrow{\partial_*} H_{k-1}(V_Y') \to \cdots.
	\end{equation}
\end{prop}

\subsection{Recursive computation of $H_*(\vr_r(X))$.}
\label{sec:recursive_computation}
There is an obstacle to using Proposition \ref{prop:MV}. If $X=Y$, $V_Y$ equals $V_X$ because we defined $V_Y$ as an induced subcomplex. This would imply $V_Y' = V_X^c$, which means that the Mayer-Vietoris sequence would give the trivial statements $H_*(V_X) = H_*(V_Y)$ and $H_*(V_X^c) = H_*(V_Y')$. Hence, we assume $X \neq Y$ in this section.\\
\indent First we focus on finding $H_*(V_X)$ when $V_Y'$ is in the non-critical regime, i.e. $\frac{l}{2l+1} < \wf(V_Y') < \frac{l+1}{2l+3}$ for some $l \in \Z$. We begin with some technical properties.

\begin{lemma}
	\label{lemma:ses}
	Let $f:A \to B$, $g:A \to C$, $h:B \to D$, $k:C \to D$ be group homomorphisms and suppose 
	\begin{equation*}
		A \xrightarrow{(f,g)} B \oplus C \xrightarrow{h-k} D \to 0
	\end{equation*}
	is exact. If $g$ is an injection/surjection/isomorphism, then so is $h$.
\end{lemma}
\begin{proof}
	Suppose $g$ is injective. If $h(b)=0$, then $h(b)-k(0) = 0$. By exactness, there exists $a \in A$ such that $f(a) = b$ and $g(a) = 0$. Since $g$ is injective, $a=0$, so $b = f(a) = 0$.\\
	\indent Suppose $g$ is surjective. Given $d \in D$, there exist $b \in B$ and $c \in C$ such that $h(b)-k(c) = d$ by exactness. Also, there exists $a \in A$ such that $g(a)=c$ by assumption. Exactness implies that $(h \circ f - k \circ g)(a) = 0$. Thus, $d = h(b)-k(c) = h(b) - k(g(a)) = h(b - f(a))$.
\end{proof}

\begin{lemma}
	\label{lemma:inclusion_is_iso}
	Suppose $\frac{l}{2l+1} < \wf(V_Y') < \frac{l+1}{2l+3}$ for some $l \in \Z$. Then $\iota_X:H_{2l+1}(V_Y') \to H_{2l+1}(V_X^c)$ is an isomorphism or 0 depending on whether $\frac{l}{2l+1} < \wf(V_X^c) < \frac{l+1}{2l+3}$ or not.
\end{lemma}
\begin{proof}
	Observe that the inclusion of cyclic complexes $V_Y' \hookrightarrow V_X^c$ is induced by the inclusion of their 1-skeleta, which is a cyclic homomorphism of cyclic graphs by \cite[Lemma 3.6]{aa17}. As a consequence, \cite[Proposition 3.8 (b)]{aa17} implies that $\wf(V_Y') \leq \wf(V_X^c)$. If in addition $\frac{l}{2l+1} < \wf(V_X^c) < \frac{l+1}{2l+3}$, then \cite[Proposition 4.9]{aa17}, implies $V_Y' \simeq V_X^c$. Otherwise, $H_{2l+1}(V_X^c)=0$ by Theorem \ref{thm:aa17}.
\end{proof}

\begin{theorem}
	\label{thm:MV_non_critical}
	Suppose $\frac{l}{2l+1} < \wf(V_Y') < \frac{l+1}{2l+3}$.
	\begin{enumerate}
		\item If $\iota_Y: H_{2l+1}(V_Y') \to H_{2l+1}(V_Y)$ is injective and $\operatorname{coker}(\iota_Y)$ is torsion-free, then
		\begin{align*}
			H_{2l+1}(V_X) &\cong H_{2l+1}(V_X^c) \oplus \left( H_{2l+1}(V_Y) / H_{2l+1}(V_Y') \right)\\
			\widetilde{H}_k(V_X) &\cong \widetilde{H}_k(V_X^c) \oplus \widetilde{H}_k(V_Y) \text{ for } k \neq 2l+1.
		\end{align*}
		\item If $\iota_Y: H_{2l+1}(V_Y') \to H_{2l+1}(V_Y)$ is the 0 map, then
		\begin{align*}
			H_{2l+1}(V_X) &\cong H_{2l+1}(V_Y) \\
			H_{2l+2}(V_X) &\cong H_{2l+2}(V_X^c) \oplus H_{2l+2}(V_Y) \oplus \ker(\iota_X) \\
			\widetilde{H}_k(V_X) &\cong \widetilde{H}_k(V_X^c) \oplus \widetilde{H}_k(V_Y) \text{ for } k \neq 2l+1, 2l+2.
		\end{align*}
	\end{enumerate}
\end{theorem}

\begin{remark}
	\label{rmk:MV_non_critical}
	There is a simpler formula for $H_{2l+2}(V_X)$ in the second case of Theorem \ref{thm:MV_non_critical}:
	\begin{equation*}
		H_{2l+2}(V_X) =
		\begin{cases}
			H_{2l+2}(V_Y) & \text{if } \frac{l}{2l+1} < \wf(V_X^c) < \frac{l+1}{2l+3} \\
			H_{2l+2}(V_X^c) \oplus H_{2l+2}(V_Y) \oplus H_{2l+1}(V_Y') & \text{otherwise.}
		\end{cases}
	\end{equation*}
	Indeed, Lemma \ref{lemma:inclusion_is_iso} gives that $H_{2l+1}(V_X^c)$ is either isomorphic to $H_{2l+1}(H_Y')$ or zero. This gives $\ker(\iota_X)=0$ and $H_{2l+2}(V_X^c)=0$ in the first case and $\ker(\iota_X) = H_{2l+1}(V_Y')$ in the second.
\end{remark}

\begin{proof}
	Recall Proposition \ref{prop:MV}:
	\begin{equation*}
		\cdots \to H_{k}(V_Y') \xrightarrow{(\iota_X, \iota_Y)} H_{k}(V_X^c) \oplus H_{k}(V_Y) \xrightarrow{\tau_X - \tau_Y} H_{k}(V_X) \xrightarrow{\partial_*} H_{k-1}(V_Y') \to \cdots.
	\end{equation*}
	By Theorem \ref{thm:aa17}, $\widetilde{H}_{k}(V_Y')$ equals $\Z$ if $k=2l+1$ and 0 otherwise, so $\widetilde{H}_{k}(V_X) \cong \widetilde{H}_{k}(V_X^c) \oplus \widetilde{H}_{k}(V_Y)$ for $k \neq 2l+1, 2l+2$. The remaining part of the sequence is as follows:
\begin{equation}
		\label{eq:MV_non_critical}
		\begin{aligned}
		0 &\to H_{2l+2}(V_X^c) \oplus H_{2l+2}(V_Y) \to H_{2l+2}(V_X) \\
		&\xrightarrow{\partial_*} H_{2l+1}(V_Y') \xrightarrow{\iota} H_{2l+1}(V_X^c) \oplus H_{2l+1}(V_Y) \to H_{2l+1}(V_X) \to 0.
		\end{aligned}
	\end{equation}
	\indent If $\iota_Y: H_{2l+1}(V_Y') \to H_{2l+1}(V_Y)$ is injective, then so is $\iota$ and the bottom row becomes a short exact sequence. This forces $\partial_* = 0$, so $H_{2l+2}(V_X) \cong H_{2l+2}(V_X^c) \oplus H_{2l+2}(V_Y)$.\\
	\indent Finding $H_{2l+1}(V_X)$ requires two cases depending on whether $H_{2l+1}(V_X^c)$ is trivial or not. If $H_{2l+1}(V_X^c) = 0$, the fact that the bottom row of (\ref{eq:MV_non_critical}) is a short exact sequence yields $H_{2l+1}(V_X) \cong H_{2l+1}(V_Y) / H_{2l+1}(V_Y')$. On the other hand, a non-trivial $H_{2l+1}(V_X^c)$ only happens when $\frac{l}{2l+1} < \wf(V_X^c) < \frac{l+1}{2l+3}$ by Theorem \ref{thm:aa17}, and in that case, Lemma \ref{lemma:inclusion_is_iso} implies that the map $\iota_X:H_{2l+1}(V_Y') \to H_{2l+1}(V_X^c)$ is an isomorphism. Hence, $\tau_Y: H_{2l+1}(V_Y) \to H_{2l+1}(V_X)$ is an isomorphism by Lemma \ref{lemma:ses}. Moreover, since $\operatorname{coker}(\iota_Y)$ is torsion-free by hypothesis, the short exact sequence
	\begin{equation}
		\label{eq:MV_non_critical_case_1_ses}
		0 \to H_{2l+1}(V_Y') \xrightarrow{\iota_Y} H_{2l+1}(V_Y) \to \operatorname{coker}(\iota_Y) \to 0
	\end{equation}
	splits. Thus,
	\begin{align*}
		H_{2l+1}(V_X)
		&\cong H_{2l+1}(V_Y)
		\cong H_{2l+1}(V_Y') \oplus \operatorname{coker}(\iota_Y) \\
		&\cong H_{2l+1}(V_X^c) \oplus \left( H_{2l+1}(V_Y) / H_{2l+1}(V_Y') \right).
	\end{align*}
	
	\indent In the second claim, $\iota_Y:H_{2l+1}(V_Y') \to H_{2l+1}(V_Y)$ is the 0 map. We also have that $H_{2l+1}(V_X^c)$ is either 0 or isomorphic to $H_{2l+1}(V_Y')$ by Lemma \ref{lemma:inclusion_is_iso}, so the image of $\iota$ is $H_{2l+1}(V_X^c)$. Since $\tau_X-\tau_Y$ is surjective, the first isomorphism theorem and exactness of (\ref{eq:MV_non_critical}) give
	\begin{align*}
		H_{2l+1}(V_X)
		&\cong (H_{2l+1}(V_Y) \oplus H_{2l+1}(V_X^c)) / \ker(\tau_X - \tau_Y)\\
		&= (H_{2l+1}(V_Y) \oplus H_{2l+1}(V_X^c)) / \operatorname{Im}(\iota)\\
		&\cong H_{2l+1}(V_Y).
	\end{align*}
	Lastly, Equation (\ref{eq:MV_non_critical}) reduces to the short exact sequence
	\begin{equation*}
			0 \to H_{2l+2}(V_X^c) \oplus H_{2l+2}(V_Y) \to H_{2l+2}(V_X)
			\xrightarrow{\partial_*} \operatorname{Im}(\partial_*) \to 0.
	\end{equation*}
	By Lemma \ref{lemma:inclusion_is_iso}, $\operatorname{Im}(\partial_*)$ equals either 0 or $H_{2l+1}(V_Y')$. In the latter case, $\operatorname{Im}(\partial_*)$ is free, so the sequence above splits. In either case, $H_{2l+2}(V_X) \cong H_{2l+2}(V_X^c) \oplus H_{2l+2}(V_Y) \oplus \operatorname{Im}(\partial_*)$. This finishes the proof.
\end{proof}

\begin{corollary}
	\label{cor:MV_non_critical}
	Suppose $\frac{l}{2l+1} < \wf(V_Y') < \frac{l+1}{2l+3}$.
	\begin{enumerate}
		\item If $\iota_Y: H_{2l+1}(V_Y') \to H_{2l+1}(V_Y)$ is injective and $\operatorname{coker}(\iota_Y)$ is torsion-free, then $\tau_X:H_{2l+1}(V_X^c) \to H_{2l+1}(V_X)$ is injective and $\operatorname{coker}(\tau_X)$ is torsion-free.
		\item If $\iota_Y: H_{2l+1}(V_Y') \to H_{2l+1}(V_Y)$ is the 0 map, then $\tau_X:H_{2l+1}(V_X^c) \to H_{2l+1}(V_X)$ is the 0 map.
	\end{enumerate}
\end{corollary}

\begin{proof}
	Suppose that $\iota_Y: H_{2l+1}(V_Y') \to H_{2l+1}(V_Y)$ is injective and $\operatorname{coker}(\iota_Y)$ torsion-free. If $H_{2l+1}(V_X^c) \cong 0$, then $\tau_X$ is immediately injective. Otherwise, consider the commutative diagram induced by inclusions:
	\begin{equation*}
		\begin{tikzcd}[row sep=tiny]
			& H_{2l+1}(V_Y) \arrow[hook]{dr}{\tau_Y}[swap]{\cong}\\
			H_{2l+1}(V_Y')
			\arrow[hook]{ur}{\iota_Y}
			\arrow[hook]{dr}[swap]{\iota_X}{\cong}
			&& H_{2l+1}(V_X). \\
			& H_{2l+1}(V_X^c) \arrow[hook]{ur}[swap]{\tau_X}
		\end{tikzcd}
	\end{equation*}
	The maps $\iota_X$ and $\tau_Y$ are isomorphisms by Lemma \ref{lemma:inclusion_is_iso} and the proof of Theorem \ref{thm:MV_non_critical}, respectively. Since Equation (\ref{eq:MV_non_critical_case_1_ses}) is short exact, $\iota_Y$ is the inclusion of the $H_{2l+1}(V_Y')$ summand into $H_{2l+1}(V_Y) \cong H_{2l+1}(V_Y') \oplus \operatorname{coker}(\iota_Y)$. Hence, $\tau_X = \tau_Y \circ \iota_Y \circ \iota_X\inv$ is the inclusion of the $H_{2l+1}(V_X^c)$ summand into $H_{2l+1}(V_X) \cong H_{2l+1}(V_X^c) \oplus \left( H_{2l+1}(V_Y) / H_{2l+1}(V_Y') \right)$. Thus, $\tau_X$ is injective and Theorem \ref{thm:MV_non_critical} implies that
	\begin{equation*}
		\operatorname{coker}(\tau_X) = H_{2l+1}(V_X) / \operatorname{Im}(\tau_X) \cong H_{2l+1}(V_Y) / H_{2l+1}(V_Y') \cong \operatorname{coker}(\iota_Y)
	\end{equation*}
	is free.\\
	\indent In the second case of Theorem \ref{thm:MV_non_critical}, we proved that $\tau_Y:H_{2l+1}(V_Y) \to H_{2l+1}(V_X)$ is an isomorphism and that $\operatorname{Im}(\iota) = H_{2l+1}(V_X^c)$. In particular, $\ker(\tau_Y) = 0$, so exactness in Equation (\ref{eq:MV_non_critical}) gives
	\begin{equation*}
		\ker(\tau_X) \cong \ker(\tau_X - \tau_Y) = \operatorname{Im}(\iota) \cong H_{2l+1}(V_X^c).
	\end{equation*}
	Hence, $\tau_X$ is the 0 map.
\end{proof}

Now assume $\wf(V_Y') = \frac{l}{2l+1}$ for some $l \in Z$. Unlike the non-critical regime, $H_{2l}(V_Y')$ usually has more than one generator.

\begin{theorem}
	\label{thm:MV_critical}
	Suppose $\wf(V_Y') = \frac{l}{2l+1}$ for some $l \in \Z$. Let $K$ and $R$ be the kernel and image, respectively, of the map $H_{2l}(V_Y') \to H_{2l}(V_X^c) \oplus H_{2l}(V_Y)$. If both $K$ and $R$ are free, then
	\begin{align*}
		H_{2l}(V_X) &\cong \left(H_{2l}(V_X^c) \oplus H_{2l}(V_Y) \right)/R; \\
		H_{2l+1}(V_X) &\cong H_{2l+1}(V_X^c) \oplus H_{2l+1}(V_Y) \oplus K; \\
		\widetilde{H}_k(V_X) &\cong \widetilde{H}_k(V_X^c) \oplus \widetilde{H}_k(V_Y) \text{ for } k \neq 2l, 2l+1.
	\end{align*}
\end{theorem}
\begin{proof}
	By Theorem \ref{thm:aa17}, $\widetilde{H}_{k}(V_Y')$ is non-zero only when $k=2l$. Then Proposition \ref{prop:MV} gives $\widetilde{H}_{k}(V_X) \cong \widetilde{H}_{k}(V_X^c) \oplus \widetilde{H}_{k}(V_Y)$ for $k \neq 2l, 2l+1$ and
	\begin{align*}
		0 &\to H_{2l+1}(V_X^c) \oplus H_{2l+1}(V_Y) \to H_{2l+1}(V_X) \\
		&\xrightarrow{\partial_*} H_{2l}(V_Y') \xrightarrow{\iota} H_{2l}(V_X^c) \oplus H_{2l}(V_Y) \to H_{2l}(V_X) \to 0.
	\end{align*}
	Since $R$ is free, the short exact sequence $0 \to K \to H_{2l}(V_Y') \xrightarrow{\iota} R \to 0$ splits, so $H_{2l}(V_Y') \cong K \oplus R$. This allows us to separate the sequence above into
	\begin{align*}
		0 &\to H_{2l+1}(V_X^c) \oplus H_{2l+1}(V_Y) \to H_{2l+1}(V_X) \to K \to 0, \text{ and}\\
		0 &\to R \hookrightarrow H_{2l}(V_X^c) \oplus H_{2l}(V_Y) \to H_{2l}(V_X) \to 0.
	\end{align*}
	The formulas for $H_{2l}(V_X)$ and $H_{2l+1}(V_X)$ follow from the first isomorphism theorem and the fact that $K$ is free, respectively.
\end{proof}

If we use homology with field coefficients, Theorems \ref{thm:MV_non_critical} and \ref{thm:MV_critical} can be used to recursively compute $H_*(V_X)$. The kernel, cokernel, and image of a linear map between finite dimensional vector spaces are themselves vector spaces and, with field coefficients, homology groups are vector spaces. If $Y_1$ is the support of $\operatorname{St}(E_X)$, we can use the results of this section to compute $H_*(V_{X})$ in terms of $H_*(V_Y)$. If $Y_1$ has a monotone metric, we can use Corollary \ref{cor:VR_circular_monotone} to find $H_*(V_{Y_1})$. Otherwise, we find $Y_2 \subset Y_1$ as above and iterate the procedure. This procedure works well when $(X, d_X)$ is circular decomposable but not monotone, and $E_X$ is small.

\begin{example}
	\label{ex:recursive_space}
	Let $X = \{1, \dots, 7\}$ and consider the circular decomposable metric $d_X$ whose isolation indices and distance matrix are shown in Table \ref{tab:recursive_space}. Given $12 < r < 13$, the 1-skeleton of $V_X = \vr_r(X)$ is shown in Figure \ref{fig:recursive_space}. We now compute the complexes $V_X^c$, $V_Y$ and $V_Y'$ defined in Section \ref{sec:cyclic_component}. Note that $M_r(1) = 2$, $d_{14} < r < d_{15}$, and that $E_X = \big\{ \{1,4\} \big\}$. Since $1$ and $4$ have no common neighbors in $\vr_r(X)$, $\operatorname{St}(E_X) = E_X$ and $Y = \{1,4\}$. Hence, $V_Y$ consists of the edge between $1$ and $4$. Additionally, $G_X^c$ equals the 1-skeleton of $V_X$ minus the edge $\{1,4\}$, so $V_X^c = \Cl(G_X^c) = \vr_r(X) \setminus \big\{ \{1,4\} \big\}$. Then $V_Y'$ consists of the isolated points $1$ and $4$.\\
	\indent According to \cite[Definition 3.7]{aa17}, $\wf(V_Y') = 0$. Now we use Theorem \ref{thm:MV_critical} with $l = 0$ to compute $H_k(V_X)$. Note that $H_0(V_Y') = \Z^2$, $H_0(V_X^c) = \Z$ and $H_0(V_Y) = \Z$. Then the map $H_0(V_Y') \to H_0(V_X^c) \oplus H_0(V_Y)$ induced by the inclusions $V_Y' \subset V_X^c$ and $V_Y' \subset V_Y$ is given by
	\begin{align*}
		\Z^2 &\to \Z \oplus \Z\\
		(x,y) &\mapsto (x+y, x+y).
	\end{align*}
	The kernel $K$ of this map is the set generated by $(1,-1)$ and its cokernel is the quotient $\Z \oplus \Z / \langle (1,1) \rangle \cong \Z \oplus 0$. Additionally, we can show that $V_X^c$ is homotopy equivalent to $\Sp^1$ by contracting the 2-simplex $\{5,6,7\}$ onto the 1-skeleton of $V_X^c$. Hence, $H_1(V_X^c) = \Z$, $H_k(V_X^c) = 0$ for $k \geq 2$ and, since $V_Y$ is a single edge, $H_k(V_Y) = 0$ for $k \geq 1$. Then by Theorem \ref{thm:MV_critical},
	\begin{itemize}
		\item $H_{0}(V_X) \cong \left(H_{0}(V_X^c) \oplus H_{0}(V_Y) \right)/R \cong \Z \oplus \Z / (\Z \oplus 0) \cong \Z$,
		\item $H_{1}(V_X) \cong H_{1}(V_X^c) \oplus H_{1}(V_Y) \oplus K \cong \Z \oplus 0 \oplus \Z = \Z^2$,
		\item $H_k(V_X) \cong H_k(V_X^c) \oplus H_k(V_Y) \cong 0$ for $k > 1$.
	\end{itemize}
	
	\begin{figure}[h]
		\centering
		\includegraphics[scale=0.50]{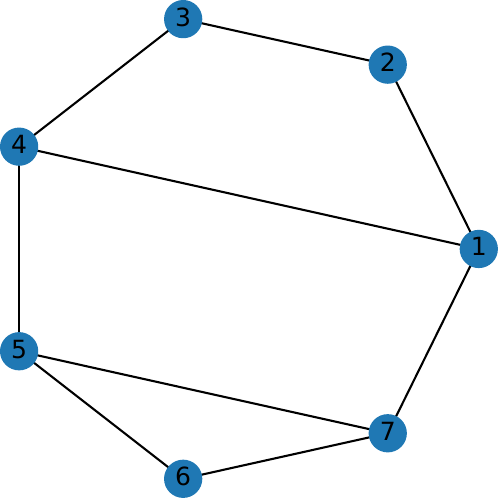}
		\caption{Given $12 < r <13$, the graph shown is the 1-skeleton of $\vr_r(X)$ where $X$ is the metric space whose distance matrix is shown in the right panel of Table \ref{tab:recursive_space}. $\vr_r(X)$ is not cyclic because it contains the edge $\{1,4\}$ despite not having all edges between the points $\{1,2,3,4\}$ or all edges between $\{4,5,6,7,1\}$. In particular, $E_X = \big\{ \{1,4\} \big\}$.}
		\label{fig:recursive_space}
	\end{figure}
	
	\begin{table}[h]
		\begin{minipage}{0.48\linewidth}
			\centering
			\begin{tabular}{l|rrrrrrr}
				& 1 & 2 & 3 & 4 & 5 & 6 & 7 \\
				\hline
				1 & 0 & 1 & 1 & 1 & 2 & 2 & 1 \\
				2 & 1 & 0 & 2 & 3 & 1 & 1 & 1 \\
				3 & 1 & 2 & 0 & 2 & 1 & 1 & 1 \\
				4 & 1 & 3 & 2 & 0 & 1 & 1 & 1 \\
				5 & 2 & 1 & 1 & 1 & 0 & 1 & 1 \\
				6 & 2 & 1 & 1 & 1 & 1 & 0 & 2 \\
				7 & 1 & 1 & 1 & 1 & 1 & 2 & 0 \\
			\end{tabular}
		\end{minipage}
		\hfill
		\begin{minipage}{0.48\linewidth}
			\centering
			\begin{tabular}{l|rrrrrrr}
				& 1 & 2 & 3 & 4 & 5 & 6 & 7 \\
				\hline
				1 & 0 & 9 & \textbf{13} & 12 & \textbf{13} & \textbf{13} & 8 \\ \cline{5-5} \cline{7-7}
				2 & 9 & 0 & 8 & \textbf{13} & \textbf{16} & \textbf{18} & \textbf{15} \\
				3 & \textbf{13} & 8 & 0 & 9 & \textbf{14} & \textbf{18} & \textbf{17} \\ \cline{3-3} \cline{8-8}
				4 & 12 & \textbf{13} & 9 & 0 & 7 & \textbf{13} & \textbf{14} \\ \cline{2-2}
				5 & \textbf{13} & \textbf{16} & \textbf{14} & 7 & 0 & 8 & 11 \\
				6 & \textbf{13} & \textbf{18} & \textbf{18} & \textbf{13} & 8 & 0 & 7 \\ \cline{4-4}
				7 & 8 & \textbf{15} & \textbf{17} & \textbf{14} & 11 & 7 & 0 \\ \cline{6-6}
			\end{tabular}
		\end{minipage}
		
		\caption{\textbf{Left:} Isolation indices of a circular decomposable metric on $\{1, \dots, 7\}$ arranged in a matrix $M$ so that $M_{ij} = \alpha_{ij}$. \textbf{Right:} Distance matrix of the circular decomposable metric on $\{1, \dots, 7\}$ produced by the isolation indices on the left. Entries larger than $12$ are in boldface and the entries $(\sigma(a), a)$ are underlined for every $a \in X$.}
		\label{tab:recursive_space}
	\end{table}
\end{example}

As mentioned at the start of Section \ref{sec:recursive_computation}, Theorems \ref{thm:MV_non_critical} and \ref{thm:MV_critical} yield non-trivial results only when $Y \subsetneq X$. Unfortunately, there are circular decomposable metrics where $Y = X$, as the next example shows.
\begin{example}
	\label{ex:non_recursive_metrics}
	\indent Let $X := \{1, \dots, 12\}$. Table \ref{tab:isolation_indices} has the matrix of isolation indices of a circular decomposable metric $d_X$ on $X$, and Table \ref{tab:non_recursive_metric} shows the distance matrix of $d_X$. Consider $\vr_r(X)$ for $r = 45$. Note that $M_r(a) = a+3$ for odd $a$, $M_r(a) = a+4$ for even $a$, and that $E_X$ consists of the edges $\{a, a+6\}$ for even $a$ (all additions are done mod 12). If $a$ is even, note that $d_{a+3,a} = d_{a+3,a+6} = 39$, so the set $\{a, a+3, a+6\}$ has diameter $44$, and thus, is a simplex of $\vr_r(X)$. In particular, $a+3$ belongs to the closed star of $\{a,a+6\}$. As $a$ ranges over all even numbers in $X$, $a+3$ ranges over all odds, so every element of $X$ belongs to $\operatorname{St}(E_X)$. Hence, $Y = X$.

	\begin{table}[h]
		\centering
		\begin{tabular}{l|rrrrrrrrrrrr}
			& 1 & 2 & 3 & 4 & 5 & 6 & 7 & 8 & 9 & 10 & 11 & 12\\
			\hline
			1 & 0 & 1 & 1 & 5 & 1 & 1 & 1 & 1 & 1 & 1 & 1 & 1 \\
			2 & 1 & 0 & 1 & 1 & 1 & 1 & 1 & 1 & 1 & 1 & 5 & 1 \\
			3 & 1 & 1 & 0 & 1 & 1 & 5 & 1 & 1 & 1 & 1 & 1 & 1 \\
			4 & 5 & 1 & 1 & 0 & 1 & 1 & 1 & 1 & 1 & 1 & 1 & 1 \\
			5 & 1 & 1 & 1 & 1 & 0 & 1 & 1 & 5 & 1 & 1 & 1 & 1 \\
			6 & 1 & 1 & 5 & 1 & 1 & 0 & 1 & 1 & 1 & 1 & 1 & 1 \\
			7 & 1 & 1 & 1 & 1 & 1 & 1 & 0 & 1 & 1 & 5 & 1 & 1 \\
			8 & 1 & 1 & 1 & 1 & 5 & 1 & 1 & 0 & 1 & 1 & 1 & 1 \\
			9 & 1 & 1 & 1 & 1 & 1 & 1 & 1 & 1 & 0 & 1 & 1 & 5 \\
			10 & 1 & 1 & 1 & 1 & 1 & 1 & 5 & 1 & 1 & 0 & 1 & 1 \\
			11 & 1 & 5 & 1 & 1 & 1 & 1 & 1 & 1 & 1 & 1 & 0 & 1 \\
			12 & 1 & 1 & 1 & 1 & 1 & 1 & 1 & 1 & 5 & 1 & 1 & 0 \\
		\end{tabular}
		\caption{Isolation indices of a circular decomposable metric on $\{1, \dots, 12\}$ arranged in a matrix $M$ so that $M_{ij} = \alpha_{ij}$.}
		\label{tab:isolation_indices}
	\end{table}
	
	\begin{table}[h]
		\centering
		\begin{tabular}{l|rrrrrrrrrrrr}
			& 1 & 2 & 3 & 4 & 5 & 6 & 7 & 8 & 9 & 10 & 11 & 12\\
			\hline
			1 & 0 & 15 & 28 & 39 & \textbf{48} & \textbf{47} & \textbf{52} & \textbf{47} & \textbf{48} & 39 & 28 & 15 \\ \cline{7-7}
			2 & 15 & 0 & 15 & 28 & 39 & 40 & \textbf{47} & 44 & \textbf{47} & 40 & 39 & 28 \\ \cline{12-12}
			3 & 28 & 15 & 0 & 15 & 28 & 39 & \textbf{48} & \textbf{47} & \textbf{52} & \textbf{47} & \textbf{48} & 39 \\ \cline{9-9}
			4 & 39 & 28 & 15 & 0 & 15 & 28 & 39 & 40 & \textbf{47} & 44 & \textbf{47} & 40 \\ \cline{2-2}
			5 & \textbf{48} & 39 & 28 & 15 & 0 & 15 & 28 & 39 & \textbf{48} & \textbf{47} & \textbf{52} & \textbf{47} \\ \cline{11-11}
			6 & \textbf{47} & 40 & 39 & 28 & 15 & 0 & 15 & 28 & 39 & 40 & \textbf{47} & 44 \\ \cline{4-4}
			7 & \textbf{52} & \textbf{47} & \textbf{48} & 39 & 28 & 15 & 0 & 15 & 28 & 39 & \textbf{48} & \textbf{47} \\ \cline{13-13}
			8 & \textbf{47} & 44 & \textbf{47} & 40 & 39 & 28 & 15 & 0 & 15 & 28 & 39 & 40 \\ \cline{6-6}
			9 & \textbf{48} & \textbf{47} & \textbf{52} & \textbf{47} & \textbf{48} & 39 & 28 & 15 & 0 & 15 & 28 & 39 \\ \cline{3-3}
			10 & 39 & 40 & \textbf{47} & 44 & \textbf{47} & 40 & 39 & 28 & 15 & 0 & 15 & 28 \\ \cline{8-8}
			11 & 28 & 39 & \textbf{48} & \textbf{47} & \textbf{52} & \textbf{47} & \textbf{48} & 39 & 28 & 15 & 0 & 15 \\ \cline{5-5}
			12 & 15 & 28 & 39 & 40 & \textbf{47} & 44 & \textbf{47} & 40 & 39 & 28 & 15 & 0 \\ \cline{10-10}
		\end{tabular}
		\caption{Distance matrix of the circular decomposable metric on $\{1, \dots, 12\}$ whose isolation indices are given in Table \ref{tab:isolation_indices}. Entries larger than $45$ are in boldface and the entries $(\sigma(a), a)$ are underlined for every $a \in X$.}
		\label{tab:non_recursive_metric}
	\end{table}
\end{example}

\begin{remark}
	\label{rmk:non_recursive_metrics}
	In light of Example \ref{ex:non_recursive_metrics}, we leave it as open questions to find conditions on the isolation indices $\alpha_{ij}$ so that $Y \subsetneq X$, and to compute the homology of $\vr_r(X)$ when these conditions fail. We also leave it to future research to find conditions on the isolation indices $\alpha_{ij}$ that ensure that the recursion described before Example \ref{ex:recursive_space} eventually reaches some $Y_m$ whose metric is monotone.
\end{remark} 	\section{Block decomposition of $\vr_r(X)$.}
\label{sec:blocks_of_X}
\indent The next section is motivated by the block decomposition of $T(X, d_{\cS, \alpha})$ for a totally decomposable metric $d_{\cS, \alpha}$ on $X$. As it turns out, totally decomposable spaces are not the only spaces whose tight span has a block decomposition. For example, the tight span of a metric wedge of two split-prime spaces $(X, d_X)$ and $(Y, d_Y)$ has at least two blocks because $T(X \vee Y, d_X \vee d_Y) = T(X, d_X) \vee T(Y, d_Y)$ by Lemma \ref{lemma:wedge_tight_span}. Hence, we will use the block decomposition of $T(X, d_X)$ to induce a decomposition of $\vr_r(X)$ into VR complexes of subsets of $X$.\\
\indent Our strategy is based on Theorem \ref{thm:VR-tight-span}. Since $\vr_{2r}(X)$ is homotopy equivalent to $B_{r}(X, T(X, d_X))$, $B_{r}(X, T(X, d_X))$ decomposes as the union of the intersections of $B_{r}(X, T(X, d_X))$ with each block of $T(X, d_X)$. The main work in this section lies in showing that each intersection is homotopy equivalent to the VR complex of a certain subset of $X$. These subsets do not form a partition of $X$ though, so we also have to figure out how these complexes paste together to form $\vr_r(X)$.

\begin{note}
	Unless stated otherwise, $d_\infty$ denotes the $L^\infty$ metric on $T(X, d_X)$ for the rest of this section.
\end{note}

\subsection{Subsets of $X$ induced by blocks of $E$}
\label{sec:}
\indent Let $h:(X, d_X) \hookrightarrow (E, d_E)$ be an isometric embedding of a finite metric space into an injective polytopal complex. Let $\cK$ be a finite cover of $E$ by connected subcomplexes such that any pair $K, K' \in \cK$ is disjoint or intersects on a cut-vertex (hence, each $K \in \cK$ is a connected union of blocks of $E$). We call any such $\cK$ a \define{block cover} of $E$. Let $\cut(\cK)$ be the set of all intersection points $c \in K \cap K'$ for $K, K' \in \cK$.

\begin{remark}
	For the rest of the section, we fix $h:(X, d_X) \hookrightarrow (E, d_E)$ as an isometric embedding of a finite metric space into an injective polytopal complex and a block cover $\cK$.
\end{remark}

\begin{defn}
	\label{def:block_metric}
	For any $K \in \cK$, define
	\begin{align*}
		X_K &:= \{x \in X : h(x) \in K\} \\
		CX_K &:= \{c \in \cut(\cK) \cap K : c \text{ separates } X_K \text { from some } x \in X \setminus X_K\} \\
		X_K' &:= X_K \cup CX_K.
	\end{align*}
	Given $c \in CX_K$, set $x_c := x$ if $c = h(x)$ for some $x \in X$. Otherwise, let $x_c$ be any point that is separated from $K$ by $c$ and minimizes $d_E(h(x_c), c)$. Define
	\begin{equation*}
		\overline{X}_K := X_K \cup \{x_c : c \in CX_K \}.
	\end{equation*}
\end{defn}

\indent Note that a point $x \in X$ may appear in several $\overline{X}_K$ if $x = x_c$ for some $c \in K$ such that $h(x) \notin K$. Even though the choice of $x_c$ is not unique, the isometry type of $\overline{X}_K$ is.
\begin{lemma}
	\label{lemma:VR_block_well_defined}
	For all $K \in \cK$, the isometry type of $\overline{X}_{K}$ is independent of the choice of $x_c$.
\end{lemma}
\begin{proof}
	We only have to verify that the distances from $x_c$ to $X_{K}$ and to any other point $x_{c'}$ depend only on $c$ and $c'$ and not on the choice of $x_c$ and $x_{c'}$. Let $c \in CX_K$ and choose any $\phi \in K$. If $c = h(x_c)$, then there is only one choice for $x_c$, so there is nothing to check. Otherwise, $c$ separates $h(x_c)$ and $\phi$, so any path between $h(x_c)$ and $\phi$ must contain $c$. In particular, this holds for the shortest path, and since $E$ is geodesic, we obtain
	\begin{equation*}
		d_E(h(x_c), \phi) = d_E(h(x_c), c) + d_E(c, \phi).
	\end{equation*}
	In particular, when $\phi = h(x')$ for $x' \in X_K$ we obtain
	\begin{equation*}
		d_X(x_c, x') = d_E(h(x_c), h(x')) = d_E(h(x), c) + d_E(c, h(x')).
	\end{equation*}
	Since $x_c$ was chosen as any point that minimizes $d_E(h(x_c), c)$, the quantity $d_{X}(x_c, x')$ only depends on $c$ and not on $x_c$.\\
	\indent If we have $c' \in CX_K$, $c' \neq c$, then the shortest path between $c$ and $c'$ in $E$ must be contained in $K$. Otherwise, a path that connects two cut-vertices in $K$ by going outside of $K$ would induce a cycle in the tree $\bctree(E)$ (recall Lemma \ref{lemma:block_cut_graph_is_tree}). Applying the argument above to $h(x_c)$ and $\phi = c'$, and to $h(x_{c'})$ and $\phi = c$ shows that the shortest path between $h(x_c)$ and $h(x_{c'})$ must contain $c$ and $c'$, so that
	\begin{equation*}
		d_{X}(x_c, x_{c'}) = d_E(h(x_c), h(x_{c'})) = d_E(h(x_c), c) + d_E(c, c') + d_E(c', h(x_{c'})).
	\end{equation*}
	Once again, this quantity does not depend on the choice of $x_c$ and $x_{c'}$.
\end{proof}

\indent As alluded to at the start of the section, the intersection of $B_r(X; T(X, d_X))$ with a block of $T(X, d_X)$ is determined by a specific subset of $X$. We prove that $\overline{X}_K$ is that set for a general block cover $\cK$.
\begin{lemma}
	\label{lemma:union_inside_block}
	For all $K \in \cK$, $B_{r}(X; E) \cap K = B_r(\overline{X}_K; E) \cap K$. As a consequence,
	\begin{equation*}
		B_r(X; E) \simeq \bigcup_{K \in \cK} \left( B_{r}(\overline{X}_{K}; E) \cap K \right).
	\end{equation*}
\end{lemma}
\begin{proof}
	Let $c \in CX_K$ and let $X_c$ be the set of $x \in X$ such that that $c$ separates $h(x)$ from $K$. Define $\ell_c := d_E(h(x_c), c)$. For every $x \in X_c$, any path in $E$ from $h(x)$ to $p \in K$ has to contain $c$, so $d_E(h(x), p) = d_E(h(x), c) + d_E(c, p)$. Then
	\begin{equation*}
		B_r(h(x); E) \cap K =
		\begin{cases}
			B_{r - d_E(h(x), c)}(c; K) & \text{if } r > \ell_c \\
			\emptyset & \text{if } \leq \ell_c.
		\end{cases}
	\end{equation*}
	Since $x_c$ minimizes $d_E(h(x), c)$ among $x \in X_c$, the set $B_r(h(x); E) \cap K$ is largest when $x = x_c$. Hence, $B_r(X_c; E) \cap K$ equals $B_{r - \ell_c}(c; K)$ if $r > \ell_c$ and is empty if $r \leq \ell_c$. Putting everything together,
	\begin{align*}
		B_{r}(X; E) \cap K
		&= \left[ B_{r}(X_K; E) \cap K \right] \cup \bigcup_{c \in CX_K} \left[ B_{r}(X_c; E) \cap K \right] \\
		&= \left[ B_{r}(X_K; E) \cap K \right] \cup \bigcup_{c \in CX_K} \left[ B_{r}(h(x_c); E) \cap K \right] \\
		&= B_r(\overline{X}_K; E) \cap K,
	\end{align*}
	and thus,
	\begin{equation*}
		B_r(X; E)
		= \bigcup_{K \in \cK} \left( B_{r}(X; E) \cap K \right)
		= \bigcup_{K \in \cK} \left( B_{r}(\overline{X}_{K}; E) \cap K \right).
	\end{equation*}
\end{proof}

\subsection{Properties of blocks of $T(X, d_X)$}
In this section, we study the elements of $\cK$ as metric spaces in their own right. In order to use Theorem \ref{thm:VR-tight-span}, we show that every $B \in \blocks(E)$, and thus every $K \in \cK$, is an injective space. In the case that $\cK = \blocks(E)$, we show that each block $B$ is the tight span of $X_B'$, and construct the tight span of $\overline{X}_B$ by attaching a line segment to $B$ for every $c \in CX_B$.
\begin{lemma}
	\label{lemma:block_is_injective}
	Let $(E, d_E)$ be an injective polytopal complex. Every $B \in \blocks(E)$ with the restriction of $d_E$ is a hyperconvex space (equiv. injective).
\end{lemma}
\begin{proof}
	Let $B$ be a block of $E$. Choose any number of points $p_i \in B$ and lengths $r_i > 0$ such that $d_E(p_i, p_j) \leq r_i + r_j$. To prove that $B$ is hyperconvex, we need to find $q \in B$ such that $d_E(p_i, q) \leq r_i$ for all $i$. Since $E$ is hyperconvex, there exists a point $q \in E$ that satisfies $d_E(p_i, q) \leq r_i$. If $q \in B$, we are done. If not, we claim that there exists a unique cut-vertex $c \in B$ such that $d_E(p_i, c) \leq r_i$ for all $i$. We know that $E$ has at least one cut-vertex $c \in B$ that separates $q$ and $B$ because a block is a maximal subcomplex of $E$ with no intrinsic cut-vertices. As a consequence, any path between $q$ and a point of $B$ must contain $c$, and, in fact, $c$ is the unique cut-vertex that separates $q$ and $B$ (removing a different cut-vertex from $B$ would not disconnect the shortest path between $q$ and $c \in B$). Thus, since $E$ is geodesic, there exist geodesics $\gamma_i$ from $p_i$ to $q$ of length no greater than $r_i$. Since these geodesics contain $c$, $d_E(p_i, c) < d_E(p_i, q) \leq r_i$ for all $i$. Thus, $B$ is hyperconvex.
\end{proof}

The above Lemma generalizes to the union of any number of blocks and, in particular, to the elements of a block cover.
\begin{corollary}
	\label{cor:block_is_injective}
	Let $(E, d_E)$ be an injective polytopal complex. Whenever a union of blocks of $E$ is connected, it is also injective. In particular, every element of a block cover of $E$ is injective.
\end{corollary}
\begin{proof}
	We proceed by induction on the number of blocks. A single block is injective by Lemma \ref{lemma:block_is_injective}. If we have a connected union of $n$ blocks $B_1, \dots, B_n$, label them so that $B_n$ is a leaf of $\bctree(B_1 \cup \cdots \cup B_n)$. Then $B_1 \cup \cdots \cup B_{n-1}$ is connected, hence injective by induction hypothesis. By assumption, $B_n$ has degree 1 on $\bctree(B_1 \cup \cdots \cup B_n)$, so there is a unique cut vertex $c \in B_n \cap (B_1 \cup \cdots \cup B_{n-1})$. Since $E$ is injective, any geodesic in $E$ between points in $B_n$ and $B_1 \cup \cdots \cup B_{n-1}$ must pass through $c$. Hence, $(B_1 \cup \cdots \cup B_{n-1}) \cup B_n$ is isometric to the metric wedge $(B_1 \cup \cdots \cup B_{n-1}) \vee_{c} B_n$, and Lemma \ref{lemma:wedge_tight_span} implies that $(B_1 \cup \cdots \cup B_{n-1}) \cup B_n$ is injective.
\end{proof}

The previous Lemma is all we need for the proofs in Section \ref{sec:block-VR-proof}. In the rest of the section, we strengthen the result in the case of $E = T(X, d_X)$. The first of these results characterizes any $K \in \cK$ as the tight span of a metric space, but not of a subset of $X$ or even of $X_K$. We have to include the cut-vertices in $CX_K$.
\begin{lemma}
	\label{lemma:block_is_tight_span}
	Let $(X, d_X)$ be a finite metric space. Suppose that $E = T(X, d_X)$ and fix $K \in \cK$. Then $K$ is isometric to $T(X_K', d_E|_{X_K' \times X_K'})$.
\end{lemma}
\begin{proof}
	By Lemma \ref{lemma:block_is_injective}, $K$ is an injective space and $X_K' \hookrightarrow K$ by definition, so we only need to verify the minimality of $K$, i.e. if $Z \subset K$ is a closed, injective subspace of $K$ that contains $X_K'$, then $Z = K$. If $|\cK|=1$, then $K = T(X, d_X)$ and $X \subset X_K'$. In particular, $Z$ is now a closed, injective subspace of $T(X, d_X)$ that contains $X$, so by minimality of $T(X, d_X)$,
	\begin{equation*}
		T(X, d_X)  = Z \subset K \subset T(X, d_X).
	\end{equation*}
	Thus, $Z = K$.\\
	\indent If $|\cK| > 1$, let $T_c$ be the union of all blocks of $T(X, d_X)$ that are separated from $K$ by $c \in CX_K$. By Lemma \ref{cor:block_is_injective}, $T_c$ is injective. Let $\overline{Z}$ be the metric wedge obtained by pasting $T_c$ to $Z$ at $c$ for all $c \in K \cap \cut(\cK)$. Since $Z$ is injective by assumption, $\overline{Z}$ is injective by Lemma \ref{lemma:wedge_tight_span}. Moreover, $X \hookrightarrow \overline{Z}$, so by minimality of the tight span, $T(X, d_X) \hookrightarrow \overline{Z} \hookrightarrow K \cup \bigcup_{i=1}^n T_i = T(X, d_X)$. This can only happen if $Z=K$.
\end{proof}

The main technical step comes next.
\begin{lemma}
	\label{lemma:tight_span_edges}
	Fix $\ell \geq 0$. Let $(X, d_X)$ and $(Y, d_Y)$ be metric spaces and fix $x_0 \in X$. Let $f:X \to Y$ be a bijection such that
	\begin{equation*}
		d_Y(f(x), f(x')) =
		\begin{cases}
			d_X(x,x'),				 &  x, x' \neq x_0,\\
			d_X(x,x_0) + \ell,	& x \neq x_0, x' = x_0,\\
			0,								& x=x'=x_0.
		\end{cases}
	\end{equation*}
	Then $T(Y, d_Y)$ is isometric to the metric gluing of $T(X, d_X)$ and a line segment of length $\ell$ at $x_0$ (or to $T(X, d_X)$ if $\ell = 0$).
\end{lemma}
\begin{proof}
	Let $y_0 = f(x_0)$. Throughout this proof, $d_{\infty,Z}$ denotes the $L^\infty$ metric in $T(Z, d_Z)$ for $Z = X, Y$. Our strategy will be to find a geodesic $\gamma_t \in T(Y, d_Y)$ with $t \in [0,\ell]$ such that $\gamma_0 = h_{y_0}$ and $d_{\infty,Y}(\gamma_\ell, h_{f(x)}) = d_X(x, x_0)$ for every $x \in X$. Assuming such a $\gamma$, the function $(X, d_X) \to T(Y, d_Y)$ defined by sending $x \neq x_0$ to $h_{f(x)}$ and $x_0$ to $\gamma_\ell$ is an isometric embedding thanks to the fact that $d_X$ and $d_Y \circ (f, f)$ coincide on $X \setminus \{x_0\}$. The minimality of $T(X, d_X)$ then implies $(X, d_X) \hookrightarrow T(X, d_X) \hookrightarrow T(Y, d_Y)$. We will then show that $\gamma_\ell$ is a cut-point of $T(Y, d_Y)$ so that $T(Y, d_Y) = T(X, d_X) \cup \gamma$ (viewing $T(X, d_X)$ as a subset of $T(Y, d_Y)$). This implies that $T(X, d_X) \cup \gamma \hookrightarrow T(Y, d_Y)$ is isometric to the metric wedge $T(X, d_X) \vee_{\gamma_\ell} \gamma$, so $T(Y, d_Y) \cong T(X, d_X) \vee_{\gamma_0} \gamma$.\\
	\indent We start by constructing $\gamma$. Let $\gamma_t(y_0) := t$ and $\gamma_t(y) := d_Y(y,y_0) - t$ for $y \neq y_0$. Assuming that $\gamma_t \in T(Y, d_Y)$, we have $d_{\infty,Y}(\gamma_t, \gamma_s) = |t-s|$, i.e. $\gamma$ is a geodesic. We now verify that $\gamma_t$ satisfies the two conditions of Definition \ref{def:tight_span}. Note that:
	\begin{itemize}
		\item $\gamma_t(y_0) + \gamma_t(y_0) = 2t \geq 0 = d_Y(y_0, y_0)$.
		\item If $y \neq y_0$, $\gamma_t(y) + \gamma_t(y_0) = (d_Y(y, y_0) - t) + t = d_Y(y, y_0)$.
		\item If $y, y' \neq y_0$,
		\begin{align*}
			\gamma_t(y) + \gamma_t(y')
			&= d_Y(y, y_0) + d_Y(y', y_0) - 2t
			\geq d_Y(y, y_0) + d_Y(y', y_0) - 2\ell \\
			& = d_X(f\inv(y), x_0) + d_X(f\inv(y'), x_0)
			\geq d_X(f\inv(y), f\inv(y'))\\
			&= d_Y(y,y').
		\end{align*}
	\end{itemize}
	In short, $\gamma_t(y) + \gamma_t(y') \geq d_Y(y, y')$ for all $y, y' \in Y$. Note that equality holds when $y \neq y_0$ and $y' = y_0$, so the supremum $\gamma_t(y) = \sup_{y' \in Y} \left( d_Y(y', y) - \gamma_t(y') \right)$ is realized with $y' = y_0$ when $y \neq y_0$ and with $y' \neq y_0$ when $y = y_0$.\\
	\indent We now construct the embedding $T(X, d_X) \hookrightarrow T(Y, d_Y)$. By Lemma \ref{lemma:eval_tight_span}, $d_{\infty,Y}(\gamma_\ell, h_{f(x)}) = \gamma_t(f(x)) = d_Y(f(x), y_0) - \ell = d_X(x, x_0)$. Since $f$ restricted to $X \setminus \{x_0\}$ is an isometry, the map $\overline{f}:X \to T(Y, d_Y)$ given by $\overline{f}(x) = h_{f(x)}$ and $\overline{f}(x_0) = \gamma_\ell$ is an isometric embedding $X \hookrightarrow T(Y, d_Y)$. By minimality of $T(X, d_X)$, $\overline{f}$ must factor through another isometric embedding $H:T(X, d_X) \to T(Y, d_Y)$ that satisfies $H(h_x) = \overline{f}(x)$ for every $x \in X$.\\
	\indent Next, we show that $\gamma_\ell$ is a cut point that separates $\gamma$ and $H(T(X, d_X))$. The key observation is that $\gamma_\ell(y) + \gamma_\ell(y_0) = d_Y(y, y_0)$ for all $y \in Y \setminus \{y_0\}$, so that, in the notation of \cite{cut-points-metric-spaces}, $\gamma_\ell$ is a \textit{virtual} cut point. Then Theorem 1 of \cite{cut-points-metric-spaces} implies that $\gamma_\ell$ is a \textit{topological} cut point, i.e. that $T(Y, d_Y) \setminus \{\gamma_\ell\}$ is disconnected. The authors also define the set
	\begin{equation*}
		O_{\gamma_\ell}(A) := \{g \in T(Y, d_Y) \setminus \{\gamma_\ell\} : \gamma_\ell(y) < g(y) \text{ for all } y \in Y \setminus A\}
	\end{equation*}
	for any $A \subset Y$ and prove that the connected components of $T(Y, d_Y) \setminus \{\gamma_\ell\}$ are the open sets $O_1 := O_{\gamma_\ell}(\{y_0\})$ and $O_2 := O_{\gamma_\ell}(Y \setminus \{y_0\})$. We claim that the sets $O_i^+ := O_i \cup \{\gamma_\ell\}$ satisfy $\gamma \subset O_1^+$ and $H(T(X, d_X)) \subset O_2^+$. First, $O_i^+$ is a sub-polytopal complex of $T(Y, d_Y)$ that, since $\gamma_\ell$ is a cut-vertex, equals a union of blocks of $T(Y, d_Y)$. By Corollary \ref{cor:block_is_injective}, $O_i^+$ is injective. Next, we have $\gamma_t \in O_1$ for all $0 \leq t < \ell$ because $\gamma_\ell(y) < \gamma_t(y)$ for any $y \in Y \setminus \{y_0\}$. Furthermore, for any $x \in X \setminus \{x_0\}$, $h_{f(x)}(y_0) = d_Y(f(x), y_0) = d_X(x, x_0) + \ell > \gamma_\ell(y_0)$, so $\overline{f}(x) \in O_2$ for all $x \neq x_0$. Since $\gamma_\ell = \overline{f}(x_0)$, we have $\gamma \subset O_1^+$ and $\overline{f}(X) \subset O_2^+$, so $H(T(X, d_X)) \subset O_2^+$ because $O_2^+$ is injective. Thus, $\gamma_\ell$ separates $\gamma$ and $H(T(X, d_X))$.\\
	\indent We are ready to finish the proof. Note that $H(T(X, d_X))$ contains $\overline{f}(X \setminus \{x_0\}) = \{h_{y} : y \in Y \setminus \{y_0\}\}$ and $\gamma$ contains $\gamma_0 = h_{y_0}$, so $H(T(X, d_X)) \cup \gamma$ contains an isometric copy of $(Y, d_Y)$. By the previous paragraph, $\gamma_\ell$ is a cut vertex so $H(T(X, d_X))$ and $\gamma$ intersect at $\gamma_\ell$. Since any path between points of $\gamma$ and $H(T(X, d_X))$ must pass through $\gamma_\ell$, $H(T(X, d_X)) \cup \gamma$ is isometric to $H(T(X, d_X)) \vee_{\gamma_\ell} \gamma$. Then by Lemma \ref{lemma:wedge_tight_span}, $H(T(X, d_X)) \cup \gamma$ is a closed, injective subset of $T(Y, d_Y)$, so minimality of the tight span yields $T(Y, d_Y) = H(T(X, d_X)) \cup \gamma \cong T(X, d_X) \vee \gamma$.
\end{proof}

We conclude the section by constructing a injective space that contains $\overline{X}_K$ by attaching line segments to $K$. Under the right conditions, this space will also be the tight span of $\overline{X}_K$.
\begin{lemma}
	\label{lemma:VR_block_tight_span}
	Fix $K \in \cK$. Let $T_K$ be the space formed by attaching a line segment of length $d_\infty(h(x_c), c)$ to $K$ at every $c \in CX_K$. Then $T_K$ is an injective space that contains an isometric copy of $\overline{X}_K$. If $E = T(X,d_X)$ in addition, then $T_K$ is isometric to the tight span of $\overline{X}_K$ equipped with the restriction of $d_X$.
\end{lemma}
\begin{proof}
	Let $\ell_c := d_E(c, h(x_c))$. By Lemma \ref{lemma:block_is_injective}, $K$ is injective and by Lemma \ref{lemma:wedge_tight_span}, pasting a line segment to $K$ at a point results in an injective space. Hence, $T_K$ is injective. Since $X_K' \hookrightarrow K$, $T_K$ contains an isometric copy of $X_K'$. Furthermore, $c \in CX_K$ separates $x_c$ from $K$, so
	\begin{equation*}
		d_X(x_c, x) = d_E(h(x_c), h(x)) = d_E(h(x_c), c) + d_E(c, h(x)) = \ell_c + d_E(c, h(x))
	\end{equation*}
	for all $x \in \overline{X}_K \setminus \{x_c\}$. Hence, the function $h_K:\overline{X}_K \to T_K$ that sends $x \in X_K$ to $h(x) \in K$ and $x_c$ to the endpoint of the edge of length $\ell_c$ pasted at $c$ is an isometric embedding.\\
	\indent Suppose now $E = T(X,d_X)$. Define a function $f:X_K' \to \overline{X}_K$ by $f(x) := x$ if $x \in X_K$ and $f(c) := x_c$ for $c \in X_K' \setminus X_K$. Since $K$ is the tight span of $X_K'$ by Lemma \ref{lemma:block_is_tight_span}, the fact that $T_K$ is the tight span of $\overline{X}_K$ now follows by repeatedly applying Lemma \ref{lemma:tight_span_edges} for every $c \in X_K' \setminus X_K$.
\end{proof}

\subsection{Decomposition of $\vr_r(X)$ induced by blocks of $T(X, d_X)$}
\label{sec:block-VR-proof}
In the previous sections, we formed a good understanding of $B_r(X; E) \cap K$ for every $K \in \cK$ and of $\overline{X}_K$ and its tight span. We now use those results to understand the relation between $\vr_{2r}(\overline{X}_K)$ with $K \subset E$. After that, we describe how to paste all $\vr_{2r}(\overline{X}_K)$ to form $\vr_{2r}(X)$.
\begin{lemma}
	\label{lemma:VR_inside_block}
	For every $K \in \cK$,
	\begin{equation*}
		\vr_{2r}(\overline{X}_K) \simeq \left[ B_r(X; E) \cap K \right] \sqcup \{x_c \in \overline{X}_K \setminus X_K: d_E(h(x_c), c) > r\}.
	\end{equation*}
\end{lemma}
\begin{proof}
	\indent Let $T_K = K \cup \bigcup_{c \in CX_K} \gamma_c$ be the injective space from Lemma \ref{lemma:VR_block_tight_span}, where $\gamma_c$ is the line segment joining the cut-vertex $c \in CX_K$ and $h(x_c)$. By \cite[Proposition 2.3]{osman-memoli},
	\begin{align*}
		\vr_{2r}(\overline{X}_K)
		\simeq B_r(\overline{X}_K; T_K)
		& = \left[ B_r(\overline{X}_K; T_K) \cap K \right] \cup \left[ \bigcup_{c \in CX_K} B_r(h(x_c); T_K) \cap \gamma_c \right] \\
		& = \left[ B_r(\overline{X}_K; T_K) \cap K \right] \cup \left[ \bigcup_{c \in CX_K, \ r > \ell_c} \gamma_c \right] \cup \left[ \bigcup_{c \in CX_K, \ r \leq \ell_c} B_r(h(x_c); \gamma_c) \right].
	\end{align*}
	Let's rewrite each term separately. In the first term, $B_r(\overline{X}_K; T_K) \cap K = B_r(\overline{X}_K; E) \cap K$. In the middle term, $\gamma_c \cap K = \{c\}$ whenever $r > \ell_c$, so we can deformation retract $\gamma_c$ onto $c$ without changing the homotopy type of $\vr_{2r}(\overline{X}_K)$. In the last term, $r \leq \ell_c$ implies $B_r(h(x_c); \gamma_c) \cap K = \emptyset$, and we deformation retract $B_r(h(x_c); \gamma_c)$ onto the point $h(x_c)$. Then by Lemma \ref{lemma:union_inside_block},
	\begin{align*}
		\vr_{2r}(\overline{X}_K)
		& \simeq \left[ B_r(\overline{X}_K; E) \cap K \right] \sqcup \{x_c \in \overline{X}_K \setminus X_K: r \leq \ell_c\} \\
		& = \left[ B_r(X; E) \cap K \right] \sqcup \{x_c \in \overline{X}_K \setminus X_K: r \leq \ell_c\}.
	\end{align*}
\end{proof}

Below, we need to be more specific with the points $x_c$, so we will denote $x_c \in \overline{X}_K$ as $x_{K,c}$ instead.
\begin{defn}
	\label{def:block_cut_graph_r}
	Let $\bctree_r(\cK)$ be the graph with vertex set $\cK \cup \cut(\cK)$ and edges $(c,K) \in \cut(\cK) \times \cK$ where $d_E(x_{K,c}, c) < r$.
\end{defn}

\begin{remark}[$\bctree_r(\cK)$ is a forest]
	\label{rmk:block_cut_graph_is_tree}
	Note that $\bctree_\infty(E; \cK)$ can be obtained from $\bctree(E)$ as follows. Select a labeling $\cK = \{K_1, \dots, K_n\}$. For every $c \in K_n \cap \cut(\cK)$, add a vertex $v_c$ to the interior of every edge $(c, B)$ with $B \subset K_i$ and $i \neq n$. Then, contract all vertices $B \in K_n$ and $c \in K_n \cap \cut(\cK)$ into a single vertex; label it as $K_n$. Repeat this procedure for all remaining $K_i$ by splitting any new edges of the form $(v_c, K_j)$ in analogous fashion. In the end, relabel $v_c$ as $c$. This graph is $\bctree_\infty(E; \cK)$. We then obtain $\bctree_r(\cK)$ by erasing any edge $(c,K) \in \cut(\cK) \times \cK$ where $x_{K,c}$ is isolated in $\vr_{2r}(X)$. None of the operations we used introduce cycles, so $\bctree_\infty(E; \cK)$ is a tree and $\bctree_r(\cK)$ is a forest.
\end{remark}

\indent Think of $\bctree_r(\cK)$ as an indexing category where an edge $(c,K)$ becomes an arrow $c \to K$. Define a functor $F_{r}: \bctree_r(\cK) \to \mathrm{Top}$ by
\begin{align*}
	c &\mapsto \{c\} \subset E\\
	K &\mapsto B_r(\overline{X}_K; T_K)
\end{align*}
where $T_K$ is the space from Lemma \ref{lemma:VR_block_tight_span}. Also, send each arrow $(c, K)$ to the inclusion map $c \mapsto x_{K,c} \in \overline{X}_K$. This functor encodes how the complexes $\vr_{2r}(\overline{X}_K)$ paste together to form $\vr_{2r}(X)$. For this reason, we view the next theorem as a decomposition of $\vr_{2r}(X)$ induced by the blocks of $E$. We don't call it a \textit{block} decomposition of $\vr_{2r}(X)$ because this VR complex might have cut-vertices and blocks of its own.\\
\indent We are ready to prove the main theorem of this section.
\begin{theorem}
	\label{thm:VR_wedge}
	Fix $r > 0$. Let $h:(X, d_X) \hookrightarrow (E, d_E)$ be an isometric embedding of a finite metric space into an injective polytopal complex, and let $\cK$ be a block cover of $E$. Then
	\begin{equation*}
		\vr_{2r}(X) \sqcup C_r \simeq \left( \bigsqcup_{K \in \cK} \vr_{2r}(\overline{X}_{K,r}) \right) / \sim
	\end{equation*}
	where $\sim$ identifies $x_{K,c} \in \vr_{2r}(\overline{X}_{K,r})$ with $x_{K',c} \in \vr_{2r}(\overline{X}_{K',r})$ if $K \cap K' = \{c\}$ and $C_r$ is the set of isolated vertices of $\bctree_r(\cK)$.
\end{theorem}
\begin{proof}
	Fix $c \in \cut(\cK)$. For any $K \in \cK$ with $c \in K$, let $\ell_{K,c} := d_E(h(x_{K,c}), c)$ and let $T_{K}$ be the space obtained by applying Lemma \ref{lemma:VR_block_tight_span} to $\overline{X}_{K}$. Recall that $T_{K}$ is obtained from $K \subset E$ by pasting a geodesic $\gamma_{K,c}$ of length $\ell_{K,c}$ that joins $h(x_{K,c})$ and $c$ for every $c \in K \cap \cut(\cK)$. Note that each $B_r(\overline{X}_K; T_K)$ satisfies one of the following properties:
	\begin{enumerate}
		\item\label{prop_1} $c \notin B_r(\overline{X}_K; T_K)$,
		\item\label{prop_2} $c \in B_r(\overline{X}_K; T_K)$ and $\gamma_{K,c} \not\subset B_r(\overline{X}_K; T_K)$,
		\item\label{prop_3} $c \in B_r(\overline{X}_K; T_K)$ and $\gamma_{K,c} \subset B_r(\overline{X}_K; T_K)$.
	\end{enumerate}
	Regarding the $K \in \cK$ such that $c \in K$, we claim that either they all satisfy Property \ref{prop_1}, exactly one satisfies Property \ref{prop_2} and the rest satisfy \ref{prop_3}, or all of them satisfy Property \ref{prop_3}. If $c \notin B_r(X; E)$, then $c \notin B_r(\overline{X}_K; T_K)$ for all $K \ni c$, i.e. they all satisfy Property \ref{prop_1}. If $c \in B_r(X; E)$, there exists $K_0 \ni c$ and $x \in \overline{X}_{K_0}$ such that $d_E(h(x), c) < r$. In particular, $K_0$ cannot satisfy Property \ref{prop_1}. For any $K \neq K_0$ with $c \in K$, we have $\ell_{K,c} = d_E(h(x_{K,c}), c) \leq d_E(h(x), c) < r$ because $x_{K,c}$ minimizes the distance to $c$ (recall Definition \ref{def:block_metric}). As a consequence, $\gamma_{K,c} \subset B_r(\overline{X}_K; T_K)$ for all $K \in \cK$ with $K \neq K_0$ and $c \in K$, so they all satisfy Property \ref{prop_3} while $K_0$ satisfies either Property \ref{prop_2} or \ref{prop_3}. This finishes the claim.\\
	\indent Suppose there are no pairs $c \in K_0$ (with $c \in \cut(\cK)$ and $K_0 \in \cK$) that satisfy Property \ref{prop_2}. For a fixed $c \in \cut(\cK)$, let $L_c$ be the space obtained by pasting the sets $B_r(\overline{X}_{K}; T_K) \cap \gamma_{K,c}$ along their respective endpoints $h(x_{K,c})$ for every $K \in \cK$ with $c \in K$. By the previous paragraph, $L_c$ completely contains none of the paths $\gamma_{K,c}$ with fixed $c$ or all of them. At the same time, \cite[Proposition 2.3]{osman-memoli} and Lemma \ref{lemma:VR_block_tight_span} yield $\vr_{2r}(\overline{X}_{K}) \simeq B_{r}(\overline{X}_{K}; T_{K})$. By abuse of notation, let's assume that, for each $c \in \cut(\cK)$, the relation $\sim$ in the Theorem's statement glues the points $h(x_{K,c}) \in B_r(\overline{X}_{K}; T_{K})$. Then it follows that
	\begin{align*}
		\left( \bigsqcup_{K \in \cK} \vr_{2r}(\overline{X}_{K}) \right) / \sim
		&\simeq \left( \bigsqcup_{K \in \cK} B_{r}(\overline{X}_{K}; T_{K}) \right) / \sim \\
		&= \left( \bigsqcup_{K \in \cK} \left( B_{r}(\overline{X}_{K}; E) \cap K \right) \cup \bigcup_{c \in K \cap \cut(\cK)} \gamma_{K,c} \right) / \sim \\
		&\cong \left( \bigsqcup_{K \in \cK} B_{r}(\overline{X}_{K}; E) \cap K \right) \cup \bigcup_{c \in \cut(\cK)} L_c \\
		&\simeq \left( \bigsqcup_{K \in \cK} B_{r}(\overline{X}_{K}; E) \cap K \right) / \sim_0 \sqcup \ C_r.
	\end{align*}
	The last homotopy equivalence is obtained by contracting the sets $L_c$ to their wedge point. If $L_c$ contains all paths $\gamma_{K,c}$, then it intersects every set $B_r(\overline{X}_{K}; E) \cap K$ with $c \in K$ at $c$, so the contraction ends up identifying all the copies of $c$ spread across the sets $B_{r}(\overline{X}_{K}; E) \cap K$. We denote these identifications with $\sim_0$. If, on the other hand, $L_c$ contains none of the paths $\gamma_{K,c}$, then it is disjoint from every $B_r(\overline{X}_K; E) \cap K$. Rather than inducing identifications, contracting $L_c$ results in an isolated point. This happens precisely when $c$ is an isolated vertex of $\bctree_r(\cK)$, so the resulting set of isolated points is in bijective correspondence with $C_r$. Moreover, two sets $B_r(\overline{X}_K; E) \cap K$ and $B_r(\overline{X}_{K'}; E) \cap K'$ have non-empty intersection if and only if $c \in K \cap K'$ and $c \in B_r(X; E)$. The latter happens if and only if the contraction of $L_c$ induces identifications, so
	\begin{equation*}
		\left( \bigsqcup_{K \in \cK} B_{r}(\overline{X}_{K}; E) \cap K \right) / \sim_0 \
		\cong \bigcup_{K \in \cK} B_{r}(\overline{X}_{K}; E) \cap K
		=  B_r(X; E).
	\end{equation*}
	Then \cite[Proposition 2.3]{osman-memoli} yields
	\begin{align*}
		\vr_{2r}(X) \sqcup C_r
		\simeq B_r(X; E) \sqcup C_r 
		&\cong \left( \bigsqcup_{K \in \cK} B_{r}(\overline{X}_{K}; E) \cap K \right) / \sim_0 \sqcup \ C_r \\
		&\simeq \left( \bigsqcup_{K \in \cK} \vr_{2r}(\overline{X}_{K}) \right) / \sim.
	\end{align*}
	
	\indent Lastly, suppose $c \in \cut(\cK)$ and $K \in \cK$ satisfy Property \ref{prop_2}. The fact that $\gamma_{K,c}$ is not completely contained in $B_r(\overline{X}_K; T_K)$ means there is a point $p \in \gamma_{K,c}$ such that $d_E(h(x), p), d_E(p, h(x_{K,c})) \geq r$ for all $x \in \overline{X}_K \setminus \{x_{K,c}\}$. Since $T_K$ is geodesic,
	\begin{equation*}
		d_X(x, x_{K,c}) = d_E(h(x), h(x_{K,c})) = d_E(h(x), p) + d_E(p, h(x_{K,c})) \geq 2r.
	\end{equation*}
	Let $X_2$ be the set of $x \in X$ such that $c$ separates $h(x)$ and $K$, and let $X_1 = X \setminus X_2$. Then for any $x \in X_1$ and $x' \in X_2$, the definition of $x_{K,c}$ (Definition \ref{def:block_metric}) yields $d_X(x, x') \geq d_X(x, x_{K,c}) \geq 2r$. As a consequence, $\vr_{2r}(X)$ contains no edge between $X_1$ and $X_2$. We conclude by applying the Theorem to $X_1$ and $X_2$ and using induction on the number of pairs $c \in K$ that satisfy Property \ref{prop_2}.
\end{proof}

The work of this section pays off below. We show that the block decomposition of $E$ induces a direct sum decomposition of the homology of $\vr_r(X)$ in terms of the homology of $\vr_r(\overline{X}_K)$ for every $K \in \cK$.
\begin{theorem}
	\label{thm:VR_wedge_homology}
	For any $k \geq 1$,
	\begin{equation*}
		H_k(\vr_{r}(X)) \cong \bigoplus_{K \in \cK} H_k(\vr_{r}( \overline{X}_{K})).
	\end{equation*}
	For $k=0$, the equation holds after a quotient by $\sim$ on the right side, where $[x_{K,c}] \sim [x_{K',c}]$ for a fixed $c \in \cut(\cK)$.
\end{theorem}
\begin{proof}
	Let $K_1, \dots, K_L$ be the elements of $\cK$, and let $V_\ell := \vr_{r}(\overline{X}_{K_\ell})$ for $\ell=1, \dots, L$. Let $\sim$ be the equivalence relation from Theorem \ref{thm:VR_wedge}, and define $\overline{V}_{1} := V_1$ and $\overline{V}_\ell := \overline{V}_{\ell-1} \cup_\sim V_\ell$. We claim that
	\begin{equation}
		\label{eq:VR_wedge_homology}
		H_k(\overline{V}_\ell) \cong \bigoplus_{i=1}^\ell H_k(V_i).
	\end{equation}
	This is immediate for $\ell=1$. For $\ell>1$, we use the Mayer-Vietoris sequence for the subsets $\overline{V}_{\ell-1}$ and $V_\ell$ of $\overline{V}_{\ell}$:
	\begin{equation*}
		\cdots \to H_k(\overline{V}_{\ell-1} \cap V_\ell) \to H_k(\overline{V}_{\ell-1}) \oplus H_k(V_\ell) \to H_k(\overline{V}_\ell) \to H_{k-1}(\overline{V}_{\ell-1} \cap V_\ell) \to \cdots.
	\end{equation*}
	Note that, when viewing $\overline{V}_{\ell-1}$ and $V_\ell$ as subsets of $\overline{V}_\ell$, the intersection $\overline{V}_{\ell-1} \cap V_\ell$ consists of one equivalence class $[x_c] := \{x_{K_i,c} : c \in K_i, 1 \leq i \leq \ell\}$ for each $c \in \cut(\cK)$. In other words, $\overline{V}_{\ell-1} \cap V_\ell$ is a set of discrete points, so the Mayer-Vietoris sequence yields an isomorphism $H_k(\overline{V}_{\ell}) \cong H_k(\overline{V}_{\ell-1}) \oplus H_k(V_{\ell})$ for $k \geq 2$. To get the same isomorphism when $k=1$, we need the boundary map $H_{1}(\overline{V}_\ell) \to H_{0}(\overline{V}_{\ell-1} \cap V_\ell)$ to be 0. This happens if and only if the map
	\begin{equation*}
		H_0(\overline{V}_{\ell-1} \cap V_\ell) \to H_0(\overline{V}_{\ell-1}) \oplus H_0(V_\ell)
	\end{equation*}
	is injective. Indeed, if $[x_c], [x_c'] \in \overline{V}_{\ell-1} \cap V_\ell$ lie in different connected components of $\overline{V}_{\ell-1} \cap V_\ell$, then they cannot come from the same connected component of $\overline{V}_{\ell-1}$ or the path that starts at $c$, goes to $c'$ through the $K$ in $\overline{V}_{\ell-1}$, and ends in $K_\ell$ would induce a cycle in $\bctree_\infty(\cK)$ (which is a tree by Remark \ref{rmk:block_cut_graph_is_tree}). By induction, (\ref{eq:VR_wedge_homology}) holds for $k \geq 1$.\\
	\indent Lastly, let $k=0$. A point $x \in X$ appears only once as a vertex of $\vr_r(X)$, but it may appear multiple times in $\bigsqcup_{K \in \cK} \vr_{r}(\overline{X}_{K})$, either as $x \in X_K \subset \overline{X}_K$ or as $x_{K,c} \in \overline{X}_K \setminus X_K$ for some $K$ and $c \in \cut(\cK)$. The equivalence relation in Theorem \ref{thm:VR_wedge} only identifies the latter type, so if we identify the classes $[x]$ and $[x_{K,c}]$ of $\bigoplus_{K \in \cK} H_0(\vr_{r}( \overline{X}_{K}))$ whenever $x = x_{K,c}$, we will obtain $H_0(\vr_{r}(X))$.
\end{proof}

We state two corollaries of the previous theorem. First, the original motivation for our work: a decomposition of $\vr_r(X)$ induced by the block decomposition of $T(X, d_X)$. Then, we state the special case in which $X$ is a subset of a metric wedge.
\begin{corollary}
	\label{cor:VR_wedge}
	Fix $r > 0$. Let $(X, d_X)$ be a finite metric space. Let $C_r$ be the set of isolated vertices of $\bctree_r\left[\blocks(T(X, d_X))\right]$. Then
	\begin{equation*}
		\vr_{2r}(X) \sqcup C_{r} \simeq \left( \bigsqcup_{B \in \blocks(T(X, d_X))} \vr_{2r}(\overline{X}_{B}) \right) / \sim
	\end{equation*}
	where we identify $x_{B,c} \in \vr_{2r}(\overline{X}_B)$ with $x_{B',c} \in \vr_{2r}(\overline{X}_{B'})$ if $B \cap B' = \{c\}$. Furthermore,
	\begin{equation*}
		H_k(\vr_{r}(X)) \cong \bigoplus_{B \in \blocks(T(X, d_X))} H_k(\vr_{r}( \overline{X}_{B}))
	\end{equation*}
	holds for $k \geq 1$ and for $k=0$ after identifying the classes $[x_{B,c}]$ for a fixed $c \in \cut(T(X, d_X))$.
\end{corollary}
\begin{proof}
	Set $E = T(X, d_X)$ and $\cK = \blocks(T(X, d_X))$ in Theorems \ref{thm:VR_wedge} and \ref{thm:VR_wedge_homology}.
\end{proof}

\begin{corollary}
	\label{cor:VR_wedge_subset}
	Let $X, Y$ be finite subsets of a metric space $Z = Z_1 \vee Z_2$ such that $Z_1 \cap Z_2 = \{z_0\}$, $X \subset Z_1$, $Y \subset Z_2$ and $X \cap Y = \emptyset$. Let $d_{Z}(X,Y) := \inf_{x \in X, y \in Y} d_Z(x,y)$. Choose any $x_0 \in X$ and $y_0 \in Y$ that satisfy $d_Z(x_0, y_0) = d_Z(X, Y)$, and let $\overline{X} = X \cup \{y_0\}$ and $\overline{Y} = Y \cup \{x_0\}$. Then
	\begin{equation*}
		\vr_r(X \cup Y) \sqcup C_{r/2} \simeq \vr_r(\overline{X}) \vee \vr_r(\overline{Y})
	\end{equation*}
	where $C_{r/2}$ is a singleton when $r \leq d_Z(X,Y)$ and empty otherwise. Additionally,
	\begin{equation*}
		H_k(\vr_r(X \cup Y)) \cong H_k(\vr_r(\overline{X})) \oplus H_k(\vr_r(\overline{Y})).
	\end{equation*}
	for all $k \geq 1$. If $k=0$, the equation holds as is when $r > d_{Z}(X,Y)$, and for all $r \leq d_Z(X,Y)$ after identifying the classes of $x_0$ in $H_0(\vr_r(X))$ and $y_0$ in $H_0(\vr_r(Y))$.
\end{corollary}
\begin{proof}
	To simplify notation, fix $r>0$ and let $d_{W} := d_Z|_{W \times W}$ for $W = X, Y, X \cup Y$. Let $c:X \cup Y \to \R_{\geq 0}$ be the function $c(p) = d_Z(p,z_0)$ and note that $c(x) + c(y) = d_Z(x,z_0) + d_Z(z_0,y) = d_{X \cup Y}(x, y)$ for any $x \in X$ and $y \in Y$. By Theorem 1 of \cite{cut-points-metric-spaces}, not only is $c$ an element of $T(X \cup Y, d_{X \cup Y})$, it is also a cut-vertex of $T(X \cup Y, d_{X \cup Y})$. For any cut-vertex $c_0$ of a polytopal complex, either $c_0$ is the intersection of two blocks or the maximal cell that contains $c_0$ is a bridge $L$ (i.e. a line segment). To simplify notation, assume there is a bridge $L$ whose length is possibly 0. Note that $X$ and $Y$ embed as vertices of $T(X \cup Y, d_{X \cup Y})$, so they are not in the interior of $L$.\\
	\indent  Since $X \cup Y$ is finite, $T(X \cup Y, d_{X \cup Y})$ is a polytopal complex. Let $\cK_X$ be the set of blocks of $T(X \cup Y, d_{X \cup Y})$ distinct from $L$ that intersect $X$; define $\cK_Y$ analogously. Let $K_X := \left( \bigcup_{B \in \cK_X} B \right) \cup L$ and $K_Y := \bigcup_{B \in \cK_X} B$ (we can also include $L$ in $K_Y$ instead of $K_X$ -- it makes no difference). Then $\cK = \{K_X, K_Y\}$ is a block cover of $T(X \cup Y, d_{X \cup Y})$. Since $d_{Z}(x_0, y_0) = d_Z(X, Y)$, we choose $x_{K_Y, c} = y_0$ and $y_{K_X, c} = x_0$ so that $\overline{X \cup Y}_{K_X} = X \cup \{y_0\}$ and $\overline{X \cup Y}_{K_Y} = Y \cup \{x_0\}$. Then by Theorem \ref{thm:VR_wedge},
	\begin{equation*}
		\vr_{r}(X \cup Y) \sqcup C_{r/2} \simeq \big[ \vr_{r}(X \cup \{y_0\}) \sqcup \vr_{r}(Y \cup \{x_0\}) \big] / (y_0 \sim x_0) \simeq \vr_r(\overline{X}) \vee \vr_r(\overline{Y}).
	\end{equation*}
	Note that $x_0$ and $y_0$ are discrete points in $\vr_r(Y \cup \{x_0\})$ and $\vr_r(X \cup \{y_0\})$, respectively, if and only if $r \leq d_Z(X,Y) = d_Z(x_0, y_0)$. Then $C_{r/2}$ equals the cut vertex in $K_X \cap K_Y$ when $r \leq d_Z(X,Y)$ and is empty otherwise. The equation on the homology follows from Theorem \ref{thm:VR_wedge_homology}.
\end{proof}

\subsection{Applications to totally decomposable spaces}
\label{sec:blocks_of_TDS}
Recall that the map $\kappa$ between the Buneman complex and the Tight span is not an isometry in general. However, it does preserve the distance from an element in the Buneman complex to every point in the embedding of $X$.
\begin{prop}
	\label{prop:homeo_TS_Buneman_balls}
	Let $(\cS, \alpha)$ be a weighted weakly compatible split system on a finite set $X$. For any $\phi \in B(\cS,\alpha)$ and $x \in X$,
	\begin{equation*}
		d_1(\phi, \phi_x) = d_\infty(\kappa(\phi), h_x).
	\end{equation*}
	As a consequence, the restriction of $\kappa:B(\cS,\alpha) \to T(d_{\cS,\alpha})$ maps $B_r(X; B(\cS,\alpha))$ onto $B_r(X, T(d_{\cS,\alpha}))$.
\end{prop}
\begin{proof}
	Let $\phi \in B(\cS,\alpha)$. Observe that
	\begin{equation*}
		d_\infty(\kappa(\phi), h_x)
		= \sup_{y \in X} |\kappa(\phi)(y) - h_x(y)|
		= \sup_{y \in X} |d_1(\phi, \phi_y) - d_{\cS, \alpha}(x,y)|.
	\end{equation*}
	Since $d_{\cS, \alpha}(x,y) = d_1(\phi_x, \phi_y)$, the reverse triangle inequality yields $|d_1(\phi, \phi_y) - d_1(x,y)| \leq d_1(\phi, \phi_x)$. This upper bound is actually realized when we set $y=x$, so $d_\infty(\kappa(\phi), h_x) = d_1(\phi, \phi_x)$. Then $\kappa$ maps $B_r(X; B(\cS,\alpha))$ onto $B_r(X; T(d_{\cS,\alpha}))$ because
	\begin{align*}
		\phi \in B_r(X; B(\cS,\alpha))
		&\Leftrightarrow \exists x \in X \text{ such that } d_1(\phi_x, \phi) < r \\
		&\Leftrightarrow \exists x \in X \text{ such that } d_\infty(h_x, \kappa(\phi)) < r \\
		&\Leftrightarrow \kappa(\phi) \in B_r(X; T(d_{\cS,\alpha})).
	\end{align*}
\end{proof}

Now we go back to circular decomposable spaces to prove one more property. Recall from Remark \ref{rmk:unimodal_sum} that $\sigma(c)$ is not defined when $\osum_{i=c+2}^c \alpha_{ic} < \alpha_{c+1,c}$. If this is the case, we use our understanding of their tight spans to show that $\vr_r(X)$ doesn't change much if we remove $c$.
\begin{prop}
	\label{prop:point_with_long_edge}
	Let $d_X$ be a circular decomposable metric on $X = \{1, \dots, n\}$. Suppose there exists a point $c \in X$ such that $\alpha_{c+1,c} > \osum_{i=c+2}^c \alpha_{ic}$ (i.e. a point for which $\sigma(c)$ in Definition \ref{def:unimodal_sum} is not defined). Let $R_c := \min_{x \in X'} d_X(x, c)$. Then
	\begin{equation*}
		\vr_r(X) \simeq
		\begin{cases}
			\vr_r(X \setminus \{c\}) & r > R_c\\
			\vr_r(X \setminus \{c\}) \sqcup \{c\} & r \leq R_c.
		\end{cases}
	\end{equation*}
\end{prop}
\begin{proof}
	Recall that for any $x \in X$, $\phi_x \in B(\cS, \alpha)$ is defined by $\phi_x(A) = \frac{1}{2} \alpha_{A|\overline{A}}$ if $x \notin A$ and $0$ if $x \in A$. Given $0 \leq t \leq \frac{1}{2}\alpha_{c,c+1}$, define $\phi_{c,t} \in B(\cS, \alpha)$ by
	\begin{equation*}
		\phi_{c,t}(A) =
		\begin{cases}
			\phi_c(A) & A \neq A_{c,c+1}, \overline{A}_{c,c+1} \\
			t & A = A_{c,c+1}, \\
			\frac{1}{2} \alpha_{c,c+1} - t & A = \overline{A}_{c,c+1}.
		\end{cases}
	\end{equation*}
	For any $0 < t < \frac{1}{2}\alpha_{c,c+1}$, $\cS(\phi_{c,t}) = \{A_{c,c+1}|\overline{A}_{c,c+1}\}$. Note that $A_{c,c+1}|\overline{A}_{c,c+1}$ is a maximal incompatible subset of $\cS$ because $A_{c,c+1} = \{c\}$ and any $S \in \cS$ has an element $B \in S$ such that $c \notin B$. Let $E$ be the minimal cell of $B(\cS, \alpha)$ that contains $\phi_{c,t}$ or, in the notation of \cite{TDM-polytopal-structure-1}, $E = [\phi_{c,t}] \subset B(\cS, \alpha)$. Since $A_{c,c+1}|\overline{A}_{c,c+1}$ is maximal incompatible, $E$ is a maximal cell of $B(\cS, \alpha)$ by (B3) of \cite{TDM-polytopal-structure-1} and by Property (B5) of \cite{TDM-polytopal-structure-1}, $\dim(E) = |\cS(\phi_{c,t})| = 1$. Hence, the maximal cell that contains $\phi_c = \phi_{c,0}$ is the edge $E$.\\
	\indent Let $\ell := \alpha_{c,c+1}/2$ and $\gamma_c := \phi_{c, \ell}$. For any $x \neq c$, $x \in \overline{A}_{c,c+1}$ and $c \in A_{c,c+1}$ imply $\phi_x(A) = \alpha_{c,c+1}/2 - \phi_c(A)$ for $A = A_{c,c+1}, \overline{A}_{c,c+1}$. By definition of $\gamma_c$, we also have $\gamma_c(A) = \alpha_{c,c+1}/2 - \phi_c(A)$ for $A = A_{c,c+1}, \overline{A}_{c,c+1}$. Then
	\begin{align*}
		d_1(\phi_c, \gamma_c)
		&= \sum_{A \in U(\cS)} |\phi_c(A) - \gamma_c(A)| \\
		&= |\phi_c(A_{c,c+1}) - \gamma_c(A_{c,c+1})| + |\phi_c(\overline{A}_{c,c+1}) - \gamma_c(\overline{A}_{c,c+1})| \\
		&= |0 - \alpha_{c,c+1}/2| + |\alpha_{c,c+1}/2 - 0| = \alpha_{c,c+1},
	\end{align*}
	and for any $x \neq c$,
	\begin{align*}
		d_1(\phi_x, \gamma_c)
		&= \sum_{A \in U(\cS)} |\phi_c(A) - \gamma_c(A)| \\
		&= \sum_{A \neq A_{c,c+1}, \overline{A}_{c,c+1}} |\phi_x(A) - \phi_c(A)| + \sum_{A = A_{c,c+1}, \overline{A}_{c,c+1}} |\phi_x(A) - \gamma_c(A)| \\
		&= \sum_{A \in U(\cS)} |\phi_x(A) - \phi_c(A)| - \sum_{A = A_{c,c+1}, \overline{A}_{c,c+1}} |\phi_x(A) - \phi_c(A)| \\
		&= d_1(\phi_x, \phi_c) - \alpha_{c,c+1}.
	\end{align*}
	In particular, $d_1(\phi_{c}, \phi_{c-1}) = d_{c,c-1} = \osum_{i=c+1}^{c-1} \alpha_{ic} = \osum_{i=c+1}^{c} \alpha_{ic}$ by Lemma \ref{lemma:circular-distance} (recall $\alpha_{cc} = 0$). Then $d_1(\phi_{c-1}, \gamma_c) = \osum_{i=c+2}^{c} \alpha_{ic}$, so by the Proposition's hypothesis, $d_1(\phi_{c-1}, \gamma_c) < d_1(\phi_c, \gamma_c)$.\\
	\indent By Proposition \ref{prop:homeo_TS_Buneman_balls}, $d_\infty(h_{c-1}, \kappa(\gamma_c)) = \osum_{i=c+2}^{c} \alpha_{ic} < \alpha_{c, c+1} = d_\infty(h_c, \kappa(\gamma_c))$. By Theorem \ref{thm:kappa-block-bijection}, $\kappa(E)$ is an edge in $T(X, d_X)$ that connects $h_c$ and $\kappa(\gamma_c)$ and has length $d_\infty(h_c, \kappa(\gamma_c)) = \alpha_{c,c+1} = 2\ell$. Then $B_{r}(h_{c}; T(X, d_X)) \cap \kappa(E)^c = B_{r - 2\ell}(\kappa(\gamma_c); T(X, d_X)) \cap \kappa(E)^c$ for any $r > 2\ell$. However, since $r > d_\infty(h_c, \kappa(\gamma_c)) > d_\infty(h_{c-1}, \kappa(\gamma_c))$, we also have $\gamma_c \in B_{r}(h_{c-1}; T(X, d_X))$. Hence, for any $f \in B_{r}(h_{c}; T(X, d_X)) \cap \kappa(E)^c$,
	\begin{align*}
		d_\infty(h_{c-1}, f)
		&\leq d_\infty(h_{c-1}, \kappa(\gamma_c)) + d_\infty(\kappa(\gamma_c), f) \\
		&\leq d_\infty(h_{c-1}, \kappa(\gamma_c)) + (r - 2\ell)\\
		&= r +d_\infty(h_{c-1}, \kappa(\gamma_c)) - d_\infty(\kappa(\gamma_c), h_c)
		< r.
	\end{align*}
	In other words, $B_{r}(h_{c}; T(X, d_X)) \cap \kappa(E)^c \subset B_{r}(h_{c-1}; T(X, d_X))$ and, thus, $B_{r}(X; T(X, d_X))$ equals $B_{r}(X \setminus \{c\}; T(X, d_X))$ union with $E_r := B_{r}(h_{c}; T(X, d_X)) \cap \kappa(E)$, a subinterval of $\kappa(E)$. If $E_r$ intersects $B_{r}(X \setminus \{c\}; T(X, d_X))$, we can contract the former onto the latter, so that $B_{r}(X; T(X, d_X)) \simeq B_{r}(X \setminus \{c\}; T(X, d_X))$. Otherwise, $E_r$ is its own connected component, so we can contract it onto $h_c$ and write $B_{r}(X; T(X, d_X)) \simeq B_{r}(X \setminus \{c\}; T(X, d_X)) \sqcup \{h_c\}$. Then, by Theorem \ref{thm:VR-tight-span}, $B_{r}(X; T(X, d_X)) \simeq \vr_{2r}(X)$ and $B_{r}(X \setminus \{c\}; T(X, d_X)) \simeq \vr_{2r}(X \setminus \{c\})$. The conclusion follows by noticing that $E_r$ intersects $B_{r}(X \setminus \{c\}; T(X, d_X))$ when $r > \min_{x \neq c} d_{xc}/2 = R_c/2$.
	
\end{proof}

\section{Algorithmic considerations}
\label{sec:discussion}
We believe the results of this paper have the potential to speed up the computation of persistent homology thanks to polynomial time algorithms in the literature that deal with split decompositions. To fix notation, suppose we are computing persistent homology in dimension $k$ of a metric space with $n$ points. In the worst case, the original algorithm of Bandelt and Dress to compute the split decomposition of a finite metric space runs in $O(n^6)$ time (see the note between Corolaries 4 and 5 of \cite{metric-decomposition}), but faster, more specialized algorithms exist.

\paragraph{Circular decomposable spaces.}
Circular decomposable spaces are very efficient to work with. Identifying a circular decomposable metric and determining the cyclic ordering can both be done in $O(n^2)$ time; see \cite{CDM-structure, CDM-recognition} for the algorithms. In fact, \cite{CDM-structure, CDM-note} show the equivalence between circular decomposable and Kalmanson metrics, a class of metrics where the Traveling Salesman Problem has a very efficient solution. See also the note after \cite[Theorem 5]{metric-decomposition}.\\
\indent To compute persistent homology using the results of Section \ref{sec:circular-metrics}, we need to compute the function $\sigma$ from Definition \ref{def:unimodal_sum} and verify that condition (\ref{eq:star}) holds. Together, these operations take at most $O(n^2)$ time. In fact, for each $a \in X$, finding $\sigma(a)$ requires adding $n-1$ isolation indices and at most $n$ subtractions and $n$ comparisons, which add up to a total $O(n)$ cost. Hence, computing $\sigma:X \to X$ takes $O(n^2)$ time, and verifying condition (\ref{eq:star}) requires checking $n$ cyclic inequalities. From here, we can use the results of \cite{ph-circle} to compute the $k$-dimensional persistent homology in $O(n^2(k + \log(n)))$ time. All in all, recognizing a circular decomposable metric that satisfies the assumptions of Section \ref{sec:circular-metrics} and computing its persistent homology is a nearly quadratic operation of cost $O(n^2(k + \log(n)))$.

\paragraph{Block decompositions.}
The results of Section \ref{sec:blocks_of_X} use a block cover of an ambient injective space $E$ to break up the computation of the persistent homology of $X$ into smaller, potentially parallelizable, operations. For example, Corollary \ref{cor:VR_wedge} requires knowledge of the block decomposition of $T(X, d_X)$, which can be computed in $O(n^3)$ time thanks to \cite{cut-points-algorithm}. The Java implementation of this algorithm is called BloDec and is available at \cite{cut-points-software}. BloDec also computes the cut vertices of $T(X, d_X)$ as functions $f:X \to \R$ and the sets $X'_B$ for every $B \in \blocks(T(X, d_X))$. This information is enough to build the tree $\bctree(T(X, d_X))$ as well. To use Corollary \ref{cor:VR_wedge}, we need to form $\overline{X}_B$ by replacing every cut vertex $c \in X_B'$ with the point $x_{B,c} \in X$. This means finding the point in $X$ that is separated from $X_B'$ by $c$ that minimizes the distance to $c$, a problem that involves no more than $O(n)$ comparisons. Since there are $4n-5 = O(n)$ cut-vertices \cite[Lemma 3.2]{cut-points-algorithm}, forming the sets $\overline{X}_B$ takes $O(n^2)$ operations. Hence, finding the block decomposition of $T(X, d_X)$ is likely to save time whenever computing persistent homology takes more than $O(n^3)$ time and $T(X, d_X)$ has multiple blocks.\\
\indent We ran computational experiments to understand the performance of Corollary \ref{cor:VR_wedge}. Choose a number of blocks $b$ and a block size $m$. We sampled sets $Y_{1}, \dots, Y_{b}$ from $\Sp^3$ uniformly at random with sizes $|Y_1| = m$ and $|Y_i| = m+2$ for $2 \leq i \leq b$. We equipped each $Y_i$ with the Euclidean metric multiplied by $i^2$. We then constructed a metric wedge $Y_1 \vee \cdots \vee Y_{b}$ by identifying the last point of $Y_{i-1}$ with the first point of $Y_{i}$ for $2 \leq i \leq b$. Note that the resulting space has $m + (b-1)(m+2) - (b-1) = bm + (b - 1)$ points, out of which $b-1$ are wedge points. Then the tight span of $Y_1 \vee \cdots \vee Y_b$ is isometric to $T(Y_1) \vee \cdots \vee T(Y_b)$, so it has at least $b$ blocks and $b-1$ cut-vertices. Lastly, we discarded the wedge points from $Y_1 \vee \cdots \vee Y_b$ to obtain a metric space $X_b$ with $b`$ points that is not a metric wedge, but whose tight span has a non-trivial block structure.\\
\indent In our experiments, we used $m = 20$ and varied $b$ from 1 to 15 (so the size of $X_b$ varied from 20 to 300 points). We computed persistent homology in dimensions 1, 2 and 3 using a Python implementation of Ripser \cite{ripser-python}. We also implemented Corollary \ref{cor:VR_wedge} via a Python interface with BloDec that saves the distance matrix of $X_b$ to disk, calls BloDec from the terminal, and reads the output back into Python. We then constructed the sets $\overline{X}_B$ and computed their persistence diagrams with the same Python implementation of Ripser. Figure \ref{fig:blodec_vs_ripser} compares the running time of Ripser and Corollary \ref{cor:VR_wedge}.\\
\indent In theory, the standard algorithm for persistent homology has a worst-case running time\footnote{This assumes that matrix multiplication of $n$-by-$n$ matrices runs in $O(n^3)$.} of $O(n^{3(k+2)})$ \cite{matrixtime}, although in many practical scenarios the performance is closer to $O(n^{k+2})$ \cite{ph-complexity-revised, av-complexity-clique-filtrations}. See \cite[Section 3.1.2]{curvature-sets-pds-dcg} for a more in-depth comparison. Assuming the stricter bound of $O(n^{k+2})$, we expected Corollary \ref{cor:VR_wedge} to have a better running time than Ripser starting from $k=2$, but Figure \ref{fig:blodec_vs_ripser} shows that is not the case. Still, if the original space has multiple blocks, Corollary \ref{cor:VR_wedge} improves the running time of persistent homology starting in dimension 3.\\
\indent It's worth noting that Corollary \ref{cor:VR_wedge} has similar running times in all dimensions, so most of the time is spent on the block decomposition rather than on the persistence diagram. Hence, we can achieve even more significant speedup if we have a priori knowledge of the block decomposition of $T(X, d_X)$ (or of an ambient injective space $E$). For example, if we have a pair of spaces $X$ and $Y$ that satisfy the hypotheses of Corollary \ref{cor:VR_wedge_subset}, we don't need to use BloDec to find the block decomposition of $T(X \cup Y)$. Instead, we only need to find a pair of points $x_0 \in X$ and $y_0 \in Y$ that achieve the minimum $\min_{x \in X, y \in Y} d_Z(x, y)$. Then, no matter how many cut-vertices $T(X \cup Y)$ may have, the persistent homology of $\overline{X}$ and $\overline{Y}$ can already be computed faster than that of $X \cup Y$.

\begin{figure}[h]
	\centering
	\includegraphics[width=\linewidth]{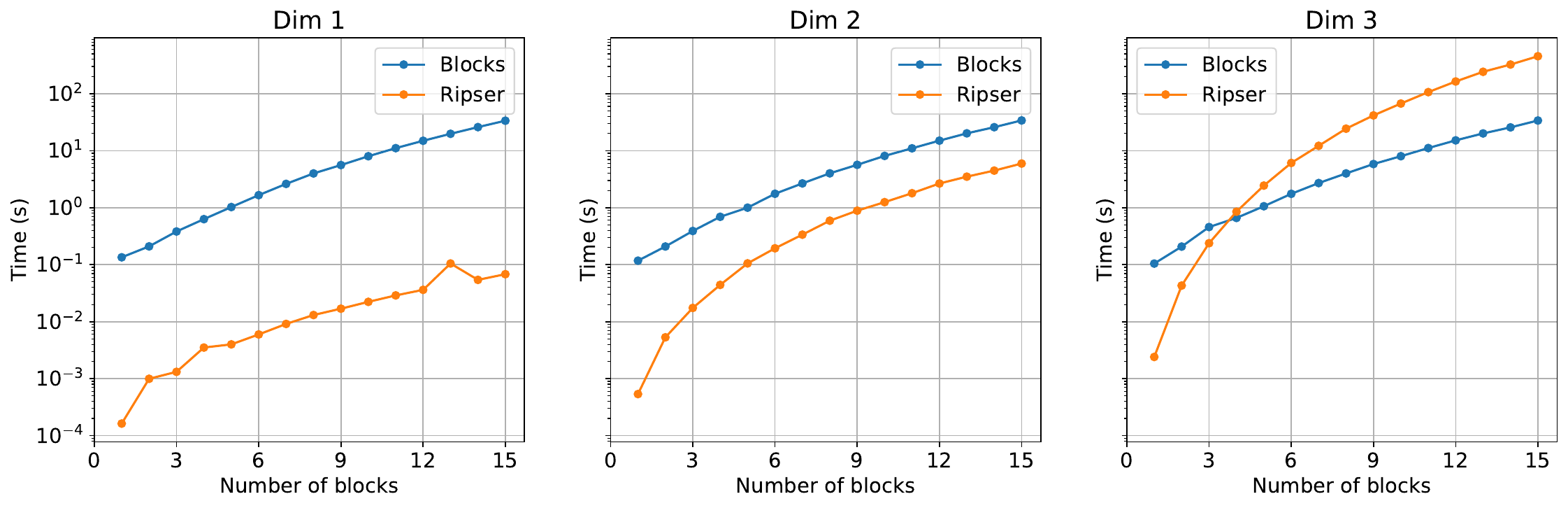}
	\caption{Runtime comparison of persistent homology computed using block decomposition (Corollary \ref{cor:VR_wedge} and BloDec) against Ripser. The input metrics have a varying number of blocks ($x$-axis), with each block being a uniformly sampled set of 20 points from 3-spheres of varying radii.}
	\label{fig:blodec_vs_ripser}
\end{figure}

\subsection{Future work}
While we have an analysis of the computational cost of the Theorems of Sections \ref{sec:circular-metrics} and \ref{sec:blocks_of_X}, we can't yet do the same for the recursive algorithms of Section \ref{sec:circular-metrics-non-monotone}. Recall from Example \ref{ex:non_recursive_metrics} that there exists a circular decomposable space $X$ where the non-cyclic subcomplex $V_Y \subset V_X = \vr_r(X)$ has vertex set $Y = X$ (see Definition \ref{def:cyclic_component}). In that case, the Mayer-Vietoris argument would give the trivial statement $H_*(V_X) = H_*(V_Y)$ and $H_*(V_X^c) = H_*(V_Y')$. We leave it to future research to find conditions on the isolation indices of a circular decomposable metric so that $Y \subsetneq X$ and, more generally, conditions so that the recursion described before Example \ref{ex:recursive_space} eventually reaches a monotone circular decomposable space. With that result, we could count the number of recursion steps involved and get a good estimate on the total computational cost of the persistent homology of a general circular decomposable space. 	

\newcommand{\etalchar}[1]{$^{#1}$}

\end{document}